\documentclass[12pt,leqno]{amsart}
\usepackage[utf8]{inputenc} 
\usepackage{amssymb}
\usepackage{enumerate}
\usepackage[all]{xy}
\usepackage{tikz-cd}
\usepackage{hyperref,cleveref,graphics,mathrsfs}
\usepackage{lmodern}
\usepackage{amsmath}
\usepackage{mathtools}
\usepackage{commath}
\usepackage{amsfonts}
\usepackage{dsfont}
\usepackage{graphicx}
\usepackage{amsthm}
\usepackage{mathabx}
\usepackage{comment}
\usepackage{vmargin} 
    \setmargins{2.5cm}{1.5cm}{16.5cm}{23.42cm}{10pt}{1cm}{0pt}{2cm}

\numberwithin{equation}{section}

\newtheorem{thm}{Theorem}[section]
\newtheorem{cor}[thm]{Corollary}
\newtheorem{lem}[thm]{Lemma}
\newtheorem{prop}[thm]{Proposition}
\theoremstyle{definition}
\newtheorem{defn}[thm]{Definition}

\newtheorem{rem}[thm]{Remark}

\newtheorem{example}[thm]{Example}
\numberwithin{equation}{section}

\newcommand{\eps}{\varepsilon}
\def\epsilon{\varepsilon}
\newcommand{\N}{\mathbb{N}}
\newcommand{\Z}{\mathbb{Z}}

\newcommand{\R}{\mathbb{R}}

\newcommand{\uno}{\mathds{1}}
\newcommand{\Ccal}{\mathcal{C}}
\newcommand{\Hcal}{\mathcal{H}}

\DeclarePairedDelimiterX{\normiii}[1]
{\vvvert}
{\vvvert}
{\ifblank{#1}{\:\cdot\:}{#1}}

\usepackage{bbm}
\def\one{\mathbbm 1}

\DeclareMathOperator{\spn}{span}

\DeclareMathOperator{\sgn}{sgn}

\title{Convexity and concavity in Banach lattices}
\author[E.~Garc\'ia-S\'anchez]{Enrique Garc\'ia-S\'anchez}
\address{Instituto de Ciencias Matem\'aticas (CSIC-UAM-UC3M-UCM)\\
Consejo Superior de Investigaciones Cient\'ificas\\
C/ Nicol\'as Cabrera, 13--15, Campus de Cantoblanco UAM\\
28049 Madrid, Spain.}
\email{enrique.garcia@icmat.es}

\date{\today}
\subjclass[2020]{46B42 (primary); 46B20, 46E30, 47B60, 47A68 (secondary)} 

\keywords{$p$-convexity; $p$-concavity; $L_p$; weak-$L_p$; factorization theory.}

\begin{document}
\begin{abstract}
These notes are a detailed, self-contained introductory course on convexity and concavity in Banach lattices, suitable for both experts and beginners. We revisit, from a modern perspective, the classical notions of $(p,q)$-convexity, $(p,q)$-concavity and upper and lower $p$-estimates, and the main relations between these properties, integrating more recent developments in the area. We explain in full detail the $p$-convexification and $p$-concavification techniques. We also provide a comprehensive exposition of the main factorization results for $(p,q)$-convex and $(p,q)$-concave operators, including well-known results from Krivine, Maurey--Nikishin, Pietsch and Pisier, and their applications to the renorming and the representation of convex and concave Banach lattices.
\end{abstract}

\maketitle
\allowdisplaybreaks
\tableofcontents

\section{Introduction}\label{sec: intro}

There is a vast literature on the study of convexity properties of the norm of a Banach space. For instance, Milman--Pettis Theorem \cite{Milman, Pettis} establishes that uniformly convex spaces are necessarily reflexive, while Enflo--James Theorem \cite{Enflo} (see also \cite[Theorem 9.14]{FHHMZBanachSpaceTheory}) asserts that a Banach space is super-reflexive if and only if it admits an equivalent uniformly convex renorming. The related properties of (Rademacher) type and cotype have also provided remarkable results, such as Kwapień's Theorem \cite{Kwapien}, that characterizes the spaces isomorphic to a Hilbert space as those that simultaneously have type and cotype equal to 2. \\

In Banach lattices, additional tools that exploit the lattice structure are available, such as $p$-convexity and $p$-concavity. The motivation for these notions is to measure how close a Banach lattice is to an $L_p$-space, in terms of the interaction between its norm and its lattice structure. Indeed, an elementary but useful fact about the $L_p$-norm is that whenever $f_1,\ldots, f_n\in L_p(\mu)$, we have
\[\norm[3]{\intoo[3]{\sum_{i=1}^n |f_i|^p}^\frac{1}{p}}_{L_p}= \intoo[3]{\sum_{i=1}^n \|f_i\|_{L_p}^p}^\frac{1}{p}.\]
This equality asserts that the $p$-sum of functions in $L_p$ using the lattice structure accumulates the same norm as the $p$-sum of the norms. In general, we can study how close a Banach lattice $X$ is to satisfying one of the two inequalities that appear in the equality above. That is, given arbitrary $x_1,\ldots, x_n\in X$, we wish to compare 
\[\norm[3]{\intoo[3]{\sum_{i=1}^n |x_i|^p}^\frac{1}{p}}\quad \text{and} \quad  \intoo[3]{\sum_{i=1}^n \|x_i\|^p}^\frac{1}{p}\]
(a precise meaning for the expression $(\sum_{i=1}^n |x_i|^p)^\frac{1}{p}$ will be given in the next section; meanwhile, it can be understood as the generalization of the pointwise $p$-sum of functions).
If
\[\norm[3]{\intoo[3]{\sum_{i=1}^n |x_i|^p}^\frac{1}{p}}\leq K  \intoo[3]{\sum_{i=1}^n \|x_i\|^p}^\frac{1}{p},\]
for some uniform constant $K>0$, this means that the norm of $X$ is at least as convex as the $L_p$-norm. In this case, we say that $X$ is $p$-convex. Conversely, if there is a constant $K>0$ such that
\[ \intoo[3]{\sum_{i=1}^n \|x_i\|^p}^\frac{1}{p}\leq K\norm[3]{\intoo[3]{\sum_{i=1}^n |x_i|^p}^\frac{1}{p}},\]
for every $x_1,\ldots, x_n\in X$, we can conclude that the $p$-sum of the elements captures at least as much norm as the $p$-sum of the norms. We then say that $X$ is $p$-concave. As one could expect, these notions are dual to each other: a Banach lattice is $p$-convex (respectively, $p$-concave) if and only if its dual is $p^*$-concave (respectively, $p^*$-convex), where $p^*$ denotes the conjugate exponent of $p$. \\

There are several related notions that extend the above definitions: On the one hand, when the inequalities above are witnessed only on pairwise disjoint sequences of vectors, we say that the Banach lattice satisfies an upper (respectively, lower) $p$-estimate. On the other hand, it is natural to replace the $p$-sums of elements in $X$ by $q$-sums with a different exponent $q$, giving rise to the concepts of $(p,q)$-convexity and $(p,q)$-concavity. It turns out that both generalizations are actually equivalent (up to constants). Namely, a Banach lattice satisfies an upper $p$-estimate if and only if it is $(p,q)$-convex for some $p<q\leq \infty$, and similarly, it satisfies a lower $p$-estimate if and only if it is $(p,q)$-concave for some $1\leq q<p$. Similar definitions can also be considered for operators. In fact, $(p,q)$-convex (respectively, $(p,q)$-concave) operators will be precisely the operators that factor through $(p,q)$-convex (respectively, $(p,q)$-concave) Banach lattices. \\

As we will see in this memoir, the interest of $(p,q)$-convex and $(p,q)$-concave operators and spaces is diverse. They were first introduced by Krivine \cite{KrivineExp2223} and Maurey \cite{MaureyExp2425} (type $\geq (q,p)$ and $\leq (q,p)$ operators), with an aim towards studying the properties of factorization through $L_p$ of certain operators, the connections with other relevant classes of operators, such as $(p,q)$-summing operators, and the similarities with other geometric properties such as (Rademacher) type and cotype. The convexity and concavity properties of Banach lattices can also capture their $L_p$ behavior. For instance, we will deduce from Kakutani's Representation Theorem for $AL_p$-spaces (cf. \cite[Theorem 2.7.1]{MN}) and the renorming results from \Cref{sec: abstract factorizations} that a Banach lattice is lattice isomorphic to an $L_p$-space if and only if it satisfies an upper and a lower $p$-estimate (see \Cref{cor: isomorphic ALp-spaces}). In other words, if a Banach lattice has the same convexity and concavity index, then it has to be $L_p$, up to some renorming.

\subsection*{Content of the paper}
This memoir originates as an introductory course that the author was invited to teach at the Department of Mathematical and Statistical Sciences of the University of Alberta during a research stay in fall 2024. Therefore, its purpose is to gather the main results of the topic and present them in a rigorous and detailed way, rather than reviewing the most recent developments in the area. To this day, the standard references for specialists have been the textbooks of J. Diestel, H. Jarchow and A. Tonge \cite{DJT}, and J. Lindenstrauss and L. Tzafriri \cite{LT2}, which devote several sections to exploring the operator theoretical and the geometrical aspects of the subject, respectively. However, their approaches leave some major topics uncovered, which were dispersed in several textbooks \cite{AK, MN, Schwarz} and research papers \cite{KrivineExp2223, MaureyExp2425, MaureyFactorization, Pisier86, Reisner, RT}. \\

Therefore, this work is conceived as an attempt to gather a (non-exhaustive) selection of topics that the author considers relevant and fundamental for the theory of $p$-convex and $p$-concave Banach lattices. It aims to present them in a format that serves beginners, with detailed proofs, examples and a self-contained exposition that only requires the reader to be familiar with the basic notions of Banach lattice theory, but also specialists, proposing a modern point of view of several classical results and introducing recent developments.\\

The paper is structured as follows: First, in \Cref{sec: preliminaries} we recall basic notions from Banach lattice theory and establish several properties of $p$-sums in Banach lattices, as these expressions play a central role in the theory. Then we proceed to introducing the main definitions in \Cref{sec: main definitions}. We illustrate these concepts with examples in the family of Lorentz spaces, putting a special focus on weak $L_p$-spaces, which serve as a model space for upper $p$-estimates. \\

Next, we establish the core results of the theory: in \Cref{sec: duality} we show that convexity and concavity satisfy a duality relation, while in \Cref{sec: basic properties} we establish the basic implications between the properties previously defined. Namely, we show that $p$-convexity implies $q$-convexity for any $q<p$ and that upper $p$-estimates are equivalent to $(p,\infty)$-convexity (and analogously in the concave setting). We also characterize $(p,q)$-concave operators in terms of $(p,q)$-summing operators on $C(K)$-spaces, which will allow us to establish in \Cref{sec: concave on C(K)} the equivalence between $(p,\infty)$-convexity and $(p,q)$-convexity for any other $q>p$. \\

\Cref{sec: convexification and concavification} is devoted to introducing in full detail the techniques of $p$-convexification and $p$-concavification of a Banach lattice and showing some simple applications. We continue with \Cref{sec: abstract factorizations}, which clarifies the connection between convex and concave operators and spaces. We show that every $(p,q)$-convex (respectively, $(p,q)$-concave) operator factors canonically through a $(p,q)$-convex (respectively, $(p,q)$-concave) Banach lattice, obtaining as a consequence a renorming result that allow us to reduce the convexity and concavity constants of any Banach lattice to one. Next, in \Cref{sec: concrete factorizations}, we prove several classical factorization results due to Krivine, Maurey--Nikishin, Pisier and Pietsch, that allow us to factor certain operators with convexity or concavity properties through $L_p$, $L_{p,\infty}$ and $L_{p,1}$. Finally, in \Cref{sec: representations} we apply the factorization theorems from the previous sections to deduce representation results for convex and concave Banach lattices.

\subsection*{Notation}

\begin{itemize}
    \item[$\bullet$] $E,F,G$ will generally denote Banach spaces, whereas $X,Y,Z$ will be used for Banach lattices.
    \item[$\bullet$] $B_E$ and $S_E$ denote the unit ball and the unit sphere of a Banach space $E$, respectively.
    \item[$\bullet$] $X_+$ denotes the positive cone of a Banach lattice $X$, that is, the set of positive elements.
    \item[$\bullet$] $\spn(A)$ denotes the subspace generated by the set $A$.
    \item[$\bullet$] An operator is a linear and bounded map between Banach spaces.
    \item[$\bullet$] An embedding is an invertible operator with continuous inverse, that is, an isomorphism onto its image.
    \item[$\bullet$] $i_E:E\rightarrow E$ denotes the identity operator on a Banach space $E$.
    \item[$\bullet$] $E^*$ denotes the dual of a Banach space $E$, and $E^{**}$ the bidual. $J_E:E\rightarrow E^{**}$ is the canonical isometric embedding from $E$ into its bidual.
    \item[$\bullet$] $C(K)$ denotes the space of continuous functions over a compact Hausdorff space $K$, and $\uno$ or $\uno_K$ will denote the constant one function on $K$.
    \item[$\bullet$] If $1\leq p\leq \infty$, $p^*$ denotes its conjugate exponent, defined by $\frac{1}{p}+\frac{1}{p^*}=1$.
    \item[$\bullet$] The norm in $L_p(\mu)$ will be indistinctly denoted by $\|\cdot\|_{L_p(\mu)}$, $\|\cdot\|_{L_p}$ or $\|\cdot \|_p$.
\end{itemize}

\section{Preliminaries}\label{sec: preliminaries}
\subsection{Banach lattices}
A \textit{(real) vector lattice} is a real vector space $X$ endowed with a partial order $\leq$ that satisfies the following properties: 
\begin{enumerate}
    \item \textbf{compatibility with the linear operations:} whenever $x,y\in X$ with $x\leq y$, we have $x+z\leq y+z$ for every $z\in X$ and $\lambda x\leq \lambda y$ for every scalar $\lambda>0$.
    \item \textbf{existence of lattice operations:} for every $x,y\in X$ there exists the least upper bound and the greatest lower bound in $X$ of the set $\{x,y\}$, called the \textit{supremum} and \textit{infimum} and denoted $x\vee y$ and $x\wedge y$, respectively.
\end{enumerate}
We denote by $X_+:=\{x\in X:x\geq 0\}$ the positive cone of $X$. A \textit{(real) Banach lattice} is a real vector lattice $(X,\leq)$ endowed with a complete norm $\|\cdot \|$ satisfying the additional property of compatibility with the order:
\begin{enumerate}
    \item[(3)] \textbf{monotonicity of the norm:} whenever $x,y\in X$ satisfy $|x|\leq |y|$, where the absolute value of an element $x$ is defined as $|x|=x\vee (-x)$, we have that $\|x\|\leq \|y\|$.
\end{enumerate}
The main focus of these notes lies in Banach lattices, so most of the definitions will be stated in this setting, even though they might hold in a more general setting, such as vector lattices. In a Banach lattice, there is a large diversity of subspaces besides linear subspaces. We say that a subspace $Y$ of a Banach lattice $X$ is: a \textit{sublattice} if it is closed under lattice operations; an \textit{ideal} if it is \textit{solid}, that is, if whenever $x\in Y$ and $|y|\leq |x|$ implies that $y\in Y$; a \textit{band} if whenever $x=\sup A$ with $A\subseteq Y$, then $x\in Y$; and a \textit{projection band} if it is a band and there exists a linear projection $P:X\rightarrow X$ such that $P(X)=Y$ and $0\leq Px\leq x$ for every $x\in X_+$. There are also several relevant classes of operators between Banach lattices, such as:
\begin{enumerate}
    \item \textbf{Positive operators:} An operator $T:X\rightarrow Y$ is \textit{positive} if $T(X_+)\subseteq Y_+$.
    \item \textbf{Regular operators:} An operator $T:X\rightarrow Y$ is \textit{regular} if it is the difference of two positive operators.
    \item \textbf{Lattice homomorphisms:} An operator $T:X\rightarrow Y$ is a \textit{lattice homomorphism} if it preserves lattice operations, that is, if $T(x\vee y)=(Tx)\vee (Ty)$ for every $x,y\in X$ (equivalently, $T(x\wedge y)=(Tx)\wedge (Ty)$, or $T|x|=|Tx|$).
    \item \textbf{(Almost) interval preserving operators:} A positive operator $T:X\rightarrow Y$ is \textit{(almost) interval preserving} if $T[0,x]=[0,Tx]$ (respectively, $T[0,x]$ is dense in $[0,Tx]$) for ever $x\in X$, where $[0,x]:=\{y\in X:0\leq y\leq x\}$ is an \textit{order interval}.
\end{enumerate}
Observe that the image of a lattice homomorphism is a sublattice and its kernel is an ideal, while the image of an interval preserving operator is an ideal. It should also be noted (cf. \cite[Theorem 1.4.19]{MN}) that the adjoint of a lattice homomorphism is an interval preserving operator, and the adjoint of an almost interval preserving operator is a lattice homomorphism.\\

Throughout these notes, we will constantly use the following property, commonly known as \textit{Krivine's positively homogeneous functional calculus}. It is worth mentioning that this feature is not exclusive to Banach lattices, but can actually be obtained for uniformly complete Archimedean vector lattices. Let us consider $\Hcal^n$, the set of continuous positively homogeneous functions over $\R^n$, i.e., the functions $f:\R^n\rightarrow \R$ such that $f(\lambda t_1,\ldots, \lambda t_n)=\lambda f(t_1,\ldots, t_n)$ for every $\lambda \geq 0$ and $t_1,\ldots,t_n\in \R$. Note that $\Hcal^n$ can be identified with $C(S_{\ell_\infty^n})$, the space of continuous functions on $S_{\ell_\infty^n}$, the unit sphere of $\ell_\infty^n$, since by homogeneity the values of any function in $\Hcal^n$ are prescribed by its values on $S_{\ell_\infty^n}$. Therefore, we can endow $\Hcal^n$ with the norm $\|\cdot\|_\infty$ of the supremum on $S_{\ell_\infty^n}$. Then for every $x_1,\ldots,x_n \in X$ and $f\in \Hcal^n$ there exists a unique way of defining the expression $f(x_1,\ldots,x_n)$ in $X$, in such a way that the operator $V:(\Hcal^n,\|\cdot\|_\infty)\rightarrow X$ that assigns to every $f\in \Hcal^n$ the element $Vf=f(x_1,\ldots,x_n)$ is a lattice homomorphism, with $\|V\|\leq \|\bigvee_{i=1}^n|x_i|\|$. More precisely (cf. \cite[Theorem 2.1.20]{MN}, see also \cite[Section 1.d]{LT2} or \cite[Chapter 16]{DJT}):

\begin{thm}\label{thm: functional calculus}
    Let $X$ be a Banach lattice, and $\overline{x}=(x_1,\ldots,x_n)\subseteq X$. For $i=1,\ldots,n$, let $p_i\in \Hcal^n$ denote the projection onto the $i$-th coordinate of $\R^n$. Then, there exists a unique lattice homomorphism $V_{\overline{x}}:\Hcal^n \rightarrow X$ such that $V_{\overline{x}}p_i=x_i$ for each $i=1,\ldots,n$. Moreover,
    \[|V_{\overline{x}}f|\leq \|f\|_\infty \bigvee_{i=1}^n |x_i|.\] 
\end{thm}

Note that, if $T:X\rightarrow Y$ is a lattice homomorphism between Banach lattices, then for every $x_1,\ldots,x_n\in X$, the uniqueness of the functional calculus yields that $V_{T\overline{x}}=T\circ V_{\overline{x}}$. In other words, lattice homomorphisms preserve functional calculus. A procedure for constructing such expressions $f(x_1,\ldots,x_n)$ is given by Kakutani's Representation Theorem for $AM$-spaces (cf. \cite[Theorem 2.1.3]{MN}). Namely, if $e=\bigvee_{i=1}^n|x_i|$, we can represent the ideal $I_e$ generated by $e$ as a space $C(K)$ for a certain compact Hausdorff space $K$. There, the function $f(x_1,\ldots,x_n)(t)=f(x_1(t),\ldots,x_n(t))$ can be defined for any continuous function $f$ of $n$ variables, and it can be shown that the corresponding element in $I_e$ does not depend on the choice of $e$ whenever $f$ is positively homogeneous.\\ 

\subsection{Properties of \texorpdfstring{$p$}{}-sums}

Functional calculus allows us to provide meaning in arbitrary Banach lattices to expressions such as the geometric means $|x_1|^{\theta_1}\cdot\ldots\cdot |x_n|^{\theta_n}$, where $\theta_i\in (0,1)$ are such that $\theta_1+\ldots+\theta_n=1$ or the $p$-sums $\intoo[0]{\sum_{i=1}^n |x_i|^p}^\frac{1}{p}$ for $0<p\leq \infty$ (for $p=\infty$ we mean $\bigvee_{i=1}^n|x_i|$). As one can expect, $p$-sums will be central in the theory of convexity and concavity that we will develop in the next sections, so let us recall some of their properties. First of all, note that for $1\leq p\leq \infty$, we can additionally represent the $p$-sums as follows:

\begin{prop}\label{prop: p-sum as supremum}
    Let $x_1,\ldots,x_n\in X$, $1\leq p\leq \infty$ and $\frac{1}{p}+\frac{1}{p^*}=1$. Then
    \[\intoo[3]{\sum_{i=1}^n |x_i|^p}^\frac{1}{p}=\sup \cbr{\sum_{i=1}^n a_ix_i:  (a_i)_{i=1}^n\in B_{\ell_{p^*}^n}}.\]
\end{prop}

\begin{proof}
    Let $e=\intoo[0]{\sum_{i=1}^n |x_i|^p}^\frac{1}{p}\in X$ and $A=\cbr{\sum_{i=1}^n a_ix_i: (a_i)_{i=1}^n\in B_{\ell_{p^*}^n}}$. We need to show that $e=\sup A$. To do so, consider the ideal $I_e$ generated by $e$ and represent it as a $C(K)$ space. Since the equality 
    \[\intoo[3]{\sum_{i=1}^n |t_i|^p}^\frac{1}{p}=\sup \cbr{\sum_{i=1}^n a_it_i: (a_i)_{i=1}^n\in B_{\ell_{p^*}^n}}\]
    holds for any choice $t_1,\ldots,t_n\in\R$, and in $C(K)$ the value at $s\in K$ of the function $e=\intoo[0]{\sum_{i=1}^n |x_i|^p}^\frac{1}{p}$ is given by $\intoo[0]{\sum_{i=1}^n |x_i(s)|^p}^\frac{1}{p}$, it follows that the proposition holds in $C(K)$. Since lattice isomorphisms preserve arbitrary suprema, the same equality holds in $X$.
\end{proof}

In particular, $p$-sums are preserved by positive operators:

\begin{lem}\label{lem: positive operators and functional calculus}
    Let $X$ and $Y$ be Banach lattices, $T:X\rightarrow Y$ a positive operator and $1\leq p\leq \infty$. Then, for every $x_1,\ldots,x_n\in X$
    \[\intoo[3]{\sum_{i=1}^n |Tx_i|^p}^\frac{1}{p}\leq T \intoo[3]{\sum_{i=1}^n |x_i|^p}^\frac{1}{p}.\]
\end{lem}

\begin{proof}
    Note that, by \Cref{prop: p-sum as supremum}, for every $(a_i)_{i=1}^n\in B_{\ell_{p^*}^n}$ we have
    \[\sum_{i=1}^n a_i Tx_i= T\intoo[3]{\sum_{i=1}^n a_i x_i}\leq T \intoo[3]{\sum_{i=1}^n |x_i|^p}^\frac{1}{p},\]
    so again by \Cref{prop: p-sum as supremum} we conclude that
    \[\intoo[3]{\sum_{i=1}^n |Tx_i|^p}^\frac{1}{p}\leq T \intoo[3]{\sum_{i=1}^n |x_i|^p}^\frac{1}{p}. \qedhere\]
\end{proof}

We will eventually need to evaluate $p$-sums of disjoint elements, that is, elements $x,y\in X$ such that $|x|\wedge |y|=0$. This task is much easier, as the following lemma shows.

\begin{lem}\label{lem: disjoint p sums}
    Let $x_1,\ldots,x_n\in X$ be pairwise disjoint elements, and $1<p<\infty$. Then
    \[\sum_{i=1}^n |x_i| = \intoo[3]{\sum_{i=1}^n |x_i|^p}^\frac{1}{p} = \bigvee_{i=1}^n |x_i|. \]
\end{lem}

\begin{proof}
    The norm one inclusions $\ell_1^n\subseteq \ell_p^n \subseteq \ell_\infty^n$ imply that
    \[\sum_{i=1}^n |x_i| \geq  \intoo[3]{\sum_{i=1}^n |x_i|^p}^\frac{1}{p} \geq \bigvee_{i=1}^n |x_i|. \]
    Now, since the vectors $x_1,\ldots,x_n$ are disjoint, it follows that $\bigvee_{i=1}^n |x_i|=\sum_{i=1}^n |x_i|$, so the statement follows.
\end{proof}

The monotonicity of $p$-sums also holds in Banach lattices.

\begin{prop}\label{prop: monotonicity of p-sums}
    Let $0<p\leq \infty$ and $x_1,\ldots,x_n,y_1,\ldots,y_n\in X$ such that $|x_i|\leq |y_i|$ for every $i=1,\ldots,n$. Then $\intoo{\sum_{i=1}^n |x_i|^p}^\frac{1}{p}\leq \intoo{\sum_{i=1}^n |y_i|^p}^\frac{1}{p}$.
\end{prop}

\begin{proof}
   Observe that the function of real variables $f(t_1,\ldots,t_n)=\intoo{\sum_{i=1}^n |t_i|^p}^\frac{1}{p}$ satisfies $f(s_1,\ldots,s_n)\leq f(t_1,\ldots,t_n)$ whenever $0\leq s_i\leq t_i$ for every $i=1,\ldots,n$, so the proposition is true in every $C(K)$ space as well. Now, let $e=\bigvee_{i=1}^n |y_i|$ and note that every $x_i$ and $y_i$ belongs to the ideal $I_e$ generated by $e$. Kakutani's Representation Theorem for $AM$-spaces and the uniqueness of the functional calculus yield the result.
\end{proof}

\begin{prop}\label{prop: Holder in BL}
    Let $X$ be a Banach lattice and $1\leq p\leq \infty$. Given $x_1,\ldots,x_n\in X$ and $x_1^*,\ldots,x_n^*\in X^*$, one has
    \[\left|\sum_{i=1}^n x_i^*(x_i) \right|\leq \intoo[3]{\sum_{i=1}^n |x_i^*|^{p^*}}^\frac{1}{p^*} \intoo[4]{\intoo[3]{\sum_{i=1}^n |x_i|^p}^\frac{1}{p}}. \]
\end{prop}

\begin{proof}
    Let $e\in X$ be such that $x_1,\ldots,x_n\in I_e$, and let us endow $I_e$ with the principal ideal norm $\|x\|_e=\inf\{\lambda>0:|x|\leq \lambda e$ for every $x\in I_e$. Note that $\|x\|\leq \lambda \|e\|$ whenever $|x|\leq \lambda e$, so $\|x\|\leq \|x\|_e \|e\|$. This implies that the restrictions of any functional $x^*\in X^*$ to $I_e$ is continuous with respect to the principal ideal norm:
    \[|x^*(x)|\leq \|x^*\|\|x\|\leq \|e\|\|x^*\|\|x\|_e.\]
    
    We claim that without loss of generality we can assume $X=(I_e,\|\cdot\|_e)$. Indeed, let $J:C(K)\rightarrow I_e\subseteq X$ be the isometric lattice identification of $I_e$ with a $C(K)$ space for some Hausdorff compact space $K$. By the above, we have that for every $f\in C(K)$
    \[\|Jf\|\leq \|e\|\|Jf\|_e =\|e\|\|f\|_\infty,\]
    so $J$ is a bounded operator from $C(K)$ into $X$. Note that $J^*:X^*\rightarrow C(K)^*$ behaves as the restriction operator to $I_e$, since $J^*x^*(f)=x^*(Jf)$ for every $f\in C(K)$ and $x^*\in X^*$. Moreover, since $J$ is an interval preserving lattice homomorphism ($J(C(K))=I_e$ is an ideal of $X$), $J^*$ is again an interval preservin lattice homomorphism. In particular, both $J$ and $J^*$ preserve the functional calculus of $C(K)$ and $X^*$, respectively. Let $f_1,\ldots,f_n\in C(K)$ be such that $Jf_i=x_i$, and denote by $\mu_i=J^*x_i^*$ for every $i=1,\ldots,n$. Then
    \begin{align*}
         \intoo[3]{\sum_{i=1}^n |x_i^*|^{p^*}}^\frac{1}{p^*} \intoo[4]{\intoo[3]{\sum_{i=1}^n |x_i|^p}^\frac{1}{p}} & = \intoo[3]{\sum_{i=1}^n |x_i^*|^{p^*}}^\frac{1}{p^*} \intoo[4]{\intoo[3]{\sum_{i=1}^n |J f_i|^p}^\frac{1}{p}} \\
        =  \intoo[3]{\sum_{i=1}^n |x_i^*|^{p^*}}^\frac{1}{p^*} \intoo[4]{J\intoo[3]{\sum_{i=1}^n |f_i|^p}^\frac{1}{p}} & =  J^*\intoo[3]{\sum_{i=1}^n |x_i^*|^{p^*}}^\frac{1}{p^*} \intoo[4]{\intoo[3]{\sum_{i=1}^n |f_i|^p}^\frac{1}{p}}\\
        = \intoo[3]{\sum_{i=1}^n |J^*x_i^*|^{p^*}}^\frac{1}{p^*} \intoo[4]{\intoo[3]{\sum_{i=1}^n |f_i|^p}^\frac{1}{p}} &=\intoo[3]{\sum_{i=1}^n |\mu_i|^{p^*}}^\frac{1}{p^*} \intoo[4]{\intoo[3]{\sum_{i=1}^n |f_i|^p}^\frac{1}{p}}.
    \end{align*}
    Therefore, it suffices to show the statement when $X=C(K)$. Let $f_1,\ldots,f_n\in C(K)$ and $\mu_1,\ldots,\mu_n \in C(K)^*$, which we can identify with measures over $C(K)$ by Riesz Representation Theorem. Let $\mu\in C(K)^*_+$ be a measure such that $|\mu_i|\leq \mu$ for every $i=1,\ldots,n$, so in particular $\mu_i\ll \mu$, and we can apply Radon--Nikodym Theorem to obtain a function $h_i\in L_1(\mu)$, the Radon--Nikodym derivative of $\mu_i$ with respect to $\mu$, such that $\mu_i(f)=\int_K f \,d\mu_i=\int_K f h_i \,d\mu$ for every $f\in C(K)$. Moreover, the identification of a measure in $B_\mu$, the band generated by $\mu$ in $C(K)^*$, with its Radon--Nikodym derivative in $L_1(\mu)$ is a lattice isometry. In particular, the derivative of $\intoo{\sum_{i=1}^n |\mu_i|^{p^*}}^\frac{1}{p^*}$ corresponds to $\intoo[3]{\sum_{i=1}^n |h_i|^{p^*}}^\frac{1}{p^*}$, so
    \begin{align*}
        \left|\sum_{i=1}^n \mu_i^*(f_i) \right| &= \left|\sum_{i=1}^n\int_K  f_i h_i \,d\mu  \right| \leq \int_K  \sum_{i=1}^n |f_i h_i| \,d\mu \\
        & \leq \int_K  \intoo[3]{\sum_{i=1}^n |h_i|^{p^*}}^\frac{1}{p^*} \intoo[3]{\sum_{i=1}^n |f_i|^p}^\frac{1}{p} \,d\mu \\
        & =
        \intoo[3]{\sum_{i=1}^n |\mu_i|^{p^*}}^\frac{1}{p^*} \intoo[4]{\intoo[3]{\sum_{i=1}^n |f_i|^p}^\frac{1}{p}},
    \end{align*}
    as we wanted to show.
\end{proof}

As a corollary, we get the following formula:

\begin{cor}\label{cor: RK for p sums}
    Let $X$ be a Banach lattice, $1\leq p\leq \infty$, and $x_1^*,\ldots,x_n^*\in X^*$. Then for every $x\in X_+$ we have
    \[\intoo[3]{\sum_{i=1}^n |x_i^*|^{p^*}}^\frac{1}{p^*}(x) =  \sup\cbr{\sum_{i=1}^n x_i^*(x_i): x_1,\ldots,x_n\in X, \intoo[3]{\sum_{i=1}^n |x_i|^{p}}^\frac{1}{p}\leq x}.\]
\end{cor}

\begin{proof}
    Let $x\in X_+$. First, assume that $x_1,\ldots,x_n\in X$ are such that $\intoo{\sum_{i=1}^n |x_i|^{p}}^\frac{1}{p}\leq x$. Then, by \Cref{prop: Holder in BL} it follows that
    \[\left|\sum_{i=1}^n x_i^*(x_i)\right| \leq \intoo[3]{\sum_{i=1}^n |x_i^*|^{p^*}}^\frac{1}{p^*} \intoo[4]{\intoo[3]{\sum_{i=1}^n |x_i|^p}^\frac{1}{p}} \leq  \intoo[3]{\sum_{i=1}^n |x_i^*|^{p^*}}^\frac{1}{p^*}(x). \]
    To prove the reverse inequality, we apply \Cref{prop: p-sum as supremum} and the Riesz--Kantorovich formula (cf. \cite[Corollary 1.3.4]{MN}) to obtain
    \begin{align*}
        \intoo[3]{\sum_{i=1}^n |x_i^*|^{p^*}}^\frac{1}{p^*}(x) &= \sup\cbr{\intoo[3]{\sum_{i=1}^n a_{i}x_i^*}: (a_{i})_{i=1}^n\in B_{\ell_p^n}}(x)\\
        &= \sup\cbr{\bigvee_{j=1}^m \intoo[3]{\sum_{i=1}^n a_{ij}x_i^*}: (a_{ij})_{i=1}^n\in B_{\ell_p^n}, j=1,\ldots,m, m\in \N}(x)\\
        & = \sup\cbr{\intoo[3]{\bigvee_{j=1}^m \intoo[3]{\sum_{i=1}^n a_{ij}x_i^*}}(x): (a_{ij})_{i=1}^n\in B_{\ell_p^n}}\\
        & = \sup\cbr{\sum_{j=1}^m \intoo[3]{\sum_{i=1}^n a_{ij}x_i^*}(z_j): z_j\geq 0,\sum_{j=1}^m z_j=x, (a_{ij})_{i=1}^n\in B_{\ell_p^n}}\\
        & = \sup\cbr{\sum_{i=1}^n x_i^*\intoo{\sum_{j=1}^m a_{ij}z_j}: z_j\geq 0,\sum_{j=1}^m z_j=x, (a_{ij})_{i=1}^n\in B_{\ell_p^n}}.
    \end{align*}
    Note that the vectors $\sum_{j=1}^m a_{ij}z_j$, $1\leq i\leq n$ with $z_j\geq 0,\sum_{j=1}^m z_j=x, (a_{ij})_{i=1}^n\in B_{\ell_p^n}$ satisfy
    \[\intoo{\sum_{i=1}^n \left|\sum_{j=1}^m a_{ij}z_j\right|^{p}}^\frac{1}{p}\leq \sum_{j=1}^m \intoo{\sum_{i=1}^n | a_{ij}z_j|^{p}}^\frac{1}{p} =\sum_{j=1}^m z_j \intoo{\sum_{i=1}^n | a_{ij}|^{p}}^\frac{1}{p} \leq \sum_{j=1}^m z_j = x, \]
    where the first inequality follows from the triangular inequality of the norm of $\ell_p^n$.  Therefore, we conclude that 
    \[\intoo[3]{\sum_{i=1}^n |x_i^*|^{p^*}}^\frac{1}{p^*}(x) \leq  \sup\cbr{\sum_{i=1}^n x_i^*(x_i): x_1,\ldots,x_n\in X, \intoo[3]{\sum_{i=1}^n |x_i|^{p}}^\frac{1}{p}\leq x},\]
    so the statement holds.
\end{proof}

We will eventually need to use geometric means as well. The following lemma establishes some useful properties of these expressions.

\begin{lem}\label{lem: geometric mean}
    Let $X$ be a Banach lattice.
    \begin{enumerate}
        \item If $0<\theta<1$ and $x,y\in X$, then
        \[\||x|^\theta|y|^{1-\theta}\|\leq \|x\|^\theta \|y\|^{1-\theta}. \]
        \item Let $1\leq p<r<q\leq \infty$, $0<\theta<1$ given by $\frac{1}{r}=\frac{\theta}{p}+\frac{1-\theta}{q}$, $x_1,\ldots,x_n\in X$ and $(a_i)_{i=1}^n$ positive scalars. Then,
        \[\norm[3]{\intoo[3]{\sum_{i=1}^n a_i|x_i|^{r}}^\frac{1}{r}}\leq \norm[3]{\intoo[3]{\sum_{i=1}^n a_i|x_i|^{p}}^\frac{1}{p}}^\theta \norm[3]{\intoo[3]{\sum_{i=1}^n a_i|x_i|^{q}}^\frac{1}{q}}^{1-\theta}.\]
    \end{enumerate}
\end{lem}

\begin{proof}
    $(1)$ From Young's inequality in $\R^2$
    \[|s|^\theta|t|^{1-\theta}\leq \theta |s|+(1-\theta)|t|\]
    it follows that for every $c>0$ we have
    \[ \||x|^\theta|y|^{1-\theta}\|=\||c^\frac{1}{\theta}x|^\theta|c^{-\frac{1}{1-\theta}}y|^{1-\theta}\|\leq \theta c^\frac{1}{\theta}\|x\|+(1-\theta) c^{-\frac{1}{1-\theta}}\|y\|. \]
    In particular, evaluating at $c=\intoo[1]{\frac{\|y\|}{\|x\|}}^{\theta(1-\theta)}$ we obtain the desired inequality.\\

    $(2)$ Hölder's inequality with exponent $s=\frac{p}{\theta r}$ yields
    \[\intoo[3]{\sum_{i=1}^n a_i|t_i|^{r}}^\frac{1}{r}=\intoo[3]{\sum_{i=1}^n a_i^{\frac{1}{s}+\frac{1}{s^*}}|t_i|^{r\theta +r(1-\theta)}}^\frac{1}{r}\leq \intoo[3]{\sum_{i=1}^n a_i|t_i|^{p}}^\frac{\theta}{p} \intoo[3]{\sum_{i=1}^n a_i|t_i|^{q}}^\frac{1-\theta}{q}\]
    for every $t_1,\ldots,t_n\in \R$. Therefore, the same inequality holds in $X$, writing $x_1,\ldots,x_n$ instead. Applying $(1)$ we obtain $(2)$.
\end{proof}

\section{Convexity and concavity}\label{sec: main definitions}
\subsection{Main definitions}
Now that we have reviewed the basic notions of Banach lattice theory and we have cleared the meaning of $p$-sums in arbitrary Banach lattices, we can introduce the main concepts of this memoir: $(p,q)$-convexity and $(p,q)$-concavity. 

\begin{defn}\label{defn: convexity and concavity}
    Let $X$ be a Banach lattice, $E$ a Banach space and $1\leq p,q\leq \infty$.
    \begin{enumerate}
        \item An operator $T:E\rightarrow X$ is said \emph{$(p,q)$-convex} if there exists a constant $K>0$ such that
        \begin{equation}\label{eq: (p,q)-convex operator}
            \norm[3]{\intoo[3]{\sum_{i=1}^n |Tu_i|^q}^\frac{1}{q}}\leq K \intoo[3]{\sum_{i=1}^n \|u_i\|^p}^\frac{1}{p}
        \end{equation}
        for every $u_1,\ldots,u_n\in E$ and $n\in \N$. The best constant $K$ satisfying \eqref{eq: (p,q)-convex operator} is called the \emph{$(p,q)$-convexity constant of $T$} and is denoted $K^{(p,q)}(T)$.
        \item An operator $T:X\rightarrow E$ is said \emph{$(p,q)$-concave} if there exists a constant $K>0$ such that
        \begin{equation}\label{eq: (p,q)-concave operator}
            \intoo[3]{\sum_{i=1}^n \|Tx_i\|^p}^\frac{1}{p} \leq K \norm[3]{\intoo[3]{\sum_{i=1}^n |x_i|^q}^\frac{1}{q}}
        \end{equation}
          for every $x_1,\ldots,x_n\in X$ and $n\in \N$. The best constant $K$ satisfying \eqref{eq: (p,q)-concave operator} is called the \emph{$(p,q)$-concavity constant of $T$} and is denoted $K_{(p,q)}(T)$.
        \item A Banach lattice $X$ is said \emph{$(p,q)$-convex} (respectively, \emph{$(p,q)$-concave}) if the identity operator $i_X:X\rightarrow X$ is $(p,q)$-convex (respectively, $(p,q)$-concave), and $K^{(p,q)}(X):=K^{(p,q)}(i_X)$ is called the \emph{$(p,q)$-convexity constant of $X$} (respectively, $K_{(p,q)}(X):=K_{(p,q)}(i_X)$ is called the \emph{$(p,q)$-concavity constant of $X$}).
        \item If $p=q$, we will simply say that $T$ or $X$ is $p$-convex (respectively, $p$-concave) and write $K^{(p)}(T)$ or $K^{(p)}(X)$ (respectively, $K_{(p)}(T)$ or $K_{(p)}(X)$) instead.
    \end{enumerate}
\end{defn}

A similar property can be introduced for Banach lattices when we require the inequalities above to be witnessed only on disjoint vectors.

\begin{defn}\label{defn: upe and lpe}
    Let $X$ be a Banach lattice and $1\leq p\leq \infty$.
    \begin{enumerate}
        \item $X$ is said to satisfy an \emph{upper $p$-estimate} if there exists a constant $K\geq 1$ such that
        \begin{equation}\label{eq: upe}
            \norm[3]{\sum_{i=1}^n x_i}\leq K \intoo[3]{\sum_{i=1}^n \|x_i\|^p}^\frac{1}{p}
        \end{equation}
        for every (positive) pairwise disjoint $x_1,\ldots,x_n\in X$ and $n\in \N$. The best constant $K$ satisfying \eqref{eq: upe} is called the \emph{upper $p$-estimates constant of $X$} and is denoted $K^{(\uparrow p)}(X)$.
        \item $X$ is said to satisfy a \emph{lower $p$-estimate} if there exists a constant $K\geq 1$ such that
        \begin{equation}\label{eq: lpe}
            \intoo[3]{\sum_{i=1}^n \|x_i\|^p}^\frac{1}{p} \leq K \norm[3]{\sum_{i=1}^n x_i}
        \end{equation}
        for every (positive) pairwise disjoint $x_1,\ldots,x_n\in X$ and $n\in \N$. The best constant $K$ satisfying \eqref{eq: lpe} is called the \emph{lower $p$-estimates constant of $X$} and is denoted $K_{(\downarrow p)}(X)$.
    \end{enumerate}
\end{defn}

Note that, by \Cref{lem: disjoint p sums}, 
\[\sum_{i=1}^n |x_i| = \intoo[3]{\sum_{i=1}^n |x_i|^p}^\frac{1}{p} = \bigvee_{i=1}^n |x_i|, \]
whenever $x_1,\ldots, x_n$ are pairwise disjoint. Therefore, satisfying \eqref{eq: upe} is actually the same as satisfying the $p$-convexity condition only on disjoint sequences of vectors. Likewise, lower $p$-estimates correspond to witnessing the $p$-convexity condition only on disjoint vectors.

\begin{rem}\label{rem: notation}
    The terminology introduced in the previous definitions is consistent with that of standard textbooks such as \cite[Chapter 16]{DJT}, and \cite[Sections 1.d and 1.f]{LT2}. It is worth mentioning that the notation for this topic in the literature has changed several times before a standard terminology was finally adopted. The original notation that Krivine \cite{KrivineExp2223} used to introduce $p$-convex and $p$-concave operators was type $\geq p$ and type $\leq p$, respectively. The definition was later extended by Maurey \cite{MaureyExp2425} to type $\geq (q,p)$ and type $\leq (q,p)$ operators, which correspond to the modern terminology of $(p,q)$-convex and $(p,q)$-concave operators, respectively. Note, in particular, that the order of the indices is reversed in both cases. This also happens in Schwarz's book \cite{Schwarz}, for $(p,q)$-convexity only. Additionally, in Meyer-Nieberg's book \cite[Section 2.8]{MN}, a lattice norm satisfying an upper (respectively, lower) $p$-estimate with constant one is called $p$-subadditive (respectively, $p$-superadditive). In this same reference, $(p,\infty)$-convex and $(p,1)$-concave operators are called $p$-majorizing and cone $p$-summing operators, while in \cite{Schwarz} they are called $p$-subadditive and $p$-superadditive, respectively.
\end{rem}

\begin{rem}\label{rem: facts about convexity}
    Let us make some comments about Definitions \ref{defn: convexity and concavity} and \ref{defn: upe and lpe}:
    \begin{enumerate}
        \item Every $(p,q)$-convex or $(p,q)$-concave operator must be bounded, and its operator norm is bounded by the corresponding $(p,q)$-convexity or $(p,q)$-concavity constant. 
        \item Every bounded operator $T:E\rightarrow X$ is $1$-convex with $K^{(1)}(T)=\|T\|$, as every Banach lattice satisfies the triangular inequality, and every operator $T:X\rightarrow E$ is $\infty$-concave with $K_{(\infty)}(T)= \|T\|$, since the norm in $X$ is monotone with respect to the order. The same argument yields that every Banach lattice satisfies an upper $1$-estimate and a lower $\infty$-estimate. Moreover, in \Cref{sec: abstract factorizations} we will see that upper $\infty$-estimates and $\infty$-convexity, as well as lower $1$-estimates and $1$-concavity, are actually equivalent (and they are related to $AM$ and $AL$-spaces, respectively), so \Cref{defn: upe and lpe} is only relevant for $1<p<\infty$.
        \item It is clear that every $p$-convex ($p$-concave) Banach lattice satisfies an upper (lower) $p$-estimate with the same constant. It is also true (see \Cref{thm: upe equivalent to mixed convexity}) that for Banach lattices the notions of upper $p$-estimates and $(p,\infty)$-convexity coincide, as well as lower $p$-estimates and $(p,1)$-concavity. \item The inequality $\intoo{\sum_{i=1}^n |t_i|^s}^\frac{1}{s}\leq \intoo{\sum_{i=1}^n |t_i|^r}^\frac{1}{r}$ whenever $1\leq r\leq s\leq \infty$, yields that any $(p,q)$-convex operator is $(r,s)$-convex for every $r\leq p$ and $s\geq q$, any $(p,q)$-concave operator is $(r,s)$-concave for every $r\geq p$ and $s\leq q$, any Banach lattice with an upper $p$-estimate satisfies an upper $q$-estimate for every $q< p$ and any Banach lattice with a lower $p$-estimate satisfies a lower $q$-estimate for every $q> p$, and the corresponding constants do not increase. 
        \item If $q<p$ and $T:E\rightarrow X$ is $(p,q)$-convex, then necessarily $T=0$. Indeed, for any $u\in E$, evaluating \eqref{eq: (p,q)-convex operator} with $n$ copies of the vector $u$, it follows that $n^\frac{1}{q}\|Tu\|\leq n^\frac{1}{p} K^{(p,q)}(T) \|u\|$ for every $n\in \N$, so $Tx=0$. Similarly, if $p<q$ and $T:X\rightarrow E$ is $(p,q)$-concave, then $T=0$. Actually, we will show in \Cref{thm: relevant cases of concavity} and \Cref{cor: relevant cases of convexity} that the only relevant cases for $(p,q)$-convexity are $q=p$ and $q=\infty$, and $q=p$ and $q=1$ for $(p,q)$-concavity.
        \item If $T:E\rightarrow X$ is $(p,q)$-convex and $S:F\rightarrow E$ is bounded, then $TS$ is also $(p,q)$-convex with constant $\|S\|K^{(p,q)}(T)$. Similarly, if $T:X\rightarrow E$ is $(p,q)$-concave and $S:E\rightarrow F$ is bounded, then $ST$ is $(p,q)$-concave with constant $\|S\|K_{(p,q)}(T)$.
        \item Using \Cref{lem: positive operators and functional calculus}, it also follows that if $T:E\rightarrow X$ is $(p,q)$-convex and $R:X\rightarrow Y$ is a positive operator between Banach lattices, then $RT$ is $(p,q)$-convex with constant $\|R\|K^{(p,q)}(T)$. Analogously, if $T:X\rightarrow E$ is $(p,q)$-concave and $R:Y\rightarrow X$ is positive, then $TR$ is $(p,q)$-concave with constant $\|R\|K_{(p,q)}(T)$.
        \item Observe that $(p,q)$-convexity, $(p,q)$-concavity, upper $p$-estimates and lower $p$-estimates are lattice isomorphic properties, and they are inherited by closed sublattices and lattice quotients. Moreover, an arbitrary $\ell_\infty$ sum of $(p,q)$-convex Banach lattices (respectively, Banach lattices with upper $p$-estimates) with uniformly bounded constant is again $(p,q)$-convex (satisfies an upper $p$-estimate). Similarly, an arbitrary $\ell_1$ sum of $(p,q)$-concave Banach lattices (respectively, Banach lattices with lower $p$-estimates) with uniformly bounded constant is again $(p,q)$-concave (satisfies a lower $p$-estimate). As we will see in Theorems \ref{thm: representation infty sums} and \ref{thm: representation 1 sums}, Banach lattices with convexity properties can be represented as sublattices of $\ell_\infty$ sums of certain model spaces, while concave Banach lattices are in correspondence with quotients of $\ell_1$ sums. 
    \end{enumerate}
\end{rem}

\subsection{\texorpdfstring{The spaces $L_p$, $L_{p,\infty}$, $L_{p,1}$}{}, and other examples}\label{sec: weak Lp}

Precisely, $L_p(\mu)$ is the prototypical example of $p$-convexity and $p$-concavity, as it satisfies both properties with constant one. From \Cref{thm: p-conv implies q-conv}, we will deduce that $L_p(\mu)$ is also $r$-convex for every $1\leq r<p$ and $s$-concave for every $p<s\leq \infty$. In fact, this is optimal: if $L_p(\mu)$ is infinite dimensional, then it is $r$-convex only for $1\leq r\leq p$, and $s$-concave for $p\leq s\leq \infty$. This follows from the fact that if an infinite dimensional Banach lattice $X$ satisfies an upper $r$-estimate and a lower $s$-estimate, then necessarily $r\leq s$. Indeed, for any $n\in\N$ let us fix $n$ normalized disjoint vectors $x_1,\ldots,x_n\in X_+$. Then, by \Cref{lem: disjoint p sums},
\begin{align*}
    n^\frac{1}{s}& =\intoo[3]{\sum_{i=1}^n \|x_i\|^s}^\frac{1}{s} \leq K_{(\downarrow s)}(X) \norm[3]{\sum_{i=1}^n x_i}\\
    & \leq  K_{(\downarrow s)}(X) K^{(\uparrow r)}(X)\intoo[3]{\sum_{i=1}^n \|x_i\|^r}^\frac{1}{r}= K_{(\downarrow s)}(X) K^{(\uparrow r)}(X) n^\frac{1}{r}.
\end{align*}
If we assume that $r>s$, then making $n$ arbitrarily big we get a contradiction.\\

The family of Lorentz spaces also provides interesting examples. In fact, the Lorentz spaces $L_{p,\infty}(\mu)$, sometimes called weak $L_p$ spaces, and $L_{p,1}(\mu)$, are the canonical spaces for upper and lower $p$-estimates, respectively. They satisfy upper $p$-estimates (respectively, lower $p$-estimates) but are not $p$-convex (respectively, $p$-concave). \\

Let us recall the definition of the Lorentz spaces. Let $(\Omega,\Sigma,\mu)$ be a measure space and $f$ a measurable function on $\Omega$. We define its decreasing rearrangement $f^*:[0,\infty)\rightarrow [0,\infty]$ by
\[f^*(t):=\inf\{\lambda>0 : \mu\{|f|>\lambda\}\leq t\}.\]
Given $1<p< \infty$ and $1\leq q\leq \infty$, the Lorentz space $L_{p,q}(\mu)$ is the set of all measurable functions $f:\Omega\rightarrow \R$ such that the quasinorm $\vvvert f\vvvert_{L_{p,q}}$ is finite, where
\[\vvvert f\vvvert_{L_{p,q}}:=\intoo[3]{\int_0^\infty (t^\frac{1}{p}f^*(t))^q\frac{dt}{t}}^\frac{1}{q}\]
for $1\leq q<\infty$ and 
\[ \vvvert f\vvvert_{L_{p,\infty}}: = \sup_{t>0} t^\frac{1}{p} f^*(t)=\sup_{\lambda>0} \lambda\mu(\{|f|>\lambda\})^{\frac{1}{p}} \]
when $q=\infty$. Recall that when $q<\infty$ the dual of $L_{p,q}$ is $L_{p^*,q^*}$. In particular, $L_{p,\infty}$ is the dual of $L_{p^*,1}$. In \cite{Creekmore}, the optimal convexity and concavity of the Lorentz spaces $L_{p,q}$ was computed:

\begin{thm}
    Let $(\Omega,\Sigma ,\mu)$ be a measure space with $\mu$ a non-atomic $\sigma$-finite measure.
    \begin{enumerate}
        \item Let $1\leq q<p$. Then $L_{p,q}$ is not $p$-concave, but satisfies a lower $p$-estimate with constant one. Furthermore, it is $q$-convex with constant one.
        \item Let $p<q\leq \infty$. Then $L_{p,q}$ is not $p$-convex, but satisfies an upper $p$-estimate with constant one. Furthermore, it is $q$-concave with constant one.
    \end{enumerate}
\end{thm}

Let us focus our attention on weak $L_p$ spaces. They arise naturally in harmonic analysis as an enlarged target space for certain operators that are not bounded into $L_p$. It should be noted that the quasinorm $\vvvert \cdot\vvvert_{L_{p,\infty}}$ is not a norm in general. However, since $1<p<\infty$, for each $1\leq r<p$, we can equip $L_{p,\infty}(\mu)$ with the equivalent norm \cite[Exercise 1.1.12]{GrafakosCFA}
\[    \|f\|_{L_{p,\infty}^{[r]}}:=\sup_{0<\mu(A)<\infty} \mu(A)^{\frac{1}{p}-\frac{1}{r}} \left(\int_A |f|^r\,d\mu\right)^{\frac{1}{r}}. \]

\begin{lem}\label{lem: renorming weak Lp}
    Let $f\in L_{p,\infty}(\mu)$ and $1\leq r<p$. Then
    \[\vvvert f\vvvert_{L_{p,\infty}} \leq \|f\|_{L_{p,\infty}^{[r]}} \leq \left(\frac{p}{p-r}\right)^{\frac{1}{r}}  \vvvert f\vvvert_{L_{p,\infty}}. \]
\end{lem}

\begin{proof}
    The first inequality follows from observing that for every $t>0$ we have
    \begin{align*}
        t\mu(\{|f|>t\})^{\frac{1}{p}} & =\mu(\{|f|>t\})^{\frac{1}{p}-\frac{1}{r}}\intoo[3]{\int_{\{|f|>t\}}t^r \,d\mu}^\frac{1}{r}\\
        & \leq \mu(\{|f|>t\})^{\frac{1}{p}-\frac{1}{r}}\intoo[3]{\int_{\{|f|>t\}}|f|^r \,d\mu}^\frac{1}{r} \leq \|f\|_{L_{p,\infty}^{[r]}}.
    \end{align*}
    In order to prove the second inequality, we start by proving that for any measurable set $A$ with $0<\mu(A)<\infty$ we have
    \[\int_A |f|^r \,d\mu \leq \frac{p}{p-r}\mu(A)^{1-\frac{r}{p}}\vvvert f\vvvert_{L_{p,\infty}}^r.\]
    To do so, observe first that, by the definition of the quasinorm $\vvvert \cdot \vvvert_{L_{p,\infty}}$, it follows that
    \[\mu (A\cap\{|f|>\alpha\})\leq \min \{\mu(A), \alpha^{-p}\vvvert f\vvvert_{L_{p,\infty}}^p\}\]
    for every $\alpha>0$. Now, we decompose the integral over $A$ as
    \[\int_A |f|^r \,d\mu=\int_{A\cap\{|f|\leq\alpha\}} |f|^r \,d\mu+\int_{A\cap\{|f|>\alpha\}} |f|^r \,d\mu.\]
    The first term is easy to bound:
    \[\int_{A\cap\{|f|\leq\alpha\}} |f|^r \,d\mu\leq \alpha^r \mu(A\cap\{|f|\leq\alpha\})\leq \alpha^r \mu(A).\]
    In order to estimate the second term, we use a Fubini-type argument:
    \begin{align*}
        \int_{A\cap\{|f|>\alpha\}} |f|^r \,d\mu& =\int_{\{|f|>\alpha\}} \chi_A(\omega) \int_0^{|f(\omega)|}rt^{r-1} \,dt \,d\mu(\omega)\\
        & =\int_0^\infty rt^{r-1} \int_{\{|f|>\max\{t,\alpha\}\}} \chi_A(\omega) \,d\mu(\omega) \,dt\\
        & =\int_0^\infty rt^{r-1} \mu(A\cap\{|f|>\max\{t,\alpha\}\})  \,dt\\
        & \leq \int_0^\infty rt^{r-1} \min \{\mu(A), \max\{t,\alpha\}^{-p}\vvvert f\vvvert_{L_{p,\infty}}^p\} \,dt.
    \end{align*}
    If $\alpha<\mu(A)^{-\frac{1}{p}}\vvvert f\vvvert_{L_{p,\infty}}$, we can split this integral in two terms to obtain the desired inequality:
    \begin{align*}
        \int_{A\cap\{|f|>\alpha\}} |f|^r \,d\mu& \leq \int_0^\infty rt^{r-1} \min \{\mu(A), \max\{t,\alpha\}^{-p}\vvvert f\vvvert_{L_{p,\infty}}^p\} \,dt\\
        & =  \int_0^{\mu(A)^{-\frac{1}{p}}\vvvert f\vvvert_{L_{p,\infty}}} rt^{r-1}\mu(A) \,dt  +\int_{\mu(A)^{-\frac{1}{p}}\vvvert f\vvvert_{L_{p,\infty}}}^\infty rt^{r-1} t^{-p}\vvvert f\vvvert_{L_{p,\infty}}^p \,dt \\
        &= \mu(A)^{1-\frac{r}{p}} \vvvert f\vvvert_{L_{p,\infty}}^r+\frac{r}{p-r} \mu(A)^{1-\frac{r}{p}} \vvvert f\vvvert_{L_{p,\infty}}^r = \frac{p}{p-r} \mu(A)^{1-\frac{r}{p}} \vvvert f\vvvert_{L_{p,\infty}}^r.
    \end{align*}
    Rearranging the terms and taking the supremum over all sets $A$ of positive finite measure we obtain the second inequality of the statement. 
\end{proof}

Given $f\in L_{p,\infty}(\mu)$, the norm $\|f\|_{L_{p,\infty}^{[r]}}$ is increasing in $r\in [1,p)$. Indeed, if $1\leq r<s<p$, then for every set $A$ with $0<\mu(A)<\infty$ we can apply Hölder's inequality with exponent $\frac{s}{r}$ to obtain that
\[\mu(A)^{\frac{1}{p}-\frac{1}{r}} \left(\int_A |f|^r\,d\mu\right)^{\frac{1}{r}}\leq \mu(A)^{\frac{1}{p}-\frac{1}{s}} \left(\int_A |f|^s\,d\mu\right)^{\frac{1}{s}}.\]

For each $1\leq r<p$, $(L_{p,\infty}(\mu),\|\cdot\|_{L_{p,\infty}^{[r]}})$ is clearly a Banach lattice. As we mentioned earlier, it is the dual of $L_{p^*,1}(\mu)$, up to renorming. It might be convenient to compute the relation between these norms explicitly. If we denote the dual norm induced by $L_{p^*,1}$ by 
\[\vvvert f \vvvert_{(L_{p^*,1})^*}:=\sup_{\vvvert g \vvvert_{L_{p^*,1}}\leq 1} \abs[3]{\int f g \, d\mu },\]
then, we have that for every normalized $g\in L_{p^*,1}$,
\[\abs[3]{\int f g \, d\mu }\leq \int_0^\infty f^*(t)g^*(t) \, dt \leq  \int_0^\infty t^\frac{1}{p} f^*(t) t^{\frac{1}{p^*}-1}g^*(t) \, dt \leq \vvvert f \vvvert_{L_{p,\infty}} \vvvert g \vvvert_{L_{p^*,1}} =\vvvert f \vvvert_{L_{p,\infty}},\]
so $\vvvert f \vvvert_{(L_{p^*,1})^*}\leq \vvvert f \vvvert_{L_{p,\infty}}$. On the other hand, given any measurable subset $A$ with $0<\mu (A)<\infty$, it is easy to see that $(\chi_A)^*=\chi_{[0,\mu(A)]}$, so $\vvvert \chi_A \vvvert_{L_{p^*,1}}=p^* \mu(A)^\frac{1}{p^*}$, so
\[\mu(A)^{\frac{1}{p}-1} \int_A |f|\,d\mu = \mu(A)^{-\frac{1}{p^*}} \int |f| \chi_A\,d\mu \leq \mu(A)^{-\frac{1}{p^*}} \vvvert \chi_A \vvvert_{L_{p^*,1}} \vvvert f \vvvert_{(L_{p^*,1})^*} = p^* \vvvert f \vvvert_{(L_{p^*,1})^*} . \]
Taking the supremum over all subsets $A$, we conclude that
\[ \vvvert \cdot \vvvert_{(L_{p^*,1})^*} \leq \vvvert \cdot \vvvert_{L_{p,\infty}} \leq \|\cdot\|_{L_{p,\infty}^{[1]}} \leq p^* \vvvert \cdot \vvvert_{(L_{p^*,1})^*}.\]
From now on, unless stated otherwise, we will assume that $L_{p,\infty}(\mu)$ is endowed with the norm $\|\cdot\|_{L_{p,\infty}^{[1]}}$, and we will simply denote this norm by $\|\cdot\|_{L_{p,\infty}}$ or $\|\cdot\|_{p,\infty}$. As usual, when the measure space is the natural numbers with the counting measure, we will write $\ell_{p,\infty}$.\\

Weak $L_p$ spaces are the canonical example of Banach lattices satisfying upper $p$-estimates. As we will show in \Cref{cor: relevant cases of convexity}, every Banach lattice with upper $p$-estimates is $r$-convex for every $1\leq r<p$. However, for $L_{p,\infty}(\mu)$ this is straightforward to check:

\begin{prop}\label{prop: wlp upe and r convex}
    Let $1\leq r<p$. Then, $(L_{p,\infty}(\mu),\|\cdot\|_{L_{p,\infty}^{[r]}})$ satisfies an upper $p$-estimate and is $r$-convex, both with constant one.
\end{prop}

\begin{proof}
To check the first assertion, let $f_1,\ldots,f_n\in L_{p,\infty}(\mu)_+$ be disjoint functions. Given any measurable set $A$ with $0<\mu(A)<\infty$, let us denote by $A_i=A\cap \text{supp}f_i$, $i=1,\ldots,n$, so that
\begin{align*}
    &\mu(A)^{\frac{1}{p}-\frac{1}{r}} \left(\int_A \intoo[3]{\sum_{i=1}^n |f_i|}^r\,d\mu\right)^{\frac{1}{r}}  = \mu(A)^{\frac{1}{p}-\frac{1}{r}} \left(\sum_{i=1}^n\int_{A_i} |f_i|^r\,d\mu\right)^{\frac{1}{r}}\\
    & \leq  \left(\sum_{i=1}^n \left(\frac{\mu(A_i)}{\mu(A)}\right)^{1-\frac{r}{p}} \|f\|_{L_{p,\infty}^{[r]}}^r\right)^{\frac{1}{r}} \leq  \left(\sum_{i=1}^n \frac{\mu(A_i)}{\mu(A)} \right)^{\frac{p-r}{pr}} \left(\sum_{i=1}^n \|f\|_{L_{p,\infty}^{[r]}}^p\right)^{\frac{1}{p}}  \leq  \left(\sum_{i=1}^n \|f\|_{L_{p,\infty}^{[r]}}^p\right)^{\frac{1}{p}},
\end{align*}
where in the second inequality we have used Hölder's inequality with exponent $\frac{p}{r}$. Therefore, $(L_{p,\infty}(\mu),\|\cdot\|_{L_{p,\infty}^{[r]}})$ satisfies an upper $p$-estimate with constant one.\\ 

Now, given $f_1,\ldots,f_n \in L_{p,\infty}(\mu)$ and a measurable set $A$ with $0<\mu(A)<\infty$, it follows that
\[\mu(A)^{\frac{1}{p}-\frac{1}{r}} \intoo[3]{\int_A \sum_{i=1}^n|f_i|^r\, d\mu}^\frac{1}{r}=  \intoo[3]{ \sum_{i=1}^n \intoo[3]{\mu(A)^{\frac{1}{p}-\frac{1}{r}}\intoo[3]{\int_A|f_i|^r\, d\mu}^\frac{1}{r}}^r}^\frac{1}{r} \leq  \intoo[3]{ \sum_{i=1}^n \|f\|_{L_{p,\infty}^{[r]}}^r}^\frac{1}{r}.\]
Taking the supremum over all the sets $A$, we obtain
\[ \norm[3]{\intoo[3]{ \sum_{i=1}^n |f_i|^r}^\frac{1}{r}}_{L_{p,\infty}^{[r]}} \leq  \intoo[3]{ \sum_{i=1}^n \|f\|_{L_{p,\infty}^{[r]}}^r}^\frac{1}{r},\] 
so $(L_{p,\infty}(\mu),\|\cdot\|_{L_{p,\infty}^{[r]}})$ is $r$-convex with constant one.
\end{proof}

Weak $L_p$ spaces provide examples of Banach lattices that satisfy upper $p$-estimates but are not $p$-convex. In fact, we have the following:

\begin{prop}\label{prop: uniform copies of fd weak Lp}
    Let $X$ be a Banach lattice and $1<p<\infty$. Assume that $X$ contains uniform lattice isomorphic copies of $\ell_{p,\infty}^n$, that is, for every $n\in \N$ there exists a lattice isomorphic embedding $T_n:\ell_{p,\infty}^n \rightarrow X$ such that $\sup_n \|T_n\|\|T_n^{-1}\|<\infty$. Then $X$ is not $p$-convex.
\end{prop}

\begin{proof}
Fix $n\in \N$ and let $(e_k)_{k=0}^{n-1}$ be the canonical basis of $\ell_{p,\infty}^n$. For $j=0,\ldots,n-1$, we write $a^{(j)}=\sum_{k=1}^n \alpha_k e_{(k+j)_n}$ for the cyclic permutations of the coefficients $\alpha_k=k^\frac{1}{p^*}-(k-1)^\frac{1}{p^*}$, $1\leq k\leq n$, over the basis $(e_k)_{k=0}^{n-1}$. Here, by $i_n$ we mean the remainder of $i$ divided by $n$. It is clear that $\|a^{(j)}\|_{\ell_{p,\infty}^n}=1$ for every $j$. Moreover,
\[\intoo[3]{\sum_{j=0}^{n-1} |a^{(j)}|^p}^\frac{1}{p}=\sum_{j=0}^{n-1} \intoo[3]{\sum_{k=1}^n \alpha_k^p}^\frac{1}{p} e_{j},\]
so the $p$-sum of the vectors $a^{(j)}$ is a multiple of the constant vector. Let $A_n=\intoo{\sum_{k=1}^n \alpha_k^p}^\frac{1}{p}$. We claim that $A_n$ diverges as $n \rightarrow \infty$. More specifically, let us show that $\alpha_k \approx k^{-\frac{1}{p}}$ uniformly on $k$, so that $A_n\approx \intoo{\sum_{k=1}^n k^{-1}}^\frac{1}{p}$, which diverges. Indeed, let $f(t)=t-t^\frac{1}{p}(t-1)^\frac{1}{p^*} = t^\frac{1}{p}(t^\frac{1}{p^*}-(t-1)^\frac{1}{p^*})$, which is continuous on $[1,\infty)$ and satisfies $f(k)=k^\frac{1}{p}\alpha_k$. We have $0<f(f)\leq t-(t-1)=1$. Moreover, making use of L'Hôpital rule, it is easy to show that
\[\lim_{t\rightarrow \infty}f(t)=\lim_{t\rightarrow \infty}\frac{t^\frac{1}{p^*}-(t-1)^\frac{1}{p^*}}{t^{-\frac{1}{p}}}=\frac{1}{p^*}.\]
Therefore, $f(t)$ is uniformly bounded from below on the interval $[1,\infty)$ by a constant $c>0$, so $c k^{-\frac{1}{p}} \leq \alpha_k \leq k^{-\frac{1}{p}}$ for every $k\geq 1$. \\

To conclude the argument, assume that $X$ is $p$-convex with constant $K>0$, and let $T_n:\ell_{p,\infty}^n \rightarrow X$ be lattice isomorphic embeddings with $M=\sup_n \|T_n\|\|T_n^{-1}\|<\infty$. Since lattice homomorphisms preserve $p$-sums, for every $n$ we have
\begin{align*}
    A_n n^\frac{1}{p} &= \norm[3]{\intoo[3]{\sum_{j=0}^{n-1} |a^{(j)}|^p}^\frac{1}{p}}_{\ell_{p,\infty}^n} \leq \|T_n^{-1}\| \norm[3]{\intoo[3]{\sum_{j=0}^{n-1} |T_na^{(j)}|^p}^\frac{1}{p}}_X \\
   & \leq \|T_n^{-1}\|  K \intoo[3]{\sum_{j=0}^{n-1} \|T_n a^{(j)}\|_X^p}^\frac{1}{p} \leq \|T_n\|\|T_n^{-1}\|  K \intoo[3]{\sum_{j=0}^{n-1} \|a^{(j)}\|_{\ell_{p,\infty}^n}^p}^\frac{1}{p} \leq M K n^\frac{1}{p},
\end{align*}
so $A_n\leq MK$ for every $n$, which is false.
\end{proof}

The following computation is constantly used to construct isometric copies of $\ell_{p,\infty}^n$:

\begin{lem}\label{lem: span of disjoint characteristics is wlp}
    Let $(\Omega,\Sigma,\mu)$ be a measure space, and let $B_1,\ldots,B_n$ be pairwise disjoint measurable subsets of finite measure. Then $\spn\cbr{\chi_{B_i}}_{i=1}^n$ is lattice isomorphic to $\ell_{p,\infty}^n$. In particular, if all the subsets have the same measure, the isomorphism becomes an isometry.
\end{lem}

\begin{proof}
    Let $w_i=\mu(B_i)$ and $g_i=w_i^{-\frac{1}{p}}\chi_{B_i}$, so that $(g_i)_{i=1}^n$ is a normalized basis of $\spn\cbr{\chi_{B_i}}_{i=1}^n$. Given $f=\sum_{i=1}^n a_i g_i$, we have that
    \begin{align*}
        \|f\|_{L_{p,\infty}}&=\sup_{0<\mu(A)<\infty} \mu(A)^{-\frac{1}{p^*}} \sum_{i=1}^n |a_i| \frac{\mu(A\cap B_i)}{w_i^\frac{1}{p}}\\
        &= \sup_{\substack{0<\mu(A)<\infty\\ A\subseteq \bigcup_i B_i}} \intoo[3]{\sum_{i=1}^n\mu(A\cap B_i)}^{-\frac{1}{p^*}} \sum_{i=1}^n |a_i| \frac{\mu(A\cap B_i)}{w_i^\frac{1}{p}}.
    \end{align*}
    
    Next, let us consider the family of functions of the form
    \[F(t)=(t+a)^{-\frac{1}{p^*}}(bt+c)\]
    defined on $[0,d]$, with $a,b,c\geq 0$ (with the additional restriction that if $a=0$, then $c=0$ too) and $d>0$. We claim that all these functions attain their maximum in the interval $[0,d]$ either at $0$ or at $d$. On the one hand, if $a=0=c$, it is clear that $F(t)=bt^\frac{1}{p}$ is increasing, so its maximum is attained at $d$. On the other hand, if $a>0$, then $F$ is differentiable on $[0,d]$, so a standard optimization argument yields that either $F'$ is non-zero in the whole interval $[0,d]$, which means that $F$ is strictly monotone, or $F$ has only a local minimum in $[0,d]$. In both cases, it follows that $F$ attains its maximum at one of the extreme points of the interval.\\

    Now that the claim has been proved, we can simplify the expression for the norm of $f=\sum_{i=1}^n a_i g_i$. More specifically, given any measurable subset $A\subseteq \bigcup_{i=1}^n B_i$ with finite measure, we can find a set of indices $I\subseteq \{1,\ldots,n\}$ such that $\bigcup_{i\in I} B_i$ approximates the norm of $f$ better than $A$. Indeed, fixing the parameters $a=\sum_{i=2}^n\mu(A\cap B_i)$, $b=\frac{|a_1|}{w_1}$, $c= \sum_{i=2}^n |a_i| \frac{\mu(A\cap B_i)}{w_i^\frac{1}{p}}$ and $d=w_1$, it follows that
    \[\intoo[3]{\sum_{i=1}^n\mu(A\cap B_i)}^{-\frac{1}{p^*}} \sum_{i=1}^n |a_i| \frac{\mu(A\cap B_i)}{w_i^\frac{1}{p}} =F(\mu(A\cap B_1))\leq \max \{F(0),F(w_1)\}=F(\theta_1w_1)\]
    for some $\theta_1\in \{0,1\}$. We can inductively apply this argument to obtain a sequence $(\theta_i)_{i=1}^n\subseteq \{0,1\}$ such that
    \begin{align*}
        \intoo[3]{\sum_{i=1}^n\mu(A\cap B_i)}^{-\frac{1}{p^*}} & \sum_{i=1}^n |a_i| \frac{\mu(A\cap B_i)}{w_i^\frac{1}{p}} \\ 
        & \leq \intoo[3]{\theta_1 w_1+\sum_{i=2}^n\mu(A\cap B_i)}^{-\frac{1}{p^*}} \intoo[3]{\theta_1 |a_1|w_1^\frac{1}{p^*}+\sum_{i=2}^n |a_i| \frac{\mu(A\cap B_i)}{w_i^\frac{1}{p}}}\\
        & \leq \intoo[3]{\sum_{i=1}^k\theta_i w_i+\sum_{i=k+1}^n\mu(A\cap B_i)}^{-\frac{1}{p^*}} \intoo[3]{\sum_{i=1}^k\theta_i |a_i|w_i^\frac{1}{p^*}+\sum_{i=k+1}^n |a_i| \frac{\mu(A\cap B_i)}{w_i^\frac{1}{p}}} \\
        & \leq \intoo[3]{\sum_{i=1}^n\theta_i w_i}^{-\frac{1}{p^*}} \intoo[3]{\sum_{i=1}^n\theta_i |a_i|w_i^\frac{1}{p^*}}\\
        & \leq \intoo[3]{\sum_{i\in I}w_i}^{-\frac{1}{p^*}} \intoo[3]{\sum_{i\in I } |a_i|w_i^\frac{1}{p^*}},
    \end{align*}
    where $I=\{i:\theta_i=1\}$. Therefore, we conclude that
    \[\|f\|_{L_{p,\infty}} = \sup_{I\subseteq \{1,\ldots,n\}} \intoo[3]{\sum_{i\in I}w_i}^{-\frac{1}{p^*}} \intoo[3]{\sum_{i\in I } |a_i|w_i^\frac{1}{p^*}},\]
    which is precisely the norm in the weighted finite-dimensional $\ell_{p,\infty}^n(w_i)_{i=1}^n$ of the sequence $\sum_{i=1}^n a_i e_i$, where $e_i=w_i^{-\frac{1}{p}}\chi_{\{i\}}$ is the normalized canonical basis. Therefore, $\spn\cbr{\chi_{B_i}}_{i=1}^n$ is lattice isometric to $\ell_{p,\infty}^n((w_i)_{i=1}^n)$. Finally, if we denote by $w_-=\min_i w_i$ and $w_+=\max_i w_i$, it is easy to see that
    \[\intoo[2]{\frac{w_-}{w_+}}^\frac{1}{p^*}\|(a_i)_{i=1}^n\|_{\ell_{p,\infty}^n} \leq \|f\|_{L_{p,\infty}} \leq \intoo[2]{\frac{w_+}{w_-}}^\frac{1}{p^*}\|(a_i)_{i=1}^n\|_{\ell_{p,\infty}^n}, \]
    so $\spn\cbr{\chi_{B_i}}_{i=1}^n$ is lattice isomorphic to $\ell_{p,\infty}^n$.
\end{proof}

\begin{example}\label{example: not p-convex wlp}
We can fully characterize the weak $L_p$ spaces that are not $p$-convex. In fact, either a weak $L_p$ space is not $p$-convex or it is $\infty$-convex, that is, it has the maximal convexity:
\begin{enumerate}
    \item Let $(\Omega,\Sigma,\mu)$ be a measure space with non-atomic part. Then, $L_{p,\infty}(\mu)$ is not $p$-convex. Indeed, for every $n\in \N$ we can find $n$ measurable subsets of $\Omega$ with the same measure, so, by \Cref{lem: span of disjoint characteristics is wlp}, $L_{p,\infty}(\mu)$ contains a sublattice isometric to $\ell^n_{p,\infty}$. \Cref{prop: uniform copies of fd weak Lp} yields that $L_{p,\infty}(\mu)$ is not $p$-convex.
    \item Let $\Gamma$ be a set and $w: \Gamma\rightarrow (0,\infty)$ be a collection of weights over $\Gamma$. Denote again by $w$ the measure that it induces on $\mathcal{P}(\Gamma)$, and consider the atomic space $\ell_{p,\infty}(\Gamma, w)$. Let $\Gamma_m=\{\gamma\in \Gamma: 2^m\leq w(\gamma)<2^{m+1}\}$. Using the ideas of \cite{Leung96}, we can distinguish two cases:
    \begin{enumerate}[(a)]
        \item If $\sup_m |\Gamma_m|=\infty$, then for every $n\in \N$ there exists $m\in \N$ such that $|\Gamma_m|\geq n$, and hence there exist $n$ distinct points $\gamma_1\ldots, \gamma_n\in\Gamma_m$ such that $\frac{1}{2}\leq \frac{w(\gamma_i)}{w(\gamma_j)}\leq 2$ whenever $i\neq j$. By \Cref{lem: span of disjoint characteristics is wlp}, $\ell_{p,\infty}(\Gamma, w)$ contains uniform lattice isomorphic copies of $\ell^n_{p,\infty}$, so \Cref{prop: uniform copies of fd weak Lp} implies that $\ell_{p,\infty}(\Gamma, w)$ is not $p$-convex.
        \item If $\sup_m |\Gamma_m|<\infty$, let $N=\sup_m |\Gamma_m|$. Observe that, in particular, $\Gamma$ is countable, and can be identified with a subset of $\Z\times \{1,\ldots,N\}$, so that each $\Gamma_m$ corresponds to the subset $\{m\}\times \{1,\ldots,|\Gamma_m|\}$. Then, following the proof of \cite[Theorem 6]{Leung96}, it is easy to see that $\ell_{p,\infty}(\Gamma, w)$ embeds lattice isomorphically into the $\ell_\infty$ sum of $N$ copies of the space $\ell_{p,\infty}(\Z,\sigma)$, where $\sigma(m)=2^m$ for every $m\in \Z$. Since $\ell_{p,\infty}(\Z,\sigma)$ is lattice isomorphic to $\ell_\infty (\Z)$ via the operator
        \[\fullfunction{T}{\ell_{p,\infty}(\Z,\sigma)}{\ell_\infty}{(a_m)_m}{(2^\frac{m}{p}a_m)_m}\]
        (see \cite[Proposition 5]{Leung96}), it follows that $\ell_{p,\infty}(\Gamma, w)$ is lattice isomorphic to a sublattice of $\ell_\infty$, and hence it must be $\infty$-convex. \Cref{thm: p-conv implies q-conv} will then imply that $\ell_{p,\infty}(\Gamma, w)$ is $p$-convex.
    \end{enumerate}
\end{enumerate}
\end{example}

As we will see in the next section, the fact that $L_{p,\infty}(\mu)$ satisfies an upper $p$-estimate with constant one but is not $p$-convex implies that its predual $L_{p^*,1}(\mu)$ satisfies a lower $p^*$-estimate but is not $p^*$-concave.

\section{General facts}\label{sec: general facts}
\subsection{Duality}\label{sec: duality}

As one could expect from \Cref{rem: facts about convexity}, convexity and concavity properties are related through duality. Before stating the results, let us recall that, given a Banach space $E$, an integer $n$ and an exponent $1\leq p\leq \infty$, if we denote by $\ell_p^n(E)$ the space of sequences of $E$ of $n$ elements, endowed with the norm $\|(u_i)_{i=1}^n\|_{\ell_p^n(E)}=\intoo{\sum_{i=1}^n\|u_i\|^p}^\frac{1}{p}$, then its dual space can be isometrically identified with $\ell_{p^*}^n(E^*)$, with the duality action given by
\[\langle (u_i^*)_{i=1}^n,(u_i)_{i=1}^n\rangle=\sum_{i=1}^n u_i^*(u_i).\]

\begin{thm}\label{thm: duality convexity concavity}
    Let $X$ be a Banach lattice, $E$ a Banach space and $1\leq p, q\leq \infty$. Then:
    \begin{enumerate}
        \item $T:E\rightarrow X$ is $(p,q)$-convex if and only if $T^*:X^*\rightarrow E^*$ is $(p^*,q^*)$-concave, and $K_{(p^*,q^*)}(T^*)=K^{(p,q)}(T)$.
        \item $S:X\rightarrow E$ is $(p,q)$-concave if and only if $S^*:E^*\rightarrow X^*$ is $(p^*,q^*)$-convex, and $K^{(p^*q^*)}(S^*)=K_{(p,q)}(S)$.
    \end{enumerate}
\end{thm}

\begin{proof}
    Let us consider first a $(p,q)$-convex operator $T:E\rightarrow X$, and fix $x_1^*,\ldots, x_n^*\in X^*$. Then, for every sequence $(u_i)_{i=1}^n\in S_{\ell_p^n(E)}$, using \Cref{prop: Holder in BL} we have that
    \begin{align*}
        \sum_{i=1}^nT^*x_i^*(u_i)&= \sum_{i=1}^nx_i^*(Tu_i)\leq \intoo[3]{\sum_{i=1}^n |x_i^*|^{q^*}}^\frac{1}{q^*} \intoo[4]{\intoo[3]{\sum_{i=1}^n |Tu_i|^q}^\frac{1}{q}}\\
        & \leq \norm[3]{\intoo[3]{\sum_{i=1}^n |x_i^*|^{q^*}}^\frac{1}{q^*}} \norm[3]{\intoo[3]{\sum_{i=1}^n |Tu_i|^q}^\frac{1}{q}}\\
        & \leq \norm[3]{\intoo[3]{\sum_{i=1}^n |x_i^*|^{q^*}}^\frac{1}{q^*}} K^{(p,q)}(T) \intoo[3]{\sum_{i=1}^n \|u_i\|^p}^\frac{1}{p} = K^{(p,q)}(T) \norm[3]{\intoo[3]{\sum_{i=1}^n |x_i^*|^{q^*}}^\frac{1}{q^*}}.
    \end{align*}
    Taking the supremum over all $(u_i)_{i=1}^n\in S_{\ell_p^n(E)}$ we conclude that $T^*$ is $(p^*,q^*)$-concave with constant $K^{(p,q)}(T)$.\\

    Next, let us consider a $(p,q)$-concave operator $S:X\rightarrow E$ and $u_1^*,\ldots, u_n^*\in E^*$. Then, for any $x\in X_+$, we know from \Cref{cor: RK for p sums} that
    \[\intoo[3]{\sum_{i=1}^n |S^*u_i^*|^{q^*}}^\frac{1}{q^*}(x) =  \sup\cbr{\sum_{i=1}^n S^*u_i^*(x_i): x_1,\ldots,x_n\in X, \intoo[3]{\sum_{i=1}^n |x_i|^{q}}^\frac{1}{q}\leq x}.\]
    Therefore, let us fix $x_1,\ldots,x_n\in X$ such that $\intoo{\sum_{i=1}^n |x_i|^{q}}^\frac{1}{q}\leq x$. Then,
    \begin{align*}
        \sum_{i=1}^n S^*u_i^*(x_i) & = \sum_{i=1}^n u_i^*(Sx_i) \leq \intoo[3]{\sum_{i=1}^n \|u_i^*\|^{p^*}}^\frac{1}{p^*} \intoo[3]{\sum_{i=1}^n \|Sx_i\|^p}^\frac{1}{p}\\
        &\leq \intoo[3]{\sum_{i=1}^n \|u_i^*\|^{p^*}}^\frac{1}{p^*} K_{(p,q)}(S) \norm[3]{\intoo[3]{\sum_{i=1}^n |x_i|^q}^\frac{1}{q}} \\
        &\leq K_{(p,q)}(S) \intoo[3]{\sum_{i=1}^n \|u_i^*\|^{p^*}}^\frac{1}{p^*} \|x\|.
    \end{align*}
    Taking the supremum over all the possible choices of $x_1,\ldots,x_n\in X$ we conclude that $S^*$ is $(p^*,q^*)$-convex with constant $K_{(p,q)}(S)$.\\

    Now, let $T:E\rightarrow X$ and assume that $T^*$ is $(p^*,q^*)$-concave. Then, $T^{**}$ is $(p,q)$-convex with constant $K_{(p^*,q^*)}(T^*)$, and so is $T^{**}J_E$, where $J_E:E\rightarrow E^{**}$ denotes the canonical isometric embedding of $E$ into its bidual. Since $T^{**}J_E=J_XT$ and $J_X$ is a lattice isometry from $X$ into $J_X(X)\subseteq X$, it follows from \Cref{rem: facts about convexity}(7) that $T=J_X^{-1}T^{**}J_E$ is $(p,q)$-convex with constant $K_{(p^*,q^*)}(T^*)$.\\

    Similarly, if $S:X\rightarrow E$ is such that $S^*$ is $(p^*,q^*)$-convex, then $S^{**}$ is $(p,q)$-concave with constant $K^{(p^*q^*)}(S^*)$, and so is $S=J_E^{-1}S^{**}J_X$.
\end{proof}

A similar statement holds for upper and lower $p$-estimates. To prove it, we will need the following lemma.

\begin{lem}\label{lem: orthogonal and disjoint system}
    Let $x_1^*,\ldots,x_n^*\in X^*_+$ be non-zero pairwise disjoint functionals, $x_1,\ldots,x_n\in X_+$ and $\delta>0$. Then, there exist pairwise disjoint $z_i\in[0,x_i]$ such that $x_i^*(z_i)\geq x_i^*(x_i)-\delta$ and $x_j^*(z_i)\leq \delta$ for every $j\neq i$.
\end{lem}

\begin{proof}
    Since $x_i^*\wedge (\sum_{j\neq i}x_j^*)=0$ for every $i$, Riesz--Kantorovich formulas imply that there exists a vector $y_i\in [0,x_i]$ such that $x_i^*(x_i-y_i)+ \sum_{j\neq i}x_j^*(y_i)\leq \frac{\delta}{n}$, so $x_i^*(y_i)\geq x_i^*(x_i)-\frac{\delta}{n}$, and $x_k^*(y_i)\leq \sum_{j\neq i}x_j^*(y_i)\leq \frac{\delta}{n}$ for every $k\neq i$. Put $z_i=(y_i-\sum_{j\neq i}y_j)_+\in [0,x_i]$ for every $i$, so that $y_i-\sum_{j\neq i}y_j \leq z_i\leq y_i$ and thus $x_i^*(z_i)\geq x_i^*(x_i)-\delta$ and $x_k^*(z_i)\leq x_k^*(y_i)\leq \delta$ for every $k\neq i$. Note that if $i\neq k$, then $0\leq z_i\wedge z_k\leq (y_i-y_k)_+\wedge (y_k-y_i)_+=0$, so the vectors $z_1,\ldots,z_n$ are pairwise disjoint.
\end{proof}

\begin{thm}\label{thm: duality upe lpe}
    Let $X$ be a Banach lattice and $1\leq  p\leq \infty$. Then:
    \begin{enumerate}
        \item $X$ satisfies an upper $p$-estimate if and only if $X^*$ satisfies a lower $p^*$-estimate, and $K_{(\downarrow p^*)}(X^*)=K^{(\uparrow p)}(X)$.
        \item $X$ satisfies a lower $p$-estimate if and only if $X^*$ satisfies an upper $p^*$-estimate, and $K^{(\uparrow p^*)}(X^*)=K_{(\downarrow p)}(X)$.
    \end{enumerate}
\end{thm}

\begin{proof}
    We start proving that if $X$ satisfies an upper $p$-estimate, then $X^*$ satisfies a lower $p^*$-estimate. To do so, let $x_1^*,\ldots,x_n^*\in X^*_+$ be disjoint non-zero functionals. Fix $\varepsilon>0$ and choose positive $x_1,\ldots,x_n\in B_X$ such that $x_i^*(x_i)\geq (1-\varepsilon)\|x_i^*\|$ for every $i=1,\ldots,n$. Applying \Cref{lem: orthogonal and disjoint system} with $\delta = \varepsilon \min_{1\leq j\leq n}\|x_j^*\|$ we can find pairwise disjoint $z_i\in [0,x_i]$ such that $x_i^*(z_i)\geq (1-2\varepsilon)\|x_i^*\|$. Let us consider arbitrary positive scalars $(a_i)_{i=1}^n\in B_{\ell_{p}^n}$. Since $X$ satisfies an upper $p$-estimate, it follows that
    \begin{align*}
        \sum_{i=1}^n a_i\|x_i^*\| & \leq \frac{1}{1-2\varepsilon} \sum_{i=1}^n  a_i x_i^*(z_i)\leq \frac{1}{1-2\varepsilon} \sum_{i=1}^n  \intoo[3]{\sum_{j=1}^n  x_j^*}(a_iz_i)\\
        & \leq \frac{1}{1-2\varepsilon}   \norm[3]{\sum_{j=1}^n  x_j^*}\norm[3]{\sum_{i=1}^na_iz_i} \leq  \frac{K^{(\uparrow p)}(X)}{1-2\varepsilon} \norm[3]{\sum_{j=1}^n  x_j^*}.
    \end{align*}
    Recall that
    \[\intoo[3]{\sum_{i=1}^n \|x_i^*\|^{p^*}}^\frac{1}{p^*}=\sup \cbr{\sum_{i=1}^n a_i\|x_i^*\|: a_i\geq 0, (a_i)_{i=1}^n\in B_{\ell_{p}^n}},\]
    so 
    \[\intoo[3]{\sum_{i=1}^n \|x_i^*\|^{p^*}}^\frac{1}{p^*}\leq \frac{K^{(\uparrow p)}(X)}{1-2\varepsilon} \norm[3]{\sum_{j=1}^n  x_j^*}\]
    for every $\varepsilon>0$, and therefore $X^*$ satisfies a lower $p^*$-estimate with $K_{(\downarrow p^*)}(X^*)\leq K^{(\uparrow p)}(X)$.\\

    In particular, if $X^*$ satisfies an upper $p^*$-estimate, it follows that $X^{**}$ satisfies a lower $p$-estimate, and therefore $X$ satisfies a lower $p$-estimate too. Moreover, $K_{(\downarrow p)}(X)\leq K^{(\uparrow p^*)}(X^*)$.\\

    Next, we prove that if $X$ satisfies a lower $p$-estimate, then $X^*$ satisfies an upper $p^*$-estimate. Again, fix disjoint non-zero functionals $x_1^*,\ldots,x_n^*\in X^*_+$ and $\varepsilon>0$, and choose a positive $z\in B_X$ such that $\|\sum_{i=1}^n  x_i^*\|\leq \sum_{i=1}^n x_i^*(z) +\varepsilon$. Using \Cref{lem: orthogonal and disjoint system} with $x_i=z$ for every $i=1,\ldots,n$ and $\delta=\frac{\varepsilon}{n}$, we obtain pairwise disjoint vectors $z_1,\ldots,z_n\in [0,z]$ such that $x_i^*(z_i)\geq x_i^*(z)-\frac{\varepsilon}{n}$. Therefore, the lower $p$-estimate of $X$ yields
    \begin{align*}
        \norm[3]{\sum_{j=1}^n  x_j^*} & \leq \sum_{i=1}^n x_i^*(z) +\varepsilon\leq \sum_{i=1}^n x_i^*(z_i) +2\varepsilon\leq \sum_{i=1}^n \|x_i^*\|\|z_i\| +2\varepsilon\\
        & \leq \intoo[3]{\sum_{i=1}^n \|x_i^*\|^{p^*}}^\frac{1}{p^*} \intoo[3]{\sum_{i=1}^n \|z_i\|^{p}}^\frac{1}{p} + 2\varepsilon\\
        & \leq K_{(\downarrow p)}(X)  \intoo[3]{\sum_{i=1}^n \|x_i^*\|^{p^*}}^\frac{1}{p^*}\norm[3]{\sum_{j=1}^n z_i} + 2\varepsilon\\
        & \leq K_{(\downarrow p)}(X)  \intoo[3]{\sum_{i=1}^n \|x_i^*\|^{p^*}}^\frac{1}{p^*} + 2\varepsilon.
    \end{align*}
    Taking $\varepsilon$ arbitrarily small we recover that $X^*$ satisfies an upper $p^*$-estimate with $K^{(\uparrow p^*)}(X^*)\leq K_{(\downarrow p)}(X)$.\\

    Finally, when $X^*$ satisfies a lower $p^*$-estimate, we get that $X^{**}$ satisfies an upper $p$-estimate, and it passes to $X$. Additionally, $K^{(\uparrow p)}(X)\leq K_{(\downarrow p^*)}(X^*)$.
\end{proof}

\subsection{Basic properties}\label{sec: basic properties}
In this section, we will study the connections between convexity and concavity properties, as well as with other geometrical properties and operator theory notions. The duality tools developed in the previous subsection will allow us to simplify many of the proofs, as we will only have to deal with the simplest case. Take as an example the following result:

\begin{thm}\label{thm: p-conv implies q-conv}
    Let $X$ be a Banach lattice and $E$ a Banach space. Then:
    \begin{enumerate}
        \item If $T:E\rightarrow X$ is $p$-convex for some $1<p\leq \infty$, then $T$ is $q$-convex for every $1\leq q\leq p$. Moreover, $K^{(q)}(T)$ is non-decreasing and continuous on any interval where it is finite.
        \item If $T:X\rightarrow E$ is $p$-concave for some $1\leq p<\infty$, then $T$ is $q$-concave for every $p\leq q\leq \infty$. Moreover, $K_{(q)}(T)$ is non-increasing and continuous on any interval where it is finite.
    \end{enumerate}
\end{thm}

\begin{proof}
    By \Cref{thm: duality convexity concavity}, it suffices to prove $(1)$, since $(2)$ follows by duality.\\

    Let $T:E\rightarrow X$ be an operator, and note that if $K^{(s)}(T)<\infty$, then
    \begin{align*}
        K^{(s)}(T)&=\sup \cbr[3]{ \norm[3]{\intoo[3]{\sum_{i=1}^n |Tx_i|^{s}}^\frac{1}{s}}: n\in \N, (x_i)_{i=1}^n\subseteq E, \sum_{i=1}^n \|x_i\|^s\leq 1}\\
        &=\sup \cbr[3]{ \norm[3]{\intoo[3]{\sum_{i=1}^n a_i|Ty_i|^{s}}^\frac{1}{s}}: n\in \N, (y_i)_{i=1}^n\subseteq E, \|y_i\|=1, a_i\geq 0, \sum_{i=1}^n a_i\leq 1}.
    \end{align*}
    Applying \Cref{lem: geometric mean}(2) for any choice $(y_i)_{i=1}^n\subseteq S_E$ and scalars $a_i\geq 0$ with $\sum_{i=1}^n a_i\leq 1$, and taking the supremum of all such quantities, we conclude that $K^{(s)}(T)\leq K^{(r)}(T)^\theta K^{(q)}(T)^{1-\theta}$ whenever $1\leq r<s<q\leq \infty$ and $\frac{1}{s}=\frac{\theta}{r}+\frac{1-\theta}{q}$, provided $K^{(r)}(T), K^{(q)}(T)<\infty$. In particular, we know that $K^{(1)}(T)=\|T\|$ and $K^{(p)}(T)<\infty$, so $K^{(s)}(T)<\infty$ for every $1\leq s \leq p$. Moreover, since $K^{(1)}(T)=\|T\|\leq K^{(q)}(T)$ for every $1<q\leq p$, we see that $K^{(s)}(T)\leq K^{(q)}(T)$ whenever $s<q$. Finally, we observe that the function $\varphi(\alpha)=\log K^{\intoo[0]{\frac{1}{\alpha}}}(T)$ is convex (and hence, continuous) in $\intcc[1]{\frac{1}{p},1}$, so $K^{(q)}(T)$ is continuous in $[1,p]$.
\end{proof}

Another remarkable fact is that upper $p$-estimates are just an equivalent formulation of $(p,\infty)$-convexity for Banach lattices.

\begin{thm}\label{thm: upe equivalent to mixed convexity}
    Let $X$ be a Banach lattice and $1\leq p\leq \infty$. Then:
    \begin{enumerate}
        \item $X$ satisfies an upper $p$-estimate if and only if $X$ is $(p,\infty)$-convex, and $K^{(p,\infty)}(X)=K^{(\uparrow p)}(X)$.
        \item $X$ satisfies a lower $p$-estimate if and only if $X$ is $(p,1)$-concave, and $K_{(p,1)}(X)=K_{(\downarrow p)}(X)$.
    \end{enumerate}
\end{thm}

\begin{proof}
    We begin by proving $(1)$.\\

    $(\Leftarrow)$ Let $x_1,\ldots,x_n\in X_+$ be disjoint elements. Then $\sum_{i=1}^n |x_i| = \bigvee_{i=1}^n |x_i|$ by \Cref{lem: disjoint p sums}, so the upper $p$-estimates inequality is just a particular instance of the $(p,\infty)$-convexity inequality, and $K^{(\uparrow p)}(X)\leq K^{(p,\infty)}(X)$.\\

    $(\Rightarrow)$ Let us first assume that $X$ is order complete, that is, that every order bounded set has a supremum in $X$. Let $x_1,\ldots,x_n\in X$. We can assume without loss of generality that they are non-negative. Recall that in an order complete Banach lattice every band is a projection band (cf. \cite[Theorem 1.2.9]{MN}), so let us denote by $z_i=\intoo{\bigvee_{j\neq i}x_j - x_i}_+$ and by $P_i=P_{z_i}$ the band projection over $B_{z_i}$, the band generated by $z_i$, for $i=1,\ldots, n$. Let $z=\bigvee_{j=1}^nx_j$, $y_1=(i_X-P_1)z$ and $y_i=P_1\ldots P_{i-1}(i_X-P_i)z$ for $2\leq i\leq n$. Note that 
    \[\bigwedge_{i=1}^n z_i= \bigwedge_{i=1}^n\intoo[3]{x_i\vee \bigvee_{j\neq i}x_j - x_i}=\bigvee_{j=1}^nx_j+\bigwedge_{i=1}^n(- x_i)=0,\]
    so $P_1\ldots P_n=P_{\bigwedge_{i} z_i}=0$. Therefore, we can iteratively check that $\sum_{i=1}^n y_i=z$. Moreover, the vectors $y_i$ are pairwise disjoint. Indeed, let $i<k$, and recall that band projections commute and the ranges of $P_i$ and $i_X-P_i$ are disjoint, so $y_i\in (i_X-P_i)(X)$ and $y_k\in P_i(X)$ must be disjoint. Finally, we want to check that $y_i\in [0, x_i]$ for every $i=1,\ldots,n$. Since $0\leq P_j\leq i_X$ for every $j$, it suffices to show that $(i_X-P_i)z\leq x_i$. Recall that $P_i$ has an explicit expression given by $P_ix=\sup_{m\in \N} x\wedge (mz_i)$ for any $x\geq 0$, so $z-x_i=z_i= z\wedge z_i\leq P_i z$, as we wanted to show. Therefore, we can apply the upper $p$-estimate condition to conclude that 
    \[\norm[3]{\bigvee_{i=1}^nx_i}=\|z\|=\norm[3]{\sum_{i=1}^n y_i}\leq K^{(\uparrow p)}(X) \intoo[3]{\sum_{i=1}^n \|y_i\|^{p}}^\frac{1}{p} \leq K^{(\uparrow p)}(X) \intoo[3]{\sum_{i=1}^n \|x_i\|^{p}}^\frac{1}{p}, \]
    and thus $K^{(p,\infty)}(X)\leq K^{(\uparrow p)}(X)$.\\

    Now, if $X$ is an arbitrary Banach lattice with upper $p$-estimates, then by \Cref{thm: duality upe lpe} $X^{**}$ also has upper $p$-estimates with the same constant. Since $X^{**}$ is order complete, by the previous part it is $(p,\infty)$-convex. As $X$ is a closed sublattice of $X^{**}$, it inherits the $(p,\infty)$-convexity.\\

    Statement $(2)$ now follows automatically by duality (Theorems \ref{thm: duality convexity concavity} and \ref{thm: duality upe lpe}): $X$ satisfies a lower $p$-estimate if and only if $X^*$ satisfies an upper $p^*$-estimate, if and only if $X^*$ is $(p^*,\infty)$-convex, if and only if $X$ is $(p,1)$-concave, and the constant remains the same throughout the process.
\end{proof}

As it was mentioned in the introduction, one of the original motivations for studying $(p,q)$-convex and $(p,q)$-concave operators was their connection with $(p,q)$-summing operators \cite{MaureyExp2425} (see also the approach followed in \cite{DJT}). Let us recall the definition of such a relevant class of operators:

\begin{defn}\label{defn: (pq)-summing operator}
    Let $1\leq p,q<\infty$. An operator $T:E\rightarrow F$ between Banach spaces is \emph{$(p,q)$-summing} if there exists a constant $K>0$ such that 
    \begin{equation}\label{eq: (p,q)-summing operator}
        \intoo[3]{\sum_{i=1}^n \|Tx_i\|^p}^\frac{1}{p} \leq K \sup_{x^*\in B_{E^*}} \intoo[3]{\sum_{i=1}^n |x^*(x_i)|^q}^\frac{1}{q}
    \end{equation}
    for every $x_1,\ldots,x_n\in E$. The best constant $K$ satisfying \eqref{eq: (p,q)-summing operator} is called the \emph{$(p,q)$-summing norm of $T$} and is denoted $\pi_{p,q}(T)$. When $p=q$, we will simply say $T$ is $p$-summing and write $\pi_p(T)$.
\end{defn}

The set $\Pi_{p,q}(E,F)$ of all $(p,q)$-summing operators from $E$ to $F$ endowed with the norm $\pi_{p,q}(\cdot)$ is a Banach space. These spaces play a fundamental role in the theory of operator ideals (cf. \cite{DJT,TJ}) and are strongly connected to Banach lattice theory, as we will see next. Note that given $x_1,\ldots,x_n\in E$ and $1\leq q<\infty$, they induce the maps
\[\fullfunction{S}{\ell_{q^*}^n}{X}{a}{\sum_{i=1}^n a_ix_i} \quad \text{and} \quad \fullfunction{S^*}{X^*}{\ell_q^n}{x^*}{(x^*(x_i))_{i=1}^n}.\]
These maps allow us to represent the right-hand side of the inequality \eqref{eq: (p,q)-summing operator}, called the \textit{weak $q$-summing norm} of the sequence $(x_i)_{i=1}^n$, in an alternative way:
\[\|(x_i)_{i=1}^n\|_{q,w} := \sup_{x^*\in B_{E^*}} \intoo[3]{\sum_{i=1}^n |x^*(x_i)|^q}^\frac{1}{q}= \|S^*\| = \|S\|= \sup_{a\in B_{\ell_{q^*}^n}}\norm[3]{\sum_{i=1}^n a_ix_i}\]
In particular, if $X$ is a Banach lattice, by \Cref{prop: p-sum as supremum} we have that for every $x_1,\ldots,x_n\in X$
\begin{equation}\label{eq: weak summing and lattice multinorms}
    \|(x_i)_{i=1}^n\|_{q,w}= \sup_{a\in B_{\ell_{q^*}^n}}\norm[3]{\sum_{i=1}^n a_ix_i}\leq \norm[3]{\intoo[3]{\sum_{i=1}^n |x_i|^q}^\frac{1}{q}}.
\end{equation}
Hence, every $(p,q)$-summing operator $T:X\rightarrow E$ from a Banach lattice into a Banach space is $(p,q)$-concave, with $K_{(p,q)}(T)\leq \pi_{p,q}(T)$. The converse is true when $X$ is a $C(K)$ space:

\begin{prop}\label{prop: weak p-summing norm in C(K)}
    Let $K$ be a compact Hausdorff space and $f_1,\ldots,f_n\in C(K)$. Then
    \[\norm[3]{\intoo[3]{\sum_{i=1}^n |f_i|^q}^\frac{1}{q}}=\|(f_i)_{i=1}^n\|_{q,w}\]
\end{prop}

\begin{proof}
    It suffices to prove that the left-hand side term is smaller than the right-hand side term. Since the evaluation functionals $\delta_t$ in $C(K)^*$, with $t\in K$, are normalized lattice homomorphism, they preserve functional calculus, so
    \begin{align*}
        \norm[3]{\intoo[3]{\sum_{i=1}^n |f_i|^q}^\frac{1}{q}} & =\sup_{t\in K} \intoo[3]{\sum_{i=1}^n |f_i|^q}^\frac{1}{q}(t) = \sup_{t\in K} \delta_t \intoo[3]{\intoo[3]{\sum_{i=1}^n |f_i|^q}^\frac{1}{q}}\\
        & = \sup_{t\in K} \intoo[3]{\sum_{i=1}^n |\delta_t (f_i)|^q}^\frac{1}{q} \leq \sup_{\mu\in B_{C(K)^*}} \intoo[3]{\sum_{i=1}^n |\mu (f_i)|^q}^\frac{1}{q} =\|(f_i)_{i=1}^n\|_{q,w}. \qedhere
    \end{align*}
\end{proof}

Actually, $(p,q)$-summing operators defined on $C(K)$ spaces can be used to provide a characterization of $(p,q)$-concave operators. 

\begin{thm}\label{thm: characterization (pq)-concave and (pq)-summing}
     Let $X$ be a Banach lattice, $E$ a Banach space, $1\leq q\leq p< \infty$ and $M>0$. Then $T:X\rightarrow E$ is $(p,q)$-concave with $K_{(p,q)}(T)\leq M$ if and only if for every compact Hausdorff space $K$ and every positive operator $S:C(K)\rightarrow X$, $TS:C(K)\rightarrow E$ is $(p,q)$-summing with $\pi_{p,q}(TS)\leq M\|S\|$.
\end{thm}

\begin{proof}
    $(\Rightarrow)$ Let $K$ be a compact Hausdorff space and $S:C(K)\rightarrow X$ be a positive operator. Then, using that, by \Cref{rem: facts about convexity}(7), $TS$ is also $(p,q)$-concave with constant $M\|S\|$ of $T$, so using \Cref{prop: weak p-summing norm in C(K)} it follows that for every $f_1,\ldots,f_n \in C(K)$ we have that
    \[\intoo[3]{\sum_{i=1}^n \|TSf_i\|^p}^\frac{1}{p}  \leq  M\|S\|  \norm[3]{\intoo[3]{\sum_{i=1}^n |f_i|^q}^\frac{1}{q}} = M\|S\| \|(f_i)_{i=1}^n\|_{q,w},\]
    and hence $TS$ is $(p,q)$-summing.\\

    $(\Leftarrow)$ Let $x_1,\ldots,x_n\in X$, and assume without loss of generality that $e=\intoo{\sum_{i=1}^n |x_i|^q}^\frac{1}{q}$ has norm one. We identify $I_e$ with some $C(K)$ space using Kakutani's Representation Theorem for $AM$-spaces, and let $j:C(K)\rightarrow I_e\subseteq X$ be the corresponding lattice injection, which has norm one, since $\|jf\|\leq \|e\| \|jf\|_e=\|f\|_\infty$ for every $f\in C(K)$, and satisfies $j \uno =e$, where $\uno$ denotes the constant one function. By the assumption, $Tj$ is $(p,q)$-summing, so if $f_i\in C(K)$ are such that $x_i=jf_i$, it follows that
    \[\intoo[3]{\sum_{i=1}^n \|Tx_i\|^p}^\frac{1}{p} =\intoo[3]{\sum_{i=1}^n \|Tjf_i\|^p}^\frac{1}{p} \leq \pi_{p,q}(Tj) \norm[3]{\intoo[3]{\sum_{i=1}^n |f_i|^q}^\frac{1}{q}}_\infty  = \pi_{p,q}(Tj) \|\uno\|_\infty \leq M\|j\| =M.\]
    We conclude that $T$ is $(p,q)$-concave.
\end{proof}

This connection between $(p,q)$-concave operators and $(p,q)$-summing operators on $C(K)$-spaces can be exploited to obtain that an operator is $(p,1)$-concave if and only if it is $(p,q)$-concave for every $1\leq q<p$, or that every $(p,1)$-concave operator is $q$-concave for every $p<q\leq \infty$ (see \Cref{thm: relevant cases of concavity}), and the corresponding dual statements for convex operators (\Cref{cor: relevant cases of convexity}). However, we will postpone the proof of these facts, as it requires a generalization of Pietsch's Factorization theorem to the setting of $(p,q)$-summing operators due to Pisier \cite{Pisier86}, which fits better in the exposition of \Cref{sec: factorization function spaces}.\\

Finally, we conclude this section by stating the connection of concavity with other relevant properties of Banach lattices that reflect the behavior of the norm with respect to the order. Recall that a Banach lattice $X$ is said \textit{order continuous} if every downward directed set $A\subseteq X_+$ whose infimum is $0$ satisfies that $\inf\{\|x\|:x\in A\}=0$, and it is called a \textit{KB-space} if every monotone sequence in $B_X$ is convergent. These two properties have multiple equivalent characterizations (see for instance \cite[Section 2.4]{MN}).

\begin{lem}\label{lem: lpe implies KB space}
    Let $X$ be a Banach lattice with a lower $p$-estimate for some $1\leq p<\infty$. Then $X$ is a KB-space. In particular, $X$ is order continuous.
\end{lem}

\begin{proof}
    By \cite[Theorem 2.4.12]{MN}, $X$ is a KB-space if and only if $X$ does not contain any sublattice lattice isomorphic to $c_0$. Assume there exists a positive disjoint sequence $(x_k)_{k=1}^\infty$ that is equivalent to the canonical basis of $c_0$. Then, there exists a constant $C>0$ such that for every sequence $(a_k)_{k=1}^\infty\in c_0$,
    \[C^{-1}\sup_k |a_k|\leq \norm[3]{\sum_{k=1}^\infty a_kx_k}\leq C\sup_k |a_k|.\]
    In particular, the lower $p$-estimate of $X$ yields that 
    \[C^{-1} n^\frac{1}{p}\leq  \intoo[3]{\sum_{k=1}^n \|x_k\|^p}^\frac{1}{p}\leq K_{(\downarrow p)}(X)\norm[3]{\sum_{k=1}^n x_k}\leq K_{(\downarrow p)}(X)C\]
    for every $n\in \N$, which leads to a contradiction.
\end{proof}

We can summarize the main results of this section (together with \Cref{thm: relevant cases of concavity} and \Cref{cor: relevant cases of convexity}) in the following diagrams. If $X$ is a Banach lattice and $1\leq q<p<r<\infty$, we have this chain of properties:
\[p\text{-convexity }\Rightarrow \text{ upper }p\text{-estimates } \Leftrightarrow (p,\infty)\text{-convexity } \Leftrightarrow  (p,r)\text{-convexity }   \Rightarrow  q\text{-convexity.}\]
On the other hand, for $1< r<p<q\leq\infty$, we have:
\[p\text{-concavity }\Rightarrow \text{ lower }p\text{-estimates } \Leftrightarrow (p,1)\text{-concavity } \Leftrightarrow  (p,r)\text{-concavity }   \Rightarrow  q\text{-concavity.}\]


\section{The \texorpdfstring{$p$}{}-convexification and \texorpdfstring{$p$}{}-concavification of a Banach lattice}\label{sec: convexification and concavification}
The aim of this section is to describe the processes of $p$-convexification and $p$-concavification of a Banach lattice, expanding the details given in \cite[Section 1.d]{LT2}. The $p$-convexification and $p$-concavification techniques are an abstract description of the process that transforms $L_r(\mu)$ into $L_{\sigma r}(\mu)$ for any $0< r<\infty$ and $\sigma >0$ by means of the map $f\mapsto \sgn (f) |f|^\frac{1}{\sigma}$. In an arbitrary Banach lattice $X$, the expression $|x|^\frac{1}{\sigma}$ is not necessarily well defined, as we discussed in \Cref{sec: intro}. Therefore, we need to modify the algebraic operations of our Banach lattice using functional calculus to generalize the map above.\\

Given $0<\sigma<\infty$, we define the function $G_\sigma:\R\rightarrow \R$ by $G_\sigma(t):=\sgn (t) |t|^\sigma$, and $F_\sigma :\R^2\rightarrow \R$ by $F_\sigma (t_1,t_2):= G_\frac{1}{\sigma}(G_\sigma(t_1)+G_\sigma(t_2))$. Note that $G_\sigma$ is a continuous and increasing function such that $G_\frac{1}{\sigma}(G_\sigma(t))=t$ and $G_\sigma(t_1t_2)=G_\sigma(t_1)G_\sigma(t_2)$, and thus $F_\sigma$ is continuous and homogeneous, i.e., $F_\sigma(\alpha t_1,\alpha t_2)=\alpha F_\sigma(t_1,t_2)$ for every $\alpha \in \R$. In particular, $F_\sigma$ is positively homogeneous. Sometimes we will denote $G_\sigma(t)=t^\sigma$ and $F_\sigma(t_1,t_2)=(t_1^\sigma+t_2^\sigma)^\frac{1}{\sigma}$ for the sake of simplicity (note, for instance, that with this notation $t^2$ can take negative values, in contrast with $|t|^2$, which coincides with the usual notion of ``square of a number''). We observe that $F_\sigma$ satisfies all the properties of an abelian group operation:
\begin{enumerate}
    \item Associativity: $F_\sigma(F_\sigma(t_1,t_2),t_3)= F_\sigma(t_1,F_\sigma(t_2,t_3))$,
    \item Commutativity: $F_\sigma(t_1,t_2)=F_\sigma(t_2,t_1)$
    \item Zero element: $F_\sigma(t,0)=t$,
    \item Opposite element: $F_\sigma(t,-t)=0$.
\end{enumerate}

Given a Banach lattice $X$, with the sum $+$, the product $\cdot$, the norm $\|\cdot\|$, the partial order $\leq$ and the lattice operations $\vee$ and $\wedge$, we define new algebraic operations on $X$ by putting $x\oplus_\sigma y:=F_\sigma(x,y) = (x^\sigma+y^\sigma)^\frac{1}{\sigma}$ and $\alpha \odot_\sigma x:= G_\frac{1}{\sigma}(\alpha)\cdot x = \alpha^\frac{1}{\sigma} \cdot x$, for $x,y\in X$ and $\alpha \in \R$. \Cref{thm: functional calculus} warrants that these operations are well defined on $X$. 

\begin{prop}\label{prop: new operations on vector lattices}
    The set $X$, endowed with the sum $\oplus_\sigma$, the product $\odot_\sigma$ and the partial order $\leq$, is a vector lattice, with the lattice operations given by $\vee$ and $\wedge$.
\end{prop}

We will sometimes denote by $X_\sigma$ the vector lattice $(X,\oplus_\sigma,\odot_\sigma,\leq)$.

\begin{proof}
    First, we observe that the properties of $F_\sigma$ described above, together with the fact that $G_\frac{1}{\sigma}(-1)=-1$, imply that $\oplus_\sigma$ is an associative and commutative operation on $X$ that has a zero element given by the zero element of $(X,+,\cdot)$, $0\in X$, and a opposite element for every $x\in X$ given by the opposite element of $(X,+,\cdot)$, $-x=(-1)\cdot x=(-1)\odot_\sigma x$. Moreover, $F_\sigma$ is homogeneous, so in particular, if we fix $\alpha \in R$, we have
    \[F_\sigma(G_\frac{1}{\sigma}(\alpha) t_1,G_\frac{1}{\sigma}(\alpha) t_2)= G_\frac{1}{\sigma}(\alpha)F_\sigma(t_1,t_2)\]
    for every $t_1,t_2 \in R$. Therefore, for every $x,y\in X$
    \[(\alpha\odot_\sigma x)\oplus_\sigma (\alpha\odot_\sigma y)=F_\sigma(G_\frac{1}{\sigma}(\alpha) \cdot x,G_\frac{1}{\sigma}(\alpha) \cdot y)= G_\frac{1}{\sigma}(\alpha)\cdot F_\sigma(x,y)=\alpha \odot_\sigma (x\oplus_\sigma y),\]
    so the product $\odot_\sigma$ is distributive and $(X,\oplus_\sigma, \odot_\sigma)$ is a vector space. Next, let us check that the order $\leq$ is compatible with the linear structure. Proceeding as in the proof of \Cref{prop: monotonicity of p-sums}, we can show that given $x,y,z\in X$ with $x\leq y$, we have $x\oplus_\sigma z\leq y\oplus_\sigma z$. Moreover, if $x,y\in X$ are such that $x\leq y$ and $\alpha \in [0,\infty)$, then
    \[\alpha \odot_\sigma x=|\alpha|^\frac{1}{\sigma}\cdot x\leq |\alpha|^\frac{1}{\sigma}\cdot y=\alpha \odot_\sigma y.\]
    Finally, the order $\leq$ is a lattice order, with lattice operations $\vee$ and $\wedge$, so $(X,\oplus_\sigma, \odot_\sigma,\leq)$ is a vector lattice.
\end{proof}

Next, we will determine when the vector lattice $X_\sigma$ can be endowed with a complete lattice norm, so that it becomes a Banach lattice. To do so, we will need the following inequalities:

\begin{lem}\label{lem: inequalities of new sums}
    Let $1<p<\infty$. Then
    \begin{enumerate}
        \item There is a constant $C_p>0$ such that
        \[|F_\frac{1}{p}(t_1,-t_2)|=|t_1^\frac{1}{p}-t_2^\frac{1}{p}|^p\leq C_p |t_1-t_2|\]
        for every $t_1,t_2\in \R$.
        \item There is a constant $D_p>0$ such that
        \[|F_p(t_1,-t_2)|=|t_1^p-t_2^p|^\frac{1}{p}\leq D_p (|t_1|\vee |t_2|)^\frac{1}{p^*} |t_1-t_2|^\frac{1}{p}\]
        for every $t_1,t_2\in \R$.
    \end{enumerate}
\end{lem}

\begin{proof}
    $(1)$ Without loss of generality, we can suppose that $t_2<t_1$. Assume first that $0\leq t_2 < t_1$. If $t_2<\frac{t_1}{2}$, then
    \[0\leq t_1^\frac{1}{p}-t_2^\frac{1}{p}\leq t_1^\frac{1}{p}=2^\frac{1}{p}\intoo[2]{\frac{t_1}{2}}^\frac{1}{p}\leq 2^\frac{1}{p}|t_1-t_2|^\frac{1}{p}.\]
    On the other hand, if $t_2\geq\frac{t_1}{2}$, by the Mean Value Theorem there is some $t_2<s<t_1$ such that
    \[|t_1^\frac{1}{p}-t_2^\frac{1}{p}|=\frac{1}{p}s^{-\frac{1}{p^*}}|t_1-t_2|=\frac{1}{p}\intoo[2]{\frac{|t_1-t_2|}{s}}^{\frac{1}{p^*}}|t_1-t_2|^\frac{1}{p}\leq \frac{1}{p}|t_1-t_2|^\frac{1}{p}.\]
    If $t_2<t_1\leq 0$, the situation is similar. Finally, if $t_2<0<t_1$, then
    \[|t_1^\frac{1}{p}-t_2^\frac{1}{p}|=t_1^\frac{1}{p}+|t_2|^\frac{1}{p}\leq 2^\frac{1}{p^*}(t_1+|t_2|)^\frac{1}{p}=2^\frac{1}{p^*}|t_1-t_2|^\frac{1}{p}.\]

    $(2)$ Again, we can assume that $t_2<t_1$. Since $G_p$ is differentiable in $\R$, by the Mean Value Theorem there exists $t_2<s<t_1$ such that
    \[|t_1^p-t_2^p|^\frac{1}{p}=p^\frac{1}{p} |s|^\frac{p-1}{p} |t_1-t_2|^\frac{1}{p}\leq p^\frac{1}{p} (|t_1|\vee |t_2|)^\frac{1}{p^*} |t_1-t_2|^\frac{1}{p}.\qedhere\]
\end{proof}

Let us fix $1<p<\infty$, and define for every $x\in X$ the function $\normiii{x}=\|x\|^\frac{1}{p}$.

\begin{prop}\label{prop: p-convexification}
    The vector lattice $X_\frac{1}{p}=(X,\oplus_\frac{1}{p}, \odot_\frac{1}{p},\leq)$, endowed with $\normiii{\cdot}$, is a Banach lattice, denoted $X^{(p)}$.
\end{prop}

\begin{proof}
    We first prove that $\normiii{\cdot}$ is a norm on $X_\frac{1}{p}$. Indeed, given $x\in X$ and $\alpha\in \R$, it is clear that $\normiii{x}=0$ if and only if $x=0$, and
    \[\normiii{\alpha \odot_\frac{1}{p} x}=\|\alpha^p \cdot x\|^\frac{1}{p}=|\alpha|\|x\|^\frac{1}{p}=|\alpha|\normiii{x}.\]
    To prove the triangular inequality, note that for every $t_1,t_2\in \R$ and $\alpha,\beta\in (0,\infty)$ with $\alpha^{p^*}+\beta^{p^*}=1$, we have
    \[|F_\frac{1}{p}(t_1,t_2)|\leq (|t_1|^\frac{1}{p}+|t_2|^\frac{1}{p})^p\leq \frac{|t_1|}{\alpha^p}+\frac{|t_2|}{\beta^p}, \]
    therefore the same inequality holds when we replace $t_1,t_2\in \R$ by $x,y\in X$. In particular, taking 
    \[\alpha^p=\frac{\|x\|^\frac{1}{p^*}}{(\|x\|^\frac{1}{p}+\|y\|^\frac{1}{p})^\frac{p}{p^*}}\quad \text{and} \quad \beta^p=\frac{\|y\|^\frac{1}{p^*}}{(\|x\|^\frac{1}{p}+\|y\|^\frac{1}{p})^\frac{p}{p^*}},\]
    we obtain
    \[\normiii{x\oplus_\frac{1}{p} y}=\| F_\frac{1}{p}(x,y)\|^\frac{1}{p}\leq \norm[3]{\frac{|x|}{\alpha^p}+\frac{|y|}{\beta^p}}^\frac{1}{p} \leq \intoo[3]{\frac{\|x\|}{\alpha^p}+\frac{\|y\|}{\beta^p}}^\frac{1}{p}= \normiii{x}+\normiii{y}.\]
    Moreover, $\normiii{\cdot}$ is a lattice norm. Indeed, if $|x|\leq |y|$, then $\normiii{x}=\|x\|^\frac{1}{p}\leq \|y\|^\frac{1}{p}=\normiii{y}$. Therefore, $X^{(p)}=(X_\frac{1}{p},\normiii{\cdot})$ is a normed vector lattice. \\

    Finally, let us check that $\normiii{\cdot}$ is a complete norm. Let $(x_n)_{n=1}^\infty\subseteq X$ be a Cauchy sequence in $X^{(p)}$. In particular, it is bounded, so there is some $M>0$ such that $\normiii{x_n}=\|x_n\|^\frac{1}{p}\leq M$ for every $n\in \N$. Our aim is to show that $(x_n)_{n=1}^\infty$ converges to some $x\in X$ with respect to the norm $\normiii{\cdot}$. To do so, we first prove that $(x_n)_{n=1}^\infty$ is a Cauchy sequence in $X$ (with its usual norm). Observe that, by \Cref{lem: inequalities of new sums}(2), for every $t_1,t_2\in \R$ we have
    \[|t_1-t_2|=|(t_1^\frac{1}{p})^p-(t_2^\frac{1}{p})^p|\leq D_p^p(|t_1|\vee |t_2|)^\frac{1}{p^*}|(t_1^\frac{1}{p}-t_2^\frac{1}{p})^p|^\frac{1}{p}.\]
    In particular, using \Cref{lem: geometric mean}(1), we get that for any $n,m\in \N$
    \[\|x_n-x_m\|\leq D_p^p \||x_n|\vee |x_m|\|^\frac{1}{p^*}\|x_n\ominus_\frac{1}{p} x_m\|^\frac{1}{p}\leq D_p^p (2M^p)^\frac{1}{p^*} \normiii{x_n\ominus_\frac{1}{p} x_m},\]
    so $(x_n)_{n=1}^\infty$ is a Cauchy sequence in $X$ for the norm $\|\cdot\|$. Since $X$ is a Banach lattice, there exists a vector $x\in X$ such that $\|x-x_n\|\rightarrow 0$ as $n \rightarrow\infty$. By \cref{lem: inequalities of new sums}(1), it follows that
    \[|x\ominus_\frac{1}{p}x_n|=|F_\frac{1}{p}(x,-x_n)|\leq C_p |x-x_n|,\]
    so
    \[\normiii{x\ominus_\frac{1}{p}x_n}=\|x\ominus_\frac{1}{p}x_n\|^\frac{1}{p}\leq C_p^\frac{1}{p} \|x-x_n\|^\frac{1}{p} \xrightarrow[n\rightarrow \infty]{}0\]
    and $X^{(p)}$ is complete.
\end{proof}

The Banach lattice $X^{(p)}=(X_\frac{1}{p},\normiii{\cdot})$ is known as the \textit{$p$-convexification} of $X$. Providing a norm for $X_\sigma$ when $\sigma=p>1$ is slightly more involved. Indeed, we need the Banach lattice $X$ to be $p$-convex.

\begin{prop}\label{prop: p-concavification}
    Let $1<p<\infty$ and $X$ be a $p$-convex Banach lattice, with norm $\|\cdot\|$. Then, the vector lattice $X_p=(X,\oplus_p, \odot_p,\leq)$ can be endowed with a complete lattice norm $\normiii{\cdot}$ such that $\normiii{\cdot}\leq \|\cdot\|^p\leq K^{(p)}(X)^p\normiii{\cdot}$.
\end{prop}

The Banach lattice $X_{(p)}=(X_p,\normiii{\cdot})$ is known as the \textit{$p$-concavification} of $X$.

\begin{proof}
    Given $x\in X$, let us define $\normiii{x}_0=\|x\|^p$ and
    \[\normiii{x}=\inf \cbr[3]{\sum_{i=1}^n \normiii{x_i}_0: n\in \N, (x_i)_{i=1}^n\subseteq X,|x|=\bigoplus_{i=1}^n |x_i| },\]
    where $\bigoplus_{i=1}^n |x_i|= |x_1|\oplus_p\ldots\oplus_p |x_n|$. It is clear that $\normiii{x}\leq \normiii{x}_0$. Moreover, the $p$-convexity of $X$ yields that for every $(x_i)_{i=1}^n\subseteq X$ such that $|x|=\bigoplus_{i=1}^n |x_i|$,
    \[\normiii{x}_0=\|x\|^p \leq K^{(p)}(X)^p \sum_{i=1}^n \|x_i\|^p,\]
    so taking the infimum over all possible decompositions we get that $\normiii{x}_0\leq K^{(p)}(X)^p\normiii{x}$. Let us check that $\normiii{\cdot}$ is a norm. Clearly, $\normiii{x}=0$ if and only if $\|x\|=0$, if and only if $x=0$. Moreover, for every $\alpha\in \R$ and $x\in X$,
    \[\normiii{\alpha \odot_p}_0=\|\alpha^\frac{1}{p}\cdot x\|^p=|\alpha|\|x\|^p=|\alpha|\normiii{x}_0,\]
    and from the definition of $\normiii{\cdot}$ we see that $\normiii{\alpha \odot_p}=|\alpha|\normiii{x}$. Before proving the triangular inequality, let us check that if $|x|\leq |y|$, then $\normiii{x}\leq \normiii{y}$. Indeed, if $|y|=\bigoplus_{i=1}^n |y_i|$ is any decomposition, then, since $X_p$ is a vector lattice, it has the Riesz decomposition property, so there exist $(x_i)_{i=1}^n\subseteq X$ such that $0\leq x_i\leq |y_i|$ and $|x|=\bigoplus_{i=1}^n x_i$. In particular, since $\|\cdot\|$, and thus $\normiii{\cdot}_0$, is monotone, we get
    \[\normiii{x}\leq \sum_{i=1}^n \normiii{x_i}_0\leq \sum_{i=1}^n \normiii{y_i}_0,\]
    so taking the infimum over all decompositions of $|y|$ we obtain the desired inequality. Now, in order to prove the triangular inequality, note that for every $t_1,t_2\in \R$,
    \[|F_p(t_1,t_2)|\leq F_p(|t_1|,|t_2|),\]
    and for any $s_1,\ldots,s_m,t_1,\ldots,t_n\in\R$,
    \[F_p(F_p(s_1,\ldots,s_m),F_p(t_1,\ldots,t_n))= F_p(s_1,\ldots,s_m,t_1,\ldots,t_n),\]
    where $F_p(r_1,\ldots,r_k)=G_\frac{1}{p}\intoo{\sum_{i=1}^k G_p(r_i)}$ is defined for any number $k\in \N$ of arguments. In particular, for any $x, y\in X$ and any decompositions $|x|=\bigoplus_{i=1}^m |x_i|$ and $|y|=\bigoplus_{j=1}^n |y_j|$ we have
    \begin{align*}
        |x\oplus_p y| &\leq |x|\oplus_p |y| =\intoo[3]{\bigoplus_{i=1}^m |x_i|}\oplus_p \intoo[3]{\bigoplus_{j=1}^n |y_j|}\\
        & =|x_1|\oplus_p\ldots \oplus_p |x_m|\oplus_p |y_1|\oplus_p\ldots\oplus_p |y_n|,
    \end{align*}
    so
    \[\normiii{x\oplus_p y}\leq \sum_{i=1}^m \normiii{x_i}_0+\sum_{j=1}^n \normiii{y_j}_0.\]
    Taking the infimum over all possible decompositions for $x$ and $y$ we obtain the triangular inequality. Therefore, $X_{(p)}=(X_p,\normiii{\cdot})$ is a normed vector lattice. \\
    
    Finally, let us check that the norm $\normiii{\cdot}$ is complete. As in the proof of \Cref{prop: p-convexification}, let $(x_n)_{n=1}^\infty\subseteq X$ be a Cauchy sequence for the norm $\normiii{\cdot}$. We will show that $(x_n)_{n=1}^\infty$ converges in $X_{(p)}$ to some $x\in X$. First, note that applying \Cref{lem: inequalities of new sums}(1) we get that for every $t_1,t_2\in \R$,
    \[|t_1-t_2|=|(t_1^p)^\frac{1}{p}-(t_2^p)^\frac{1}{p}|\leq C_p^\frac{1}{p}|t_1^p-t_2^p|^\frac{1}{p},\]
    so
    \[\|x_n-x_m\|\leq C_p^\frac{1}{p}\|x_n\ominus_p x_m\|\]
    for every $n,m\in \N$. Thus, $(x_n)_{n=1}^\infty$ is a Cauchy sequence in $X$ (with its usual norm). In particular, it is convergent and bounded. Let $x\in X$ be the limit of the sequence $(x_n)_{n=1}^\infty$, and let $M>0$ be such that $\|x\|,\|x_n\|\leq M$ for every $n\in \N$. Then, by \Cref{lem: inequalities of new sums}(2) and \Cref{lem: geometric mean} we get that
    \begin{align*}
        \normiii{x\ominus_p x_n}& \leq \normiii{x\ominus_p x_n}_0=\|x\ominus_p x_n\|^p \leq D_p^p \| (|x|\vee|x_n|)^\frac{1}{p^*}|x-x_n|^\frac{1}{p}\|^p \\
        & \leq D_p^p \| |x|\vee|x_n|\|^\frac{p}{p^*}\|x-x_n\| \leq  D_p^p (2M)^\frac{p}{p^*}\|x-x_n\| \xrightarrow[n\rightarrow \infty]{} 0,
    \end{align*}
    as we wanted to show.
\end{proof}

As the names indicate, the $p$-convexification and $p$-concavification of a Banach lattice are interesting because of their convexity and concavity properties. Therefore, we need to understand how the functional calculus operates in the vector lattices $X_\sigma$. To do so, we first study the case $X=\Hcal^n$:

\begin{prop}\label{prop: functional calculus in Hcal}
    Let $n\in \N$, $S=S_{B_{\ell_\infty}^n}$ and consider $\Hcal^n$ endowed with the norm $\|f\|=\|\eval[0]{f}_S\|_\infty$ for every $f\in \Hcal^n$. Then, for any $0<\sigma<\infty$:
    \begin{enumerate}
        \item The expression $\normiii{\cdot}=\|\cdot\|^\sigma$ defines a (complete) lattice norm in $\Hcal^n_\sigma$. In particular, for every $\overline{f}=(f_i)_{i=1}^n\subseteq \Hcal^n_\sigma<$, $\Hcal^n_\sigma$ admits a functional calculus $U_{\overline{f}}^\sigma:\Hcal\rightarrow \Hcal^n_\sigma$ associated to $\overline{f}$.
        \item If $\phi_\sigma:\R^n\rightarrow \R^n$ is given by $\phi_\sigma(t_1,\ldots,t_n)=(G_\sigma(t_i))_{i=1}^n=(t_i^\sigma)_{i=1}^n$, then the operator
        \[\fullfunction{U}{\Hcal^n}{\Hcal^n_\sigma}{f}{Uf=G_\frac{1}{\sigma}\circ f\circ \phi_\sigma}\]
        is an isometric lattice homomorphism such that $Up_i=p_i$ for every $i=1,\ldots,n$, where $p_i\in \Hcal^n$ denotes the projection onto the $i$-th coordinate. In particular, if $\overline{p}=(p_i)_{i=1}^n$, then $U_{\overline{p}}^\sigma=U$.
    \end{enumerate}
\end{prop}

\begin{proof}
    $(1)$ If $\sigma=1$, there is nothing to prove, as $X_1$ coincides with $X$ as a vector lattice. If $0<\sigma<1$, by \Cref{prop: p-convexification} $\normiii{\cdot}$ is the norm of the $\frac{1}{\sigma}$-convexification of $\Hcal^n$. If $1<\sigma<\infty$, note that $\Hcal^n$ is $\infty$-convex with constant one, so by \Cref{thm: p-conv implies q-conv} $\Hcal^n$ is $\sigma$-convex with constant one. In particular, by \Cref{prop: p-concavification} $\normiii{\cdot}$ is the norm of the $\sigma$-concavification of $\Hcal^n$. In every case, $\normiii{\cdot}$ is a complete lattice norm on $X_\sigma$. Therefore, \Cref{thm: functional calculus} provides a functional calculus operator $U_{\overline{f}}^\sigma:\Hcal\rightarrow \Hcal^n_\sigma$ for every finite sequence $\overline{f}$. \\

    $(2)$ Note that $\phi_\sigma$ is continuous and bijective, with inverse given by $\phi_\frac{1}{\sigma}$, which is also continuous. Moreover, $\phi_\sigma(S)=S$: if $(t_i)_{i=1}^n\in S$, there exists some index $j$ such that $|t_i|\leq 1=|t_j|$ for every $i=1,\ldots,n$, so $|t_i^\sigma|\leq 1=|t_j^\sigma|$. Hence, $\phi_\sigma(t_1,\ldots,t_n)\in S$. Lastly, for every $\alpha \geq 0$, $\phi_\sigma(\alpha t_1,\ldots,\alpha t_n)=\alpha^\sigma \phi_\sigma(t_1,\ldots,t_n)$. Therefore, for every $f\in \Hcal^n$ the function $Uf$ is continuous and positively homogeneous, so $U$ is well defined. Next, we check that $U$ is linear:
    \[U(\alpha \cdot f)=G_\frac{1}{\sigma}((\alpha \cdot f)\circ \phi_\sigma)= G_\frac{1}{\sigma}(\alpha \cdot (f\circ \phi_\sigma))=G_\frac{1}{\sigma}(\alpha)\cdot G_\frac{1}{\sigma}(f\circ \phi_\sigma)= \alpha \odot_\sigma Uf\]
    and
    \[U(f+g)=G_\frac{1}{\sigma}((f+g)\circ \phi_\sigma)= G_\frac{1}{\sigma}(G_\sigma\circ G_\frac{1}{\sigma}\circ f\circ \phi_\sigma +G_\sigma\circ G_\frac{1}{\sigma}\circ g\circ \phi_\sigma)=(Uf)\oplus_\sigma (Ug).\]
    Clearly,
    \[|Uf|=|f\circ\phi_\sigma|^\frac{1}{\sigma}=|f|^\frac{1}{\sigma}\circ\phi_\sigma = U|f|\]
    and, since $S=\phi_\frac{1}{\sigma}(S)$,
    \[\normiii{Uf}=\left(\sup_S |G_\frac{1}{\sigma}\circ f\circ \phi_\sigma|\right)^\sigma = \sup_{\phi_\frac{1}{\sigma}(S)} |G_\frac{1}{\sigma}\circ f\circ \phi_\sigma|^\sigma =\sup_S | f|=\|f\|,\]
    so $U$ is an isometric lattice homomorphism. Finally,
    \[Up_i=G_\frac{1}{\sigma}\circ p_i\circ \phi_\sigma = G_\frac{1}{\sigma}\circ G_\sigma \circ p_i = p_i,\]
    so $U$ satisfies the properties that uniquely define the functional calculus associated to $U_{\overline{p}}^\sigma:\Hcal\rightarrow \Hcal^n_\sigma$, according to \Cref{thm: functional calculus}. Therefore, both operators are the same.
\end{proof}

Note that in particular, we can compute for any $0<q<\infty$ the $q$-sums in $\Hcal^n_\sigma$ using the operator $U_{\overline{p}}^\sigma$. Indeed, if we write $F_q(t_1,\ldots,t_n)= \intoo{\sum_{i=1}^n t_i^q}^\frac{1}{q}$ and $H_q(t_1,\ldots,t_n)=\intoo{\sum_{i=1}^n |t_i|^q}^\frac{1}{q}$, i.e., $F_q= G_\frac{1}{q}\intoo{\sum_{i=1}^n G_q\circ p_i}$ and $H_q=G_\frac{1}{q}\intoo{\sum_{i=1}^n G_q\circ |p_i|}$, then
\[ U_{\overline{p}}^\sigma F_q= G_\frac{1}{\sigma}\circ F_q\circ \phi_\sigma = G_\frac{1}{\sigma}\circ G_\frac{1}{q}\intoo[3]{\sum_{i=1}^n G_q\circ p_i\circ \phi_\sigma}=  G_\frac{1}{\sigma q}\intoo[3]{\sum_{i=1}^n G_q\circ G_\sigma\circ p_i}= F_{\sigma q},\]
and similarly $U_{\overline{p}}^\sigma H_q= H_{\sigma q}$. \\

Now, we want to extend this explicit way of computing the functional calculus of $\Hcal^n_\sigma$ to any $X_\sigma$. In order to do so, note the following:

\begin{lem}\label{lem: lattice homomorphism with new operations}
    Let $T:X\rightarrow Y$ be a lattice homomorphism between Banach lattices. Then, for every $0<\sigma<\infty$, the map $T:X_\sigma\rightarrow Y_\sigma$ is a lattice homomorphism with respect to the vector lattice structure $(\oplus_\sigma, \odot_\sigma,\leq)$.
\end{lem}

\begin{proof}
    First note that $X_\sigma=X$ and $Y_\sigma=Y$ as sets, so the map $T:X_\sigma\rightarrow Y_\sigma$ is well defined. Now, if $x,y\in X$ and $\alpha\in \R$, we get that
    \[T(\alpha\odot_\sigma x)= T(\alpha^\frac{1}{\sigma}\cdot x)=\alpha^\frac{1}{\sigma}\cdot Tx=\alpha\odot_\sigma Tx,\]
    and, since $T$ preserves functional calculus,
    \[T(x\oplus_\sigma y)=T(F_\sigma(x,y))=F_\sigma(Tx,Ty)=Tx\oplus_\sigma Ty.\]
    Therefore, $T$ is linear. Moreover, $T$ preserves the order and the lattice operations of $X_\sigma$ and $Y_\sigma$, since they coincide with the ones from $X$ and $Y$.
\end{proof}

The previous two facts combined allow us to understand the functional calculus on $X_\sigma$, whenever it can be defined.

\begin{cor}\label{cor: functional calculus with new operations}
    Let $X$ be a Banach lattice, $0<\sigma<\infty$,  $n\in \N$, and assume that the vector lattice $X_\sigma$ can be endowed with a complete lattice norm (i.e., $0<\sigma\leq 1$ or $1<\sigma<\infty$ and $X$ $\sigma$-convex), so that it admits a functional calculus. Let $V_{\overline{x}}:\Hcal^n\rightarrow X$ and $U^\sigma_{\overline{p}}:\Hcal^n\rightarrow \Hcal^n_\sigma$ be the functional calculus in $X$ and $\Hcal^n_\sigma$ associated to $\overline{x}=(x_i)_{i=1}^n\subseteq X$ and $\overline{p}=(p_i)_{i=1}^n\subseteq \Hcal^n$ respectively, where $p_i$ denotes the projections of $\R^n$ onto the $i$-th coordinate. Then, the functional calculus in $X_\sigma$ associated to $\overline{x}=(x_i)_{i=1}^n$, $V_{\overline{x}}^\sigma:\Hcal^n\rightarrow X_\sigma$, is given by $V_{\overline{x}}\circ U^\sigma_{\overline{p}}$.
\end{cor}

\begin{proof}
    The statement follows from the fact that $V_{\overline{x}}:\Hcal^n_\sigma\rightarrow X_\sigma$ is by \Cref{lem: lattice homomorphism with new operations} a lattice homomorphism such that $V_{\overline{x}}\circ U^\sigma_{\overline{p}}p_i=V_{\overline{x}}p_i=x_i$. Therefore, the uniqueness of functional calculus implies that $V_{\overline{x}}^\sigma=V_{\overline{x}}\circ U^\sigma_{\overline{p}}$.
\end{proof}

\begin{rem}
    The assumptions that $X$ is a Banach lattice and $X_\sigma$ can be endowed with a complete lattice norm can be relaxed. The only technical requirement for the $p$-concavification and $p$-convexification of a vector lattice to be defined is that there exists a positively homogeneous functional calculus on the vector lattice. A sufficient condition for the functional calculus to exist is that the vector lattice $X$ is Archimedean and uniformly complete. Under this assumption, it can be shown that the vector lattice $X_\sigma$ is also Archimedean and uniformly complete for every $0<\sigma<\infty$, so it has a functional calculus, and \Cref{cor: functional calculus with new operations} holds.
\end{rem}

As a consequence, we conclude that whenever $X_\sigma$ admits a functional calculus, for every $0<q<\infty$ a $q$-sum of vectors $x_1,\ldots,x_n$ in $X_\sigma$ can be computed as a $\sigma q$-sum in $X$, or equivalently, that
\begin{equation}\label{eq: q-sums with new operations}
    \intoo[3]{\bigoplus_{i=1}^n x_i^q}^\frac{1}{q}:= V_{\overline{x}}^\sigma F_q = V_{\overline{x}}( U^\sigma_{\overline{p}}F_q)= V_{\overline{x}} F_{\sigma q}=\intoo[3]{\sum_{i=1}^n x_i^{\sigma q}}^\frac{1}{\sigma q}
\end{equation}
and
\begin{equation}\label{eq: q-|sums| with new operations}
    \intoo[3]{\bigoplus_{i=1}^n |x_i|^q}^\frac{1}{q}:= V_{\overline{x}}^\sigma H_q = V_{\overline{x}}( U^\sigma_{\overline{p}}H_q)= V_{\overline{x}} H_{\sigma q}=\intoo[3]{\sum_{i=1}^n |x_i|^{\sigma q}}^\frac{1}{\sigma q}.
\end{equation}
Now, we have all the tolls needed to study the convexity and concavity properties of $p$-convexifications and $p$-concavifications.

\begin{prop}\label{prop: properties of p-convexification}
    Let $X$ be a $(r,s)$-convex and $(u,v)$-concave Banach lattice, with $1\leq r\leq s\leq \infty$, $1\leq v\leq u\leq \infty$ and $r\leq u$, and let $1<p<\infty$. Then, $X^{(p)}$, the $p$-convexification of $X$, is $(pr,ps)$-convex with constant $K^{(r,s)}(X)^\frac{1}{p}$ and $(pu,pv)$-concave with constant $K_{(u,v)}(X)^\frac{1}{p}$. In particular, the $p$-convexification of a Banach lattice is always $p$-convex with constant one.
\end{prop}

\begin{proof}
    Recall that the norm in $X^{(p)}=(X_\frac{1}{p},\normiii{\cdot})$ is given by $\normiii{\cdot}=\|\cdot\|^\frac{1}{p}$. Therefore, by \eqref{eq: q-|sums| with new operations} we get that
    \begin{align*}
        \normiii[\Bigg]{\intoo[3]{\bigoplus_{i=1}^n |x_i|^{ps}}^\frac{1}{ps}}& = \norm[3]{\intoo[3]{\sum_{i=1}^n |x_i|^{s}}^\frac{1}{s}}^\frac{1}{p}\leq  K^{(r,s)}(X)^\frac{1}{p} \intoo[3]{\sum_{i=1}^n \|x_i\|^{r}}^\frac{1}{pr}\\
        & =K^{(r,s)}(X)^\frac{1}{p} \intoo[3]{\sum_{i=1}^n \normiii{x_i}^{pr}}^\frac{1}{pr}
    \end{align*}
    and
    \begin{align*}
        \normiii[\Bigg]{\intoo[3]{\bigoplus_{i=1}^n |x_i|^{pv}}^\frac{1}{pv}}& = \norm[3]{\intoo[3]{\sum_{i=1}^n |x_i|^{v}}^\frac{1}{v}}^\frac{1}{p}\geq  K_{(u,v)}(X)^{-\frac{1}{p}} \intoo[3]{\sum_{i=1}^n \|x_i\|^{u}}^\frac{1}{pu}\\
        & =K_{(u,v)}(X)^{-\frac{1}{p}} \intoo[3]{\sum_{i=1}^n \normiii{x_i}^{pu}}^\frac{1}{pu}. \qedhere
    \end{align*}
\end{proof}

\begin{prop}\label{prop: properties of p-concavification}
    Let $X$ be a $p$-convex Banach lattice, such that it is also $(r,s)$-convex and $(u,v)$-concave, for $1\leq r\leq s\leq \infty$, $p\leq v\leq u\leq \infty$ and $1<p< r\leq u$. Then, $X_{(p)}$, the $p$-concavification of $X$, is $(\frac{r}{p},\frac{s}{p})$-convex with constant $K^{(p)}(X)^pK^{(r,s)}(X)^p$ and $(\frac{u}{p},\frac{v}{p})$-concave with constant $K^{(p)}(X)^pK_{(u,v)}(X)^p$. If $X$ satisfies a lower $u$-estimates (i.e., $(u,1)$-concavity) instead of being $(u,v)$-concave, then $X_{(p)}$ satisfies a lower $\frac{u}{p}$-estimate with constant $K^{(p)}(X)^pK_{(\downarrow u)}(X)^p$.
\end{prop}

\begin{proof}
    Recall that the norm in $X_{(p)}=(X_\frac{1}{p},\normiii{\cdot})$ satisfies $\normiii{\cdot}\leq \|\cdot\|^p\leq K^{(p)}(X)^p\normiii{\cdot}$. Again, by \eqref{eq: q-|sums| with new operations} we get that
    \begin{align*}
        \normiii[\Bigg]{\intoo[3]{\bigoplus_{i=1}^n |x_i|^{\frac{s}{p}}}^\frac{p}{s}}& \leq \norm[3]{\intoo[3]{\sum_{i=1}^n |x_i|^{s}}^\frac{1}{s}}^p\leq  K^{(r,s)}(X)^p \intoo[3]{\sum_{i=1}^n \|x_i\|^{r}}^\frac{p}{r}\\
        & \leq K^{(p)}(X)^p K^{(r,s)}(X)^p \intoo[3]{\sum_{i=1}^n \normiii{x_i}^{\frac{r}{p}}}^\frac{p}{r}
    \end{align*}
    and
    \begin{align*}
        \normiii[\Bigg]{\intoo[3]{\bigoplus_{i=1}^n |x_i|^{\frac{v}{p}}}^\frac{p}{v}}& \geq K^{(p)}(X)^{-p} \norm[3]{\intoo[3]{\sum_{i=1}^n |x_i|^{v}}^\frac{1}{v}}^p\geq K^{(p)}(X)^{-p} K_{(u,v)}(X)^{-p} \intoo[3]{\sum_{i=1}^n \|x_i\|^{u}}^\frac{p}{u}\\
        & \geq K^{(p)}(X)^{-p} K_{(u,v)}(X)^{-p} \intoo[3]{\sum_{i=1}^n \normiii{x_i}^{\frac{u}{p}}}^\frac{p}{u}
    \end{align*}
    (note that all the exponents involved are greater or equal than one). If $X$ satisfies a lower $u$-estimate instead, then for every pairwise disjoint vectors $x_1,\ldots,x_n\in X_{(p)}$, we observe that they are also pairwise disjoint in $X$, so
    \begin{align*}
        \normiii[\Bigg]{\bigoplus_{i=1}^n |x_i|}& \geq K^{(p)}(X)^{-p} \norm[3]{\intoo[3]{\sum_{i=1}^n |x_i|^{p}}^\frac{1}{p}}^p = K^{(p)}(X)^{-p} \norm[3]{\sum_{i=1}^n |x_i|}^p \\
        & \geq K^{(p)}(X)^{-p} K_{(\downarrow u)}(X)^{-p} \intoo[3]{\sum_{i=1}^n \|x_i\|^{u}}^\frac{p}{u}\\
        & \geq K^{(p)}(X)^{-p} K_{(\downarrow u)}(X)^{-p} \intoo[3]{\sum_{i=1}^n \normiii{x_i}^{\frac{u}{p}}}^\frac{p}{u}.\qedhere
    \end{align*}
\end{proof}

The fact that convexity and concavity are nicely preserved by the $p$-convexification and $p$-concavification processes allows us to assume that the convexity constant of a Banach lattice is one:

\begin{lem}\label{lem: p-convexification of the p-concavification}
    Let $X$ be a $p$-convex Banach lattice for some $p>1$. Then $X$ is lattice isomorphic to $Y=(X_{(p)})^{(p)}$, the $p$-convexification of the $p$-concavification of X, with distortion $K^{(p)}(X)$. Moreover, $Y$ has $p$-convexity constant one.
\end{lem}

\begin{proof}
    We first note that $X$, $X_{(p)}$ and $Y$ coincide as sets, and the order and lattice operations are the same. Next, recall that 
    \[\|\cdot\|_{X_{(p)}}\leq \|\cdot\|^p\leq K^{(p)}(X)^p \|\cdot\|_{X_{(p)}},\]
    so the norm of $Y$ is equivalent to the norm of $X$: 
    \[K^{(p)}(X)^{-1}\|\cdot\|\leq \|\cdot\|_Y=\|\cdot\|_{X_{(p)}}^\frac{1}{p}\leq \|\cdot\|.\]
    Now, given $\overline{x}=(x_i)_{i=1}^n$ in the set $X$, let us denote by $V_{\overline{x}}:\Hcal_n\rightarrow X$ the functional calculus in the Banach lattice $X$ associated to $\overline{x}$, and by $U^\sigma_{\overline{p}}:\Hcal_n\rightarrow \Hcal_n^\sigma$ the functional calculus associated to the projections $\overline{p}=(p_i)_{i=1}^n$ on the vector lattice $\Hcal_n^\sigma$ for any $0<\sigma<\infty$, as we did in \Cref{cor: functional calculus with new operations}. Then, we know by this same result that the functional calculus on $X_{(p)}$, $V^p_{\overline{x}}:\Hcal_n\rightarrow X_{(p)}$, is given by $V^p_{\overline{x}}=V_{\overline{x}}\circ U^p_{\overline{p}}$. Applying \Cref{cor: functional calculus with new operations} again, this time to $X_{(p)}$ and its $p$-convexification $Y$, we obtain that the functional calculus on $Y$, $(V_{\overline{x}}^p)^\frac{1}{p}:\Hcal_n\rightarrow Y=(X_{(p)})^{(p)}$, is given by $(V^p_{\overline{x}})^\frac{1}{p}=V^p_{\overline{x}}\circ U^\frac{1}{p}_{\overline{p}}=V_{\overline{x}}\circ U^p_{\overline{p}}\circ U^\frac{1}{p}_{\overline{p}}=V_{\overline{x}}$, that is, it coincides with the original functional calculus on $X$. In particular, the sum and multiplication by scalars in $Y$ coincide with the original sum and multiplication by scalars in $X$. In other words, the vector lattice structures in $Y$ and $X$ coincide. Since the norms are equivalent, the identity on $X$ is a lattice isomorphism from $X$ to $Y$ with distortion $K^{(p)}(X)$, as we wanted to show. Finally, the last statement follows from \Cref{prop: properties of p-convexification}.
\end{proof}

This fact was used in \cite[Proposition 1.d.8 and Lemma 1.f.11]{LT2} to establish that every Banach lattice that is $p$-convex and additionally is $q$-concave or satisfies a lower $q$-estimate, respectively, can be renormed so that both constants become one. However, we omit the proof, as we will shortly see using different techniques that any $(p_1,q_1)$-convex and $(p_2,q_2)$-concave Banach lattice can be renormed so that both constants become one (see \Cref{thm: renorming convex and concave}).

\begin{example}\label{ex: convexification}
    Let $1<p<\infty$:
    \begin{enumerate}
        \item $L_p(\mu)_{(p)}=L_1(\mu)$ and $L_p(\mu)=L_1(\mu)^{(p)}$ for every measure $\mu$. It suffices to show that
    \[\fullfunction{T}{L_1(\mu)}{L_p(\mu)_{(p)}}{f}{f^\frac{1}{p}}\]
    is an onto lattice isometry. After this is established, we can use Lemmas \ref{lem: lattice homomorphism with new operations} and \ref{lem: p-convexification of the p-concavification} to conclude the second identification, as $L_p(\mu)$ is $p$-convex with constant one, so it is lattice isometric to $(L_p(\mu)_{(p)})^{(p)}$. Clearly, the map $T$ is well defined, as $f^\frac{1}{p}\in L_p(\mu)$ for every $f\in L_1(\mu)$. Moreover, it is a bijection, with inverse $h\mapsto h^p$. Let us check that it is a linear lattice homomorphism: given $f,g\in L_1(\mu)$ and $\lambda \in \R$,
    \[Tf \oplus_p Tg =(f^\frac{1}{p})\oplus_p (g^\frac{1}{p})=((f^\frac{1}{p})^p + (g^\frac{1}{p})^p)^\frac{1}{p} = (f + g)^\frac{1}{p} =T(f+g),\]
    \[\lambda \odot_p Tf=\lambda \odot_p f^\frac{1}{p} =\lambda^\frac{1}{p} \cdot f^\frac{1}{p} = (\lambda f)^\frac{1}{p} = T(\lambda f),\]
    and
    \[|Tf|=|f^\frac{1}{p}|=|f|^\frac{1}{p}=T|f|.\]
    Finally, $T$ is an isometry:
    \[\|Tf\|_{(L_p)_{(p)}} = \|f^\frac{1}{p}\|_{L_p}^p = \int |f| \, d\mu = \|f\|_{L_1}.\]

    \item $C(K)$ is lattice isometric to both $C(K)_{(p)}$ and $C(K)^{(p)}$ for every compact Hausdorff space $K$. The proof is similar to the first example.

    \item $C(K, L_p(\mu))_{(p)}=C(K, L_1(\mu))$ and $C(K, L_p(\mu))=C(K, L_1(\mu))^{(p)}$ for every compact Hausdorff space $K$ and every measure $\mu$. The isometry that establishes the first identification is
    \[\fullfunction{T}{L_1(\mu)}{L_p(\mu)_{(p)}}{f}{f^\frac{1}{p}}.\]
    \end{enumerate}
\end{example}

We conclude the section with a couple of applications of the $p$-concavification and $p$-convexifi-cation process to the existence of universal constructions within the class of $p$-convex Banach lattices. To this day, the most optimal proofs available are standard applications of these techniques to the corresponding results in the broader class of all Banach lattices. This means that there is no straightforward generalization of these results to the upper $p$-estimates setting (see \cite{GLTT2} for a more detailed discussion on the matter).\\

Our first application generalizes the following result from \cite[Theorems 4.3 and 4.4]{AT}, which establishes the existence of isometric push-outs in the category of Banach lattices with lattice homomorphisms (we recall the definition of an isometric push-out in the statement). One important feature about these push-outs is that they additionally preserve isometric embeddings.

\begin{thm}\label{push-out}
Given Banach lattices $X_0,X_1,X_2$ and lattice homomorphisms $T_i:X_0\rightarrow X_i$ for $i=1,2$, there exists an isometric push-out diagram, that is, there is a Banach lattice $PO$ and contractive lattice homomorphisms $S_i: X_i \rightarrow PO$, $i=1,2$, so that the following diagram commutes:
\begin{equation*}
	\xymatrix{
		X_1 \ar@{->}[r]^{S_1}& PO  \\
		X_0 \ar@{->}[r]_{T_2} \ar@{->}[u]^{T_1} & X_2 \ar@{->}[u]_{S_2} 
	}
\end{equation*}
Additionally, whenever there is a Banach lattice $Y$ and lattice homomorphisms $R_i:X_i \rightarrow Y$, $i=1,2$, satisfying that $R_1T_1=R_2T_2$, there is a unique lattice homomorphism $\gamma: PO\rightarrow Y$ with $\|\gamma\|\leq \max \{\|R_1\|, \|R_2\|\}$ such that the following diagram commutes:
\begin{equation*}
	\xymatrix{
		&  &  Y \\
        X_1 \ar@{->}[r]^{S_1} \ar@/^1.0pc/[rru]^{R_1}& PO \ar@{->}[ru]^{\gamma} & \\
		X_0 \ar@{->}[r]_{T_2} \ar@{->}[u]^{T_1} & X_2 \ar@{->}[u]_{S_2}  \ar@/_1.0pc/[uur]_{R_2} &
	}
\end{equation*}

Even more, if $T_1$ is contractive and $T_2$ is an isometric embedding, then $S_1$ is an isometric embedding.
\end{thm}

Using a $p$-concavification and $p$-convexification argument, we can extend this result to the class of $p$-convex Banach lattices (with constant one) \cite{GLTT2}.

\begin{thm}\label{thm: push-out p-conv}
Given $p$-convex Banach lattices $X_0,X_1,X_2$ with constant one and lattice homomorphisms $T_i:X_0\rightarrow X_i$ for $i=1,2$, there is a $p$-convex Banach lattice $PO_p$ with constant one and contractive lattice homomorphisms $S_i: X_i \rightarrow PO$, $i=1,2$, such that $S_1T_1=S_2T_2$, and whenever there is a $p$-convex Banach lattice $Y$ with constant one and lattice homomorphisms $R_i:X_i \rightarrow Y$, $i=1,2$, satisfying that $R_1T_1=R_2T_2$, there is a unique lattice homomorphism $\gamma: PO\rightarrow Y$ such that $\|\gamma\|\leq \max \{\|R_1\|, \|R_2\|\}$ and $R_i=\gamma S_i$, $i=1,2$.\\

Even more, if $T_1$ is contractive and $T_2$ is an isometric embedding, then $S_1$ is an isometric embedding.
\end{thm}

\begin{proof}
Let $X_{i(p)}$ be the $p$-concavification of $X_i$ for $i=0,1,2$. Then, by \Cref{lem: lattice homomorphism with new operations} we get that the maps $T_i:X_{0(p)}\rightarrow X_{i(p)}$, $i=1,2$, are lattice homomorphisms. Let us apply \Cref{push-out} to obtain a push-out $PO$ in the category of all Banach lattices for the $p$-concavified diagram:
\begin{equation*}
	\xymatrix{
		X_{1(p)} \ar@{->}[r]^{R_1}& PO  \\
		X_{0(p)} \ar@{->}[r]_{T_2} \ar@{->}[u]^{T_1} & X_{2(p)} \ar@{->}[u]_{R_2} 
	}
\end{equation*}
If we $p$-convexify this diagram, we find that for every $i=0,1,2$, $(X_{i(p)})^{(p)}=X_i$ by \Cref{lem: p-convexification of the p-concavification}. Arguing as before, the lattice homomorphisms $R_i:X_i\rightarrow PO^{(p)}$, $i=1,2$, are contractive. We claim that $PO^{(p)}$ satisfies the universal property of the statement. Consider a $p$-convex Banach lattice $Y$ with constant one and lattice homomorphisms $R_i:X_i \rightarrow Y$, $i=1,2$, satisfying that $R_1T_1=R_2T_2$. By $p$-concavifying the new diagram and applying the universal property of $PO$, we get that there exists a unique lattice homomorphism $\gamma : PO\rightarrow Y_{(p)}$ such that $\|\gamma\|\leq \max_{i=1,2}\|R_i:X_{i(p)}\rightarrow Y_{(p)}\|$ and $R_i=\gamma S_i$, $i=1,2$. If we $p$-convexify back, we get a unique lattice homomorphism from $PO^{(p)}$ to $Y$ satisfying the properties of the statement.\\

Finally, if $T_2$ was originally an isometric embedding, then it remains an isometric embedding after the $p$-concavification, since $\|\cdot\|_{X_{i(p)}}=\|\cdot\|_{X_{i}}^p$. \Cref{push-out} then yields that $R_1:X_{1(p)} \rightarrow PO$ is an isometric embedding, so it remains an isometry after $p$-convexifying.
\end{proof}

It is worth mentioning that the existence of isometric push-outs in the class of Banach lattices with upper $p$-estimates was established in \cite{GLTT} using more general tools. However, since $p$-convexification and $p$-concavification tools are not available in this setting, only an isomorphic version of the last part of the statement can be obtained with the current techniques (see \cite{GLTT2}).\\

Our second application concerns the existence of separable universal spaces for the class of all separable $p$-convex Banach lattices. In the general setting of Banach lattices, the following result was established \cite{LLOT}:

\begin{thm}\label{thm: separable universal BL}
    Every separable Banach lattice embeds lattice isometrically into $C(\Delta, L_1[0,1])$, where $\Delta$ denotes the Cantor set.
\end{thm}

This result extends automatically to the $p$-convex setting using the techniques developed in this section \cite[Theorem 6.1]{GLTT2}. It remains an open question whether there is an analogous result in the upper $p$-estimates setting.

\begin{thm}\label{thm: universal p convex}
   Every $p$-convex Banach lattice with constant 1 embeds lattice isometrically into $C(\Delta, L_p[0,1])$.
\end{thm}

\begin{proof}
    Recall from \Cref{ex: convexification}(3) that $C(\Delta, L_p[0,1])$ can be identified with the $p$-convexification of $C(\Delta, L_1[0,1])$, so it suffices to show that $C(\Delta, L_1[0,1])^{(p)}$ is universal for the class of separable $p$-convex Banach lattices. Let $X$ be a $p$-convex Banach lattice with constant 1, and consider its $p$-concavification $X_{(p)}$, which is also a Banach lattice. By \Cref{thm: separable universal BL}, there exists a lattice isometric embedding $T:X_{(p)}\rightarrow C(\Delta, L_1[0,1])$. If we $p$-convexify both spaces, we find that $T$ is again a lattice isometry from $X=(X_{(p)})^{(p)}$ into $C(\Delta, L_1[0,1])^{(p)}$, as we wanted to show.
\end{proof}

\section{Abstract factorizations of convex and concave operators}\label{sec: abstract factorizations}

We go back to the study of $(p,q)$-convex and $(p,q)$-concave operators and their connections to $(p,q)$-convex and $(p,q)$-concave spaces by means of the following factorization results. The theory of ``abstract'' factorizations of $(p,q)$-convex and $(p,q)$-concave operators (in contrast to the factorizations through concrete functions spaces, that will be treated in the next section) was originally introduced by Reisner \cite{Reisner} for $p$-convex and $p$-concave operators, and Meyer-Nieberg \cite[Section 2.8]{MN} and Byrd \cite{Byrd} in the setting of $(p,\infty)$-convexity and $(p,1)$-concavity, but has been gradually refined in subsequent works \cite{RT, GLTT2}. Essentially, the next two theorems reduce the study of $(p,q)$-convex and $(p,q)$-concave operators to the understanding of $(p,q)$-convex and $(p,q)$-concave Banach lattices. Moreover, the constructions that we will obtain will be canonical in a certain sense. Namely, the factorization for $(p,q)$-concave Banach lattices will satisfy some ``maximality'' property, while its analogue for $(p,q)$-convex operators will be ``minimal''. As a byproduct, we will obtain a renorming result that will allow us to assume that the convexity and concavity constants of a Banach lattice are simply one.

\begin{thm}\label{thm: factorization concave}
    Let $X$ be a Banach lattice, $E$ a Banach space, $1\leq q\leq p\leq \infty$, and $T:X\rightarrow E$ a $(p,q)$-concave operator. Then, there exists a $(p,q)$-concave Banach lattice $Z^0$ with constant one, a lattice homomorphism with dense range $\phi^0: X\rightarrow Z^0$ and an operator $S^0:Z^0\rightarrow E$ such that $T=S^0  \phi^0$.\\ 
    
    Moreover, if $Z$ is a $(p,q)$-concave Banach lattice, $\phi: X\rightarrow Z$ is a lattice homomorphism with dense range, and $S:Z\rightarrow E$ is an operator such that $T=S  \phi$, then there exists a lattice homomorphism with dense range $\psi: Z\rightarrow Z^0$ such that $\phi^0=\psi \phi$ and $S=S^0\psi$.
    \begin{center}
        \begin{tikzcd}
        X  \ar[dr, "\phi^0"] \ar[ddr, "\phi"'] \ar[rr, "T"]  &    & E  \\
        &  Z^0 \ar[ur, "S^0"]  & \\
        &  Z \ar[uur, "S"'] \ar[u, dashed, "\psi" near end]   & 
        \end{tikzcd} 
        \end{center}
\end{thm}

\begin{proof}
    Note that if $p=\infty$, the result is trivially true considering $Z^0=X$ (every Banach lattice is $(\infty,q)$-concave with constant one for every $1\leq q\leq \infty$) and $\phi^0$ the identity operator on $X$. Therefore, we assume that $p$ (and hence, $q$) is finite. Given $x\in X$, we define 
    \[\rho(x)=\sup \cbr[3]{\intoo[3]{\sum_{i=1}^n \|Tx_i\|^{p}}^\frac{1}{p} : x_1,\ldots,x_n \in X,\intoo[3]{\sum_{i=1}^n |x_i|^{q}}^\frac{1}{q} \leq |x|}.\]
    Note that for every possible choice of $x_1,\ldots,x_n \in X$, the $(p,q)$-concavity of $T$ yields that 
    \[\intoo[3]{\sum_{i=1}^n \|Tx_i\|^{p}}^\frac{1}{p} \leq K_{(p,q)}(T)\norm[3]{\intoo[3]{\sum_{i=1}^n |x_i|^{q}}^\frac{1}{q}} \leq K_{(p,q)}(T)\|x\|,\]
    so in particular $\rho(x)$ is finite for every $x\in X$ and $\rho(x)\leq K_{(p,q)}(T)\|x\|$. It is also clear that $\|Tx\|\leq \rho(x)$. Let us show that $\rho$ is a lattice seminorm on $X$. It is easy to check that $\rho$ is positively homogeneous. Moreover, if $|x|\leq |y|$, then every finite sequence $x_1,\ldots,x_n \in X$ such that $\intoo{\sum_{i=1}^n |x_i|^{q}}^\frac{1}{q} \leq |x|$ can also be used to compute $\rho(y)$, so $\rho(x)\leq \rho(y)$. To prove the triangular inequality, we fix $x,y\in X$, and let $z=|x|+|y|$ and $I_z$ the principal ideal generated by $z$, that we can identify with a certain $C(K)$ through an onto lattice isomorphism $J:I_z\rightarrow C(K)$ such that $Jz=\one$, the constant one function. For any choice $z_1,\ldots,z_n\in X$ such that $\intoo{\sum_{i=1}^n |z_i|^{q}}^\frac{1}{q}\leq |x+y|$, it is clear that $z_1,\ldots,z_n\in I_z$. Let us denote by $f=Jx$, $g=Jy$ and $h_i=Jz_i$. It follows that $|f|+|g|=\one$ and 
    \[\intoo{\sum_{i=1}^n |h_i|^{q}}^\frac{1}{q}=J\intoo{\sum_{i=1}^n |z_i|^{q}}^\frac{1}{q}\leq J|x+y|\leq Jz=\one.\]
    We construct $f_i=|f|h_i$ and $g_i=|g|h_i$, $i=1,\ldots,n$, so that
    \[\intoo[3]{\sum_{i=1}^n |f_i|^{q}}^\frac{1}{q}=|f|\intoo[3]{\sum_{i=1}^n |h_i|^{q}}^\frac{1}{q}\leq |f|, \quad \intoo[3]{\sum_{i=1}^n |g_i|^{q}}^\frac{1}{q}=|g|\intoo[3]{\sum_{i=1}^n |h_i|^{q}}^\frac{1}{q}\leq |g|\] 
    and $h_i=f_i+g_i$ for every $i=1,\ldots,n$. It follows that $x_i=J^{-1}f_i$ and $y_i=J^{-1}g_i$ satisfy $\intoo{\sum_{i=1}^n |x_i|^{q}}^\frac{1}{q}\leq |x|$ and $\intoo{\sum_{i=1}^n |y_i|^{q}}^\frac{1}{q}\leq |y|$, so
    \[\intoo[3]{\sum_{i=1}^n \|Tz_i\|^{p}}^\frac{1}{p} \leq \intoo[3]{\sum_{i=1}^n (\|Tx_i\|+\|Ty_i\|)^{p}}^\frac{1}{p} \leq \intoo[3]{\sum_{i=1}^n \|Tx_i\|^{p}}^\frac{1}{p} +\intoo[3]{\sum_{i=1}^n \|Ty_i\|^{p}}^\frac{1}{p} \leq \rho(x)+\rho(y).  \]
    Since the choice of $(z_i)_{i=1}^n$ was arbitrary, we conclude that $\rho$ satisfies the triangular inequality.\\

    Now, let $Z^0$ be the completion of $X/\ker\rho$ with $\hat{\rho}$, the norm induced by $\rho$ on the quotient. Let $\phi^0:X\rightarrow Z^0$ be the composition of the quotient map from $X$ to $X/\ker\rho$, which is an onto interval preserving lattice homomorphism, and the formal inclusion from $X/\ker\rho$ into the completion $Z^0$, which is an injective lattice homomorphism with dense range, so that $\phi^0$ is an almost interval preserving lattice homomorphism with dense range, and $\|\phi^0\|\leq K_{(p,q)}(T) $. Let us define for every equivalence class $\phi x \in X/\ker\rho$, with $x\in X$, the map $S^0(\phi^0 x)=Tx$. It is well defined, since for every $y\in X$ such that $\phi^0 y=\phi^0 x$ we have that $\|Tx-Ty\|\leq \rho(x-y)=0$, so $S^0 (\phi^0 x)$ does not depend on the representative of the class. Moreover, $\|S^0(\phi^0 x)\|=\|Ty\|\leq \rho(y)$ for every $y$ in the class $\phi^0 x$, so $\|S^0(\phi^0 x)\|\leq \hat{\rho}(\phi^0 x)$ and $S^0: X/\ker\rho\rightarrow E$ is bounded. We denote its unique extension to the completion $Z^0$ by $S^0$ again. The identity $S^0 \phi^0=T$ follows from the definition.\\

    Next, we need to show that $Z^0$ is $(p,q)$-concave with constant one. Let $z_1,\ldots,z_n\in Z^0$, and fix $\eps>0$. We can find $x_1,\ldots, x_n \in X$ such that $\hat{\rho}(z_i-\phi^0 x_i)<\eps/n^\frac{1}{p}$, and for each $x_i$ choose $(y_{i,j})_{j=1}^{k_i}\subseteq X$ such that 
    \[ \intoo[3]{\sum_{j=1}^{k_i} |y_{i,j}|^{q}}^\frac{1}{q} \leq |x_i| \quad \text{and} \quad \hat{\rho}(\phi^0 x_i)^p=\rho(x_i)^p\leq \sum_{j=1}^{k_i} \|Ty_{i,j}\|^{p}+ \frac{\eps^p}{n}.\]
    It follows using \Cref{prop: monotonicity of p-sums} that $\intoo[1]{\sum_{i=1}^n \sum_{j=1}^{k_i} |y_{i,j}|^{q}}^\frac{1}{q} \leq \intoo{\sum_{i=1}^n |x_i|^{q}}^\frac{1}{q}$. Moreover, since $\phi^0$ is a lattice homomorphism, it preserves the functional calculus, so
    \begin{align*}
        \phi^0 \intoo[3]{\sum_{i=1}^n |x_i|^{q}}^\frac{1}{q}= \intoo[3]{\sum_{i=1}^n |\phi^0 x_i|^{q}}^\frac{1}{q} & \leq \intoo[3]{\sum_{i=1}^n |z_i|^{q}}^\frac{1}{q}+ \intoo[3]{\sum_{i=1}^n |z_i-\phi^0 x_i|^{q}}^\frac{1}{q} \\
        & \leq  \intoo[3]{\sum_{i=1}^n |z_i|^{q}}^\frac{1}{q}+\sum_{i=1}^n |z_i-\phi^0 x_i| .
    \end{align*}
    To conclude the argument, we just observe that
    \begin{align*}
        \intoo[3]{\sum_{i=1}^n \hat{\rho}(z_i)^{p}}^\frac{1}{p} & \leq \intoo[3]{\sum_{i=1}^n \hat{\rho}(\phi^0 x_i)^{p}}^\frac{1}{p} + \intoo[3]{\sum_{i=1}^n \hat{\rho}(z_i-\phi^0 x_i)^{p}}^\frac{1}{p} \leq  \intoo[3]{\sum_{i=1}^n  \intoo[3]{\sum_{j=1}^{k_i} \|Ty_{i,j}\|^{p}+ \frac{\eps^p}{n}}}^\frac{1}{p} + \eps \\
        & \leq \rho \intoo[3]{\intoo[3]{\sum_{i=1}^n |x_i|^{q}}^\frac{1}{q}} + 2\eps = \hat{\rho} \intoo[3]{\phi^0 \intoo[3]{\sum_{i=1}^n |x_i|^{q}}^\frac{1}{q}} + 2\eps\\
        & \leq \hat{\rho} \intoo[3]{ \intoo[3]{\sum_{i=1}^n |z_i|^{q}}^\frac{1}{q}} + \sum_{i=1}^n \hat{\rho}(z_i-\phi^0 x_i)+ 2\eps \leq \hat{\rho} \intoo[3]{ \intoo[3]{\sum_{i=1}^n |z_i|^{q}}^\frac{1}{q}} + (2+n^\frac{1}{p^*})\eps,
    \end{align*}
    so $Z$ is $(p,q)$-concave with constant one.\\

    Finally, let us show the last claim. Given $Z$, $\phi$ and $S$ as in the statement, we can define the map $\psi$ that sends every $\phi x\in \phi(X)\subseteq Z$ to $\phi^0 x\in Z^0$. This map is well defined and bounded. Indeed, if $x_1,\ldots, x_n\in X$ are such that $\intoo{\sum_{i=1}^n |x_i|^{q}}^\frac{1}{q}\leq |x|$, since $\phi$ is a lattice homomorphism we have
    \begin{align*}
        \intoo[3]{\sum_{i=1}^n \|Tx_i\|^{p}}^\frac{1}{p}& = \intoo[3]{\sum_{i=1}^n \|S\phi x_i\|^{p}}^\frac{1}{p} \leq \|S\| \intoo[3]{\sum_{i=1}^n \|\phi x_i\|^{p}}^\frac{1}{p} \leq \|S\| K_{(p,q)}(Z) \norm[3]{\intoo[3]{\sum_{i=1}^n |\phi x_i|^{q}}^\frac{1}{q}}\\
        & \leq \|S\| K_{(p,q)}(Z) \norm[3]{\phi \intoo[3]{\intoo[3]{\sum_{i=1}^n | x_i|^{q}}^\frac{1}{q} }} \leq \|S\| K_{(p,q)}(Z) \|\phi x\|. 
    \end{align*}
    Taking the supremum over all possible choices of $x_1,\ldots, x_n\in X$ we get that $\hat{\rho}(\phi^0 x) = \rho (x)\leq \|S\| K_{(p,q)}(Z) \|\phi x\|$. This implies, in particular, that $\phi^0 x=0$ whenever $\phi x=0$, so $\psi$ is well defined, and then the previous inequality reads $\hat{\rho}(\psi(\phi x)) \leq  \|S\| K_{(p,q)}(Z) \|\phi x\|$, so $\psi$ is bounded. Since $\phi$ has dense range, $\psi$ can be extended to the whole $Z$. It is also clear that $\psi$ is a lattice homomorphism, and that $\phi^0=\psi \phi$. Since $\phi^0$ has dense range, $\psi$ does too. Finally, for every $x\in X$ we have $S^0\psi (\phi x)=S^0\phi^0 x =Tx =S(\phi x)$, so we conclude that $S=S^0\psi$. Therefore, the proof is concluded.  
\end{proof}

\begin{rem}\label{rem: positivity in factorization concave}
    Note that if $E$ is a Banach lattice and $T$ is positive (respectively, a lattice homomorphism), then $S^0$ can be made positive (respectively, a lattice homomorphism).
\end{rem}

As a consequence of \Cref{thm: factorization concave}, we obtain that $(p,q)$-concave operators are precisely those that factor through $(p,q)$-concave Banach lattices:

\begin{cor}\label{cor: factorization concave}
    Let $X$ be a Banach lattice, $E$ a Banach space and $1\leq q\leq p\leq \infty$. Then, the operator $T:X\rightarrow E$ is $(p,q)$-concave if and only if there exists a $(p,q)$-concave Banach lattice $Z$, a positive operator $\phi: X\rightarrow Z$ and an operator $S:Z\rightarrow E$ such that $T=S  \phi$.
\end{cor}

\begin{proof}
    One implication follows from \Cref{thm: factorization concave}. To prove the reverse, let $Z$, $\phi: X\rightarrow Z$ and $S:Z\rightarrow E$ be as in the statement. Since $\phi$ is positive, by \Cref{lem: positive operators and functional calculus}
    \begin{align*}
      \intoo[3]{\sum_{i=1}^n \|Tx_i\|^{p}}^\frac{1}{p} &\leq \|S\|\intoo[3]{\sum_{i=1}^n \|\phi x_i\|^{p}}^\frac{1}{p} \leq \|S\| K_{(p,q)}(Z) \norm[3]{\intoo[3]{\sum_{i=1}^n |\phi x_i|^{q}}^\frac{1}{q}}\\
      & \leq \|S\|  K_{(p,q)}(Z)  \norm[3]{\phi \intoo[3]{\sum_{i=1}^n |x_i|^{q}}^\frac{1}{q}} \leq \|S\| \|\phi\|  K_{(p,q)}(Z)  \norm[3]{ \intoo[3]{\sum_{i=1}^n |x_i|^{q}}^\frac{1}{q}}
    \end{align*}
    for every $x_1,\ldots, x_n \in X$, so $T$ is $(p,q)$-concave.
\end{proof}

Note from \Cref{rem: positivity in factorization concave} that, if $T$ is the identity operator on $X$, then $Z^0$ is a renorming of $X$ such that the $(p,q)$-concavity constant becomes one. This renorming actually preserves convexity properties, so, in combination with \Cref{thm: duality convexity concavity}, it provides a way of renorming a $(p_1,q_1)$-convex and $(p_2,q_2)$-concave Banach lattice so that both constants become one. 

\begin{thm}\label{thm: renorming convex and concave}
    Let $X$ be a $(p_1,q_1)$-convex and $(p_2,q_2)$-concave Banach lattice, where $1\leq p_1\leq p_2\leq \infty$, $p_1\leq q_1\leq \infty$ and $1\leq q_2\leq p_2$.  
    \begin{enumerate}
        \item There is a $K_{(p_2,q_2)}(X)$-equivalent lattice renorming of $X$ so that the $(p_2,q_2)$-concavity constant becomes one and the $(p_1,q_1)$-convexity remains the same.
        \item Moreover, there exists a $K^{(p_1,q_1)}(X)K_{(p_2,q_2)}(X)$-equivalent lattice renorming of $X$ that is $(p_1,q_1)$-convex and $(p_2,q_2)$-concave with constant one.
    \end{enumerate}
\end{thm}

\begin{proof}
    $(1)$ Let us consider for every $x\in X$ the expression from the proof of \Cref{thm: factorization concave} when $T$ is the identity operator on $X$:
    \[\rho(x)=\sup \cbr[3]{\intoo[3]{\sum_{i=1}^n \|x_i\|^{p_2}}^\frac{1}{p_2} : x_1,\ldots,x_n \in X,\intoo[3]{\sum_{i=1}^n |x_i|^{q_2}}^\frac{1}{q_2} \leq |x|}.\]
    It was established that $\rho$ is a lattice seminorm satisfying $\rho(x)\leq K_{(p_2,q_2)}(X) \|x\|$. However, in this particular case, $\|x\|\leq \rho(x)$, so it defines an equivalent renorming of $X$. Moreover, for every $x_1,\ldots, x_n\in X$, it was shown that
    \[\intoo[3]{\sum_{i=1}^n \rho(x_i)^{p_2}}^\frac{1}{p_2} \leq  \rho \intoo[3]{ \intoo[3]{\sum_{i=1}^n |x_i|^{q_2}}^\frac{1}{q_2}},\]
    so $(X,\rho)$ is $(p_2,q_2)$-concave with constant one.\\

    Now, let us fix $x_1,\ldots, x_n\in X$ and $\eps>0$, and find $y_1,\ldots, y_m\in X$ such that 
    \[\intoo[3]{\sum_{j=1}^m |y_j|^{q_2}}^\frac{1}{q_2} \leq \intoo[3]{\sum_{i=1}^n |x_i|^{q_1}}^\frac{1}{q_1} \quad \text{and} \quad \rho \intoo[3]{\intoo[3]{\sum_{i=1}^n |x_i|^{q_1}}^\frac{1}{q_1}} \leq \intoo[3]{\sum_{j=1}^m \|y_j\|^{p_2}}^\frac{1}{p_2} + \eps .\]
    Define $e= \intoo{\sum_{i=1}^n |x_i|^{q_1}}^\frac{1}{q_1}$, and let us represent $I_e$, the principal ideal generated by $e$, as a $C(K)$-space using Kakutani's Representation Theorem for $AM$-spaces. Under this lattice isomorphism, $e$ corresponds to $\uno$, the constant one function over $K$, and the vectors $x_i$ and $y_j$ can be identified with some $f_i$ and $g_j$ in $C(K)$ such that
    \[\intoo[3]{\sum_{j=1}^m |g_j|^{q_2}}^\frac{1}{q_2} \leq \intoo[3]{\sum_{i=1}^n |f_i|^{q_1}}^\frac{1}{q_1} =\uno.\]
    Let $h_{ij}=f_ig_j$, which satisfy that
    \[\intoo[3]{\sum_{i=1}^n |h_{ij}|^{q_1}}^\frac{1}{q_1} = |g_j|\intoo[3]{\sum_{i=1}^n |f_i|^{q_1}}^\frac{1}{q_1} = |g_j| \quad \text{and} \quad  \intoo[3]{\sum_{j=1}^m |h_{ij}|^{q_2}}^\frac{1}{q_2} = |f_i|\intoo[3]{\sum_{j=1}^m |g_j|^{q_2}}^\frac{1}{q_2} \leq |f_i|  \]
    for every $j=1,\ldots, m$ and $i=1,\ldots,n$, respectively. Let $z_{ij}\in I_e$ be the corresponding image of $h_{ij}$. Since $I_e$ and $C(K)$ are lattice isomorphic, 
    \[\intoo[3]{\sum_{i=1}^n |z_{ij}|^{q_1}}^\frac{1}{q_1} = |y_j| \quad \text{and} \quad  \intoo[3]{\sum_{j=1}^m |z_{ij}|^{q_2}}^\frac{1}{q_2} \leq |x_i|.  \]
    Therefore, applying the $(p_1,q_1)$-convexity of $X$, Minkowski's (integral) inequality with exponent $\frac{p_2}{p_1}$ and the definition of $\rho$, we obtain that
    \begin{align*}
        \rho \intoo[3]{ \intoo[3]{\sum_{i=1}^n |x_i|^{q_1}}^\frac{1}{q_1}} -\eps & \leq \intoo[3]{\sum_{j=1}^m \|y_j\|^{p_2}}^\frac{1}{p_2} =  \intoo[3]{\sum_{j=1}^m \norm[3]{\intoo[3]{\sum_{i=1}^n |z_{ij}|^{q_1}}^\frac{1}{q_1} }^{p_2}}^\frac{1}{p_2}\\
        & \leq K^{(p_1,q_1)}(X) \intoo[3]{\sum_{j=1}^m \intoo[3]{\sum_{i=1}^n \|z_{ij}\|^{p_1}}^\frac{p_2}{p_1} }^\frac{1}{p_2}\\
        & \leq  K^{(p_1,q_1)}(X) \intoo[3]{\sum_{i=1}^n \intoo[3]{\sum_{j=1}^m \|z_{ij}\|^{p_2}}^\frac{p_1}{p_2} }^\frac{1}{p_1}\\
        & \leq  K^{(p_1,q_1)}(X) \intoo[3]{\sum_{i=1}^n  \rho(x_i)^{p_1} }^\frac{1}{p_1},
    \end{align*}
    so $(X,\rho)$ is $(p_1,q_1)$-convex with constant $K^{(p_1,q_1)}(X)$, as we wanted to show.\\

    $(2)$ It follows from applying $(1)$ twice. First, we apply it to $X$, obtaining $Y=(X,\rho)$. Next, we apply it to $Y^*$, which, by \Cref{thm: duality convexity concavity}, is $(p_1^*,q_1^*)$-concave with constant $K^{(p_1,q_1)}(X)$ and $(p_2^*,q_2^*)$-convex with constant one. We then obtain an isomorphic Banach lattice $Z$ whose $(p_1^*,q_1^*)$-concavity and $(p_2^*,q_2^*)$-convexity constants are both one. Taking duals again, it is clear that $X^{**}$ is lattice isomorphic to $Z^*$, so the result follows.
\end{proof}

In particular, we can characterize the Banach lattices that simultaneously satisfy an upper and a lower $p$-estimate for some $1\leq p \leq \infty$ as those which are lattice isomorphic to some $L_p$-space (or $AM$-space when $p=\infty$).

\begin{cor}\label{cor: isomorphic ALp-spaces}
    Let $X$ be a Banach lattice and $1\leq p\leq \infty$. The following are equivalent:
    \begin{enumerate}
        \item $X$ is $p$-convex and $p$-concave.
        \item $X$ satisfies an upper $p$-estimate and a lower $p$-estimate.
        \item $X$ is lattice isomorphic to $L_p(\mu)$ for some measure $\mu$ (respectively, to an $AM$-space when $p=\infty$).
    \end{enumerate}
\end{cor}

\begin{proof}
        $(3)\Rightarrow (1) \Rightarrow (2)$ Trivial.\\

    $(2)\Rightarrow (3)$ Recall from \Cref{thm: upe equivalent to mixed convexity} that upper and lower $p$-estimates are equivalent to $(p,\infty)$-convexity and $(p,1)$-concavity, respectively. Therefore, if $Y$ denotes the equivalent renorming of $X$ given by \Cref{thm: renorming convex and concave}, then $Y$ satisfies an upper and a lower $p$-estimate with constant one, and hence its norm is $p$-additive on disjoint elements. If $p<\infty$, by Kakutani's Representation Theorem for $AL_p$-spaces (cf. \cite[Theorem 2.7.1]{MN}), it follows that $Y$ is lattice isometric to an $L_p(\mu)$ space for some measure $\mu$, as we wanted to show. If $p=\infty$, we observe that $Y^*$ satisfies an upper and a lower $1$-estimate with constant one, so it is lattice isometric to an $L_1(\mu)$-space. Consequently, $X$ is lattice isomorphic to a sublattice of $Y^{**}$, which is a $C(K)$-space, so the proof is concluded.
\end{proof}

Next, we prove the corresponding statement for $(p,q)$-convex operators. But first, let us introduce the following definition: an operator $T: X\rightarrow Y$ between Banach lattices is said to belong to the \textit{class $\Ccal$} if it is an injective interval preserving lattice homomorphism such that $T(B_{X})$ is closed in $Y$.

\begin{thm}\label{thm: factorization convex}
    Let $X$ be a Banach lattice, $E$ a Banach space, $1\leq p\leq q\leq \infty$, and $T:E\rightarrow X$ a $(p,q)$-convex operator. Then, there exists a $(p,q)$-convex Banach lattice $Z_0$ with constant one, a class $\Ccal$ operator $\phi_0: Z_0\rightarrow X$, and an operator $S_0:E\rightarrow Z_0$ such that $T=\phi_0 S_0$.\\

    Moreover, if $Z$ is a $(p,q)$-convex Banach lattice, $\phi: Z\rightarrow X$ is a class $\Ccal$ operator, and $S:E\rightarrow Z$ is an operator such that $T=  \phi S$, then there exists a class $\Ccal$ operator $\psi: Z_0\rightarrow Z$ such that $\phi_0=\phi \psi$ and $S=\psi S_0$.
    \begin{center}
        \begin{tikzcd}
        E  \ar[dr, "S_0"] \ar[ddr, "S"'] \ar[rr, "T"]  &    & X  \\
        &  Z_0 \ar[ur, "\phi_0"] \ar[d, dashed, "\psi" near start] & \\
        &  Z \ar[uur, "\phi"']   & 
        \end{tikzcd} 
        \end{center}
\end{thm}

\begin{proof}
    Let us consider the adjoint operator $T^*:X^*\rightarrow E^*$, which is $(p^*,q^*)$-concave with $K_{(p^*,q^*)}(T^*)=K^{(p,q)}(T)$ by \Cref{thm: duality convexity concavity}. As in the proof of \Cref{thm: factorization concave}, we define for every $x^*\in X^*$ the lattice seminorm
    \[\rho(x^*)=\sup \cbr[3]{\intoo[3]{\sum_{i=1}^n \|T^*x^*_i\|^{p^*}}^\frac{1}{p^*} : x^*_1,\ldots,x^*_n \in X^*,\intoo[3]{\sum_{i=1}^n |x^*_i|^{q^*}}^\frac{1}{q^*} \leq |x^*|}.\]
    Recall that $\|T^*x^*\|\leq\rho(x^*)\leq K^{(p,q)}(T)\|x^*\|$ for every $x^*\in X^*$. Let us define 
    \[A=\{x^*\in X^*:\rho(x^*)\leq 1\},\]
    which is a convex balanced solid subset of $X^*$, and consider its polar set 
    \[A^\circ=\{x\in X: x^*(x)\leq 1 \,\,\forall x^*\in A\}\subseteq X\]
    (throughout the proof, we will take polars with respect to the pair $\langle X^*,X\rangle$). Let $\tau$ be the Minkowski of $A^\circ$, which can also be computed using the expression
    \[\tau(x)=\sup \{x^*(x):x^*\in A\} \quad \text{ for every }\quad x\in X.\]
    We claim that $\tau$ is a complete lattice norm in $Z_0=\spn A^\circ$. Clearly, $\tau$ is positively homogeneous and satisfies the triangular inequality. Next, note that if $x\geq 0$, 
    \[\tau(x)=\sup\{x^*(x):x^*\in X^*_+\cap A\}.\]
    From this we get that $\tau(x)\leq\tau(y)$ whenever $0\leq x\leq y$. Additionally, using Riesz--Kantorovich formulas and the fact that $A$ is solid we can show that $\tau(x)= \tau(|x|)$:
    \begin{align*}
        \tau(|x|)&=\sup \{x^*(|x|):x^*\in X^*_+\cap A\}= \sup \{|y^*(x)|:0\leq |y^*|\leq x^*, x^*\in A\}\\
        &=  \sup \{y^*(x):y^*\in A\} = \tau(x).
    \end{align*}
    We conclude that $\tau(x)\leq \tau(y)$ when $|x|\leq |y|$. Since $B_{X^*}\subseteq K^{(p,q)}(T) A$, for every $x\in Z_0$ we get that
    \[\|x\|\leq  K^{(p,q)}(T) \tau(x). \]
    In particular, $\tau(x)=0$ if and only if $x=0$. From all these facts, we conclude that $Z_0$ is a (not necessarily closed) ideal of $X$ and that $\tau$ is a lattice norm in $Z_0$. It remains to show that $\tau$ is complete. To this end, note that $A^\circ$ is closed in $X$. Now, take a sequence $(x_n)_n\subseteq Z_0$ such that $\sum_n \tau(x_n) <1$. In particular, $\sum_n \|x_n\| < K^{(p,q)}(T)$, so $\sum_{n=1}^m x_n$ converges in $X$ to some $x$. By the definition of the Minkowski functional, we can find scalars $(\lambda_n)_n\in \R_+$ so that $x_n \in \lambda_n A^\circ$  for all $n$ and $\sum_n \lambda_n <1$. Since 
    \[\sum^m_{n=1}x_n \in \sum^m_{n=1}\lambda_n\cdot A^\circ \subseteq A^\circ\] 
    for all $m$ and $A^\circ$ is closed, $x\in A^\circ$. Now, if $k >m$,
    \[ \sum^k_{n=m+1}x_n  \in \sum^k_{n=m+1}\lambda_n\cdot A^\circ \subseteq \sum^\infty_{n=m+1}\lambda_n\cdot A^\circ.\]
    Taking the limit $k\to \infty $ in $X$, we see that  $x - \sum^m_{n=1}x_n\in \sum^\infty_{n=m+1}\lambda_n\cdot A^\circ$.  Hence, $\rho(x - \sum^m_{n=1}x_n)\to 0$ as $m\to \infty$, so $\rho$ is complete and $(Z_0,\tau)$ is a Banach lattice.\\
    
    Now, consider the formal identity $\phi_0:Z_0\rightarrow X$, which is an injective interval preserving lattice homomorphism with $\phi_0(B_{Z_0})=A^\circ$, which is closed. In other words, $\phi_0$ is a class $\Ccal$ operator. Moreover, $\|\phi_0\|\leq K^{(p,q)}(T)$. Next, we note that $T(E)\subseteq Z_0$ and $\tau(Tu)\leq \|u\|$ for every $u\in E$. Indeed, given $x^*\in A$,
    \[|x^*(Tu)|=|T^*x^*(u)|\leq \|T^*x^*\|\|u\|\leq \rho(x^*)\|u\|\leq \|u\|.\]
    Therefore, the operator $S_0:E\rightarrow Z_0$ given by $S_0u=Tu$ is well defined and bounded, with $\|S_0\|\leq 1$, and $T$ factors as $T=\phi_0S_0$.\\

    Next, we need to show that $Z_0$ is $(p,q)$-convex with constant one. Let us fix $(x_i)_{i=1}^n\subseteq Z_0$, $x^*\in A$ and $\eps>0$. By \Cref{cor: RK for p sums} (applied in $X^*$ instead of $X$), we can find $(x_i^*)_{i=1}^n\subseteq X^*$ such that
    \[\intoo[3]{\sum_{i=1}^n |x^*_i|^{q^*}}^\frac{1}{q^*} \leq |x^*|\quad \text{and} \quad |x^*|\intoo[3]{\intoo[3]{\sum_{i=1}^n |x_i|^{q}}^\frac{1}{q}}\leq \sum_{i=1}^n |x_i^*|(|x_i|)+\eps.\]
    In the proof of \Cref{thm: factorization concave} it was proven that $\rho$ satisfies the $(p^*,q^*)$-concavity inequality, so we know that
    \[\intoo[3]{\sum_{i=1}^n \rho(x^*_i)^{p^*}}^\frac{1}{p^*} \leq \rho \intoo[3]{\intoo[3]{\sum_{i=1}^n |x^*_i|^{q^*}}^\frac{1}{q^*}}\leq \rho (x^*)\leq 1.\]
    Therefore, we obtain that
    \begin{align*}
        |x^*|\intoo[3]{\intoo[3]{\sum_{i=1}^n |x_i|^{q}}^\frac{1}{q}}&\leq \sum_{i=1}^n |x_i^*|(|x_i|)+\eps \leq \sum_{i=1}^n \rho(x_i^*)\tau(x_i)+\eps\\
        &\leq \intoo[3]{\sum_{i=1}^n \rho(x^*_i)^{p^*}}^\frac{1}{p^*} \intoo[3]{\sum_{i=1}^n \tau(x_i)^{p}}^\frac{1}{p}+\eps \leq  \intoo[3]{\sum_{i=1}^n \tau(x_i)^{p}}^\frac{1}{p}+\eps.
    \end{align*}    
    Making $\eps$ tend to $0$ and taking the supremum over all $x^*\in A$ we conclude that 
    \[\tau \intoo[3]{\intoo[3]{\sum_{i=1}^n |x_i|^{q}}^\frac{1}{q}}\leq \intoo[3]{\sum_{i=1}^n \tau(x_i)^{p}}^\frac{1}{p},\]
    as we wanted to show.\\
    
    To conclude the proof, let $Z$ be a $(p,q)$-convex Banach lattice, $\phi:Z\rightarrow X$ a class $\Ccal$ operator and $S:E\rightarrow Z$ an operator, so that $T=\phi S$. Let us write $K=K^{(p,q)}(Z)$ and, without loss of generality, assume that $S$ is contractive. We claim that $(K\phi(B_Z))^\circ \subseteq A$. Indeed, assume that $x^*\notin A$, i.e., $\rho(x^*)>1$. Then, there exists a sequence $(x_i^*)_{i=1}^n\subseteq X^*$ such that  
    \[\intoo[3]{\sum_{i=1}^n |x^*_i|^{q^*}}^\frac{1}{q^*} \leq |x^*|\quad \text{but} \quad \intoo[3]{\sum_{i=1}^n \|T^*x^*_i\|^{p^*}}^\frac{1}{p^*} >1.\]
    We can find $(u_i)_{i=1}^n\subseteq E$ such that $\intoo{\sum_{i=1}^n \|u_i\|^{p}}^\frac{1}{p}\leq 1 $ and
    \begin{align*}
        1&< \sum_{i=1}^n T^*x^*_i(u_i)=\sum_{i=1}^n x^*_i(Tu_i)\leq \intoo[3]{\sum_{i=1}^n |x^*_i|^{q^*}}^\frac{1}{q^*}\intoo[3]{\intoo[3]{\sum_{i=1}^n |Tx_i|^{q}}^\frac{1}{q}}\\
        &\leq |x^*|\intoo[3]{\intoo[3]{\sum_{i=1}^n |Tx_i|^{q}}^\frac{1}{q}} = |x^*|\intoo[3]{\phi \intoo[3]{\sum_{i=1}^n |Sx_i|^{q}}^\frac{1}{q}},
    \end{align*}
    where we have used that $\phi$ is a lattice homomorphism, combined with \Cref{prop: Holder in BL}. Note that the $(p,q)$-convexity of $Z$ and the contractivity of $S$ yield that $\intoo{\sum_{i=1}^n |Sx_i|^{q}}^\frac{1}{q} \in KB_Z$, so $\phi \intoo{\sum_{i=1}^n |Sx_i|^{q}}^\frac{1}{q} \in K\phi(B_Z)$. It follows that $x^*\notin (K^{(p,q)}(Z) \phi(B_Z))^\circ$, as we claimed. Taking polars and noting that $\phi(B_Z)$ is closed and convex, so $\phi(B_Z) = (\phi(B_Z))^{\circ\circ}$, we obtain that
    \[\phi_0(B_{Z_0})=A^\circ \subseteq (K  \phi(B_Z))^{\circ\circ} =K \phi(B_Z).\]
    Thus, for every $x\in Z_0$, there exists $z\in K \|x\| B_Z$ such that $\phi_0 x=\phi z$. We define $\psi:Z_0\rightarrow Z$ as $\psi x= z$. Clearly, $\psi$ is an injective interval preserving lattice homomorphism that satisfies $\phi_0=\phi \psi$ and $S=\psi S_0$. Finally, if $(x_n)_n\subseteq B_{Z_0}$ is a sequence such that $\psi x_n\rightarrow z$, it follows that $\phi_0 x_n=\phi\psi x_n\rightarrow \phi z$. Since $\phi_0(B_{Z_0})$ is closed, we conclude that $\phi z=\phi_0x$ for some $x\in B_{Z_0}$. Therefore, $z=\psi x$ and $\psi(B_{Z_0})$ is closed in $Z$, so $\psi$ is of class $\Ccal$ and the proof is concluded.
\end{proof}

\begin{rem}
    Again, observe that when $E$ is a Banach lattice and $T$ is positive (respectively, a lattice homomorphism), then $S_0$ can be made positive (respectively, a lattice homomorphism). 
\end{rem}

\begin{cor}\label{cor: factorization convex}
    Let $X$ be a Banach lattice, $E$ a Banach space and $1\leq p\leq q\leq \infty$. Then, the operator $T:E\rightarrow X$ is $(p,q)$-convex if and only if there exists a $(p,q)$-convex Banach lattice $Z$, a positive operator $\phi: Z\rightarrow X$ and an operator $S:E\rightarrow Z$ such that $T=\phi  S$.
\end{cor}

\begin{proof}
    Again, one implication is just \Cref{thm: factorization convex}. To obtain the converse, assume that $Z$, $\phi$ and $S$ are as in the statement. Since $\phi$ is positive, by \Cref{lem: positive operators and functional calculus}
    \begin{align*}
       \norm[3]{ \intoo[3]{\sum_{i=1}^n |Tu_i|^{q}}^\frac{1}{q}} & \leq  \norm[3]{ \phi \intoo[3]{\sum_{i=1}^n |Su_i|^{q}}^\frac{1}{q}}\leq \|\phi\|  \norm[3]{ \intoo[3]{\sum_{i=1}^n |Su_i|^{q}}^\frac{1}{q}}\\
       & \leq \|\phi\| K^{(p,q)}(Z) \intoo[3]{\sum_{i=1}^n \|Su_i\|^{p}}^\frac{1}{p} \leq \|\phi\|  \|S\| K^{(p,q)}(Z) \intoo[3]{\sum_{i=1}^n \|u_i\|^{p}}^\frac{1}{p}
    \end{align*}
    for every $u_1,\ldots, u_n \in E$, so $T$ is $(p,q)$-convex.
\end{proof}

Theorems \ref{thm: factorization concave} and \ref{thm: factorization convex} provide canonical factorizations for $(p,q)$-concave and $(p,q)$-convex operators, respectively. As we have seen, these factorizations are, respectively, maximal and minimal within the appropriate classes of factorizations. It turns out that the canonical factorizations associated with an operator and its adjoint satisfy a partial duality relation. Consider, for instance, a $(p,q)$-convex operator $T:E\rightarrow X$, and let us denote by the triple $(Z_0,S_0,\phi_0)$ the minimal factorization of $T$ given by \Cref{thm: factorization convex}. By \Cref{thm: duality convexity concavity}, $T^*:X^*\rightarrow E^*$ is $(p^*,q^*)$-concave, so let $(Z^0,\phi^0,S^0)$ be the maximal factorization of $T^*$ obtained in \Cref{thm: factorization concave}. Recall that the construction yields that $Z_0$ is an ideal of $X$ and $Z^0$ is the completion of the quotient of $X^*$ by the kernel of a lattice seminorm $\rho$, so that $\phi_0$ is just the formal identity, and $\phi^0$ is the composition of the quotient map with the inclusion into the completion. In this setting, the following duality relation holds. We omit the proof and refer the reader to \cite{RT} and \cite{GLTT2} for details.

\begin{thm}\label{thm: pseudo duality for factorization of convex operator}
Let $T:E\to X$ be a $(p,q)$-convex operator. Let $(Z_0,S_0,\phi_0)$ be the minimal factorization of $T$ and let $(Z^0,\phi^0,S^0)$ be the maximal factorization of $T^*$. Let the operators $U$ and $V$ be determined by 
\begin{align*}
U: Z_0 \to (Z^0)^*, \quad & (Uz)(\phi^0x^*) = x^*(\phi_0 z),\\
V: Z^0\to (Z_0)^*, \quad  & (V \phi^0x^*)(z)= x^*(\phi_0 z),
\end{align*}
for every $z\in Z_0$ and $x^*\in X^*$. Then $U$ and $V$ are lattice isometric embeddings. Moreover, $V$ is of class $\Ccal$, and, if $X$ has an order continuous norm, then $U$ is of class $\Ccal$.  
\end{thm}

A similar result holds for $(p,q)$-concave operators and their adjoints. Let $T:X\to E$ be a $(p,q)$-concave operator. Let the triple $(Z^0,\phi^0,S^0)$ denote the maximal factorization of $T$ obtained in \Cref{thm: factorization concave} and $(Z_0,S_0,\phi_0)$ the minimal factorization of $T^*$ given by \Cref{thm: factorization convex}. Then, $Z^0$ is the completion of a quotient of $X$ and $\phi^0$ is the composition of the quotient map and the inclusion into the completion. On the other hand, $Z_0$ is an ideal of $X^*$ and $\phi_0$ is the formal inclusion.

\begin{thm}\label{thm: pseudo duality for factorization of concave operator}
Let $T:X\to E$ be a $(p,q)$-concave operator. Let $(Z^0,\phi^0,S^0)$ be the maximal factorization of $T$ and let $(Z_0,S_0,\phi_0)$ be the minimal factorization of $T^*$. Let the operators $U$ and $V$ be determined by 
\begin{align*}
U: Z^0 \to (Z_0)^*, \quad & (U\phi^0x)(z^*)=(\phi_0 z^*)(x) ,\\
V: Z_0\to (Z^0)^*, \quad  & (Vz^*)(\phi^0x)= (\phi_0 z^*)(x) ,
\end{align*}
for every $z^*\in Z_0$ and $x\in X$. Then $U$ and $V$ are lattice isometric embeddings. Moreover, $V$ is of class $\Ccal$, and, if $X$ has an order continuous norm, then $U$ is of class $\Ccal$.  
\end{thm}

The applications of the factorization techniques introduced in this subsection go far beyond Theorems \ref{thm: factorization concave} and \ref{thm: factorization convex}. For instance, they can be used to establish factorization results for operators between Banach lattices which are simultaneously $(p,p')$-convex and $(q,q')$-concave, for any of the four possible choices of $p'\in \{p,\infty\}$ and $q'\in \{1,q\}$ (see \cite[Theorem 15]{RT} and \cite[Section 3.5]{GLTT2}). 

\section{Factorizations through function spaces}\label{sec: concrete factorizations}
\subsection{The factorization theorems of Krivine, Maurey--Nikishin and Pisier}\label{sec: factorization function spaces}
We switch now to more specific situations in which the operators factor not only through a generic convex or concave Banach lattice, but through model spaces such as $L_p$, $L_{p\infty}$ or $L_{p,1}$. A specially relevant case of this situation is the following result due to Krivine \cite[Théorème 2]{KrivineExp2223}.

\begin{thm}\label{thm: factorization Krivine}
    Let $E$ and $F$ be Banach spaces, $X$ a Banach lattice, $1\leq p<\infty$, $U:E\rightarrow X$ a $p$-convex operator and $V:X\rightarrow F$ a $p$-concave operator. Then, there exists a measure $\mu$  and bounded operators $\widetilde{U}:E\rightarrow  L_p(\mu)$ and $\widetilde{V}:L_p(\mu)\rightarrow F$ such that $\widetilde{V} \widetilde{U}=V U$, $\|\widetilde{U}\|\leq K^{(p)}(U)$ and $\|\widetilde{V}\|\leq K_{(p)}(V)$.
\end{thm}

\begin{proof}
    It suffices to show the following claim:\\
    
    \textbf{Claim:} Every contractive lattice homomorphism $T: Y\rightarrow Z$, where $Y$ is $p$-convex and $Z$ is $p$-concave, both with constant one, factors through $L_p(\mu)$ for some measure, with the factors being contractive lattice homomorphisms.\\
    
    Indeed, by Corollaries \ref{cor: factorization concave} and \ref{cor: factorization convex} there exists a $p$-convex Banach lattice $Y$ and a $p$-concave Banach lattice $Z$, both with constant one, lattice homomorphisms $\psi:Y\rightarrow X$ and $\phi:X\rightarrow Z$ and operators $R:E\rightarrow Y$ and $S: Z\rightarrow F$, such that the following diagram commutes, where $T=\phi \psi$:
    \[\xymatrix{
		E  \ar@{->}[rd]_{R} \ar@{->}[rr]^{U} & & X \ar@{->}[rd]_{\phi} \ar@{->}[rr]^{V} & & F  \\
		& Y \ar@{->}[ru]_{\psi} \ar@{->}[rr]_{T} & & Z \ar@{->}[ru]_{S} &
    }\]
    Moreover, by replacing the norm in $Y$ by $K^{(p)}(U)$ times the same norm, we get that $\|R\|\leq K^{(p)}(U)$ and $\psi$ is contractive. Similarly, we can assume that $\|S\|\leq K_{(p)}(V)$ and $\phi$ is contractive. Therefore, we just need to factor the contractive lattice homomorphism $T=\phi \psi$ through some $L_p(\mu)$ space. \\
    
    \textit{Proof of the claim.} First we need to reduce $T:Y\rightarrow Z$ to an injective lattice homomorphism. It is clear that $\ker T$ is a closed ideal of $Y$, so $\widetilde{Y}=Y/\ker T$ is a Banach lattice and $T$ can be factored as $T=\widetilde{T} Q$, where $Q:Y\rightarrow \widetilde{Y}$ is the canonical lattice quotient and $\widetilde{T}:\widetilde{Y}\rightarrow Z$ is an injective and contractive lattice homomorphism. Moreover, $\widetilde{Y}$ is $p$-convex with constant one. Next, we apply a $p$-concavification procedure (see \Cref{prop: p-concavification}) to both $\widetilde{Y}$ and $Z$. Since $\widetilde{Y}$ is $p$-convex with constant one, $\widetilde{Y}_{(p)}$ is a Banach lattice with the norm $\normiii{\cdot}=\normiii{\cdot}_0=\|\cdot\|^p$. In the case of $Z$, we cannot find \textit{a priori} a lattice norm in the vector lattice $Z_p$ due to the lack of convexity assumptions on $Z$. However, $\normiii{\cdot}=\|\cdot\|^p$ induces a lattice quasi-norm in $Z_p$, meaning that it is positively homogeneous, it vanishes only at zero, $\normiii{x}\leq \normiii{z}$ for every $x,z\in Z$ with $|x|\leq |z|$, and
    \[\normiii{x\oplus_p z}\leq 2^\frac{p}{p^*}(\normiii{x}+\normiii{z})\]
     for every $x,z\in Z$. Moreover, $\normiii{\cdot}$ is complete, and the $p$-concavity of $Z$ implies that $\normiii{\cdot}$ satisfies that for every $z_1,\ldots, z_n\in Z$
     \[\normiii{|z_1|\oplus_p\ldots\oplus_p |z_n|}=\norm[3]{\intoo[3]{\sum_{i=1}^n |z_i|^{p}}^\frac{1}{p}}^p\geq \sum_{i=1}^n \|z_i\|^{p}= \sum_{i=1}^n \normiii{z_i}. \]
     By \Cref{lem: lattice homomorphism with new operations}, $\widetilde{T}:\widetilde{Y}_{(p)}\rightarrow (Z_p,\normiii{\cdot})$ is lattice homomorphism with norm $\normiii{\widetilde{T}}=\|\widetilde{T}\|^p\leq 1$. We define in $\widetilde{Y}_{(p)}$ the sets
     \[C_1=\{y\in \widetilde{Y}_{(p)}: \normiii{y}<1\}\quad \text{and}\quad C_2= \{y\in \widetilde{Y}_{(p)}: \widetilde{T}y\geq 0, \normiii{\widetilde{T}y}\geq \normiii{\widetilde{T}}\}. \]
     $C_1$ is clearly convex and open, as it is the open unit ball of $\widetilde{Y}_{(p)}$, and $C_2$ is convex due to the $1$-concavity of $\normiii{\cdot}$ in $Z_p$. Moreover, $C_1\cap C_2=\varnothing$, so by Hahn--Banach Theorem there exists a functional $y^*\in (\widetilde{Y}_{(p)})^*$ such that $y^*(y)\leq 1$ for every $y\in C_1$ and $y^*(y)\geq 1$ for every $y\in C_2$. From this we learn that $y^*\in B_{(\widetilde{Y}_{(p)})^*}$, and that $y^*$ is positive. Indeed, if $0<y\in \widetilde{Y}_{(p)}$, then $\widetilde{T}y>0$, since $\widetilde{T}$ is positive and injective, so $\normiii{\widetilde{T}}y/\normiii{\widetilde{T}y}\in C_2$ and thus $y^*(y)\geq \normiii{\widetilde{T}y}/\normiii{\widetilde{T}}>0$. It follows that
     \[\frac{\normiii{\widetilde{T}y}}{\normiii{\widetilde{T}}}\leq y^*(|y|)\leq \normiii{y}\]
     for every $y\in \widetilde{Y}_{(p)}$, so $\sigma(y)=y^*(|y|)$ is a lattice seminorm that induces an $AL$-norm in $\widetilde{Y}_{(p)}/\ker \sigma$. By Kakutani's Representation Theorem for $AL$-spaces, the completion of this quotient can be isometrically identified with $L_1(\mu)$ for some measure $\mu$. Let $P:\widetilde{Y}_{(p)} \rightarrow L_1(\mu)$ be the composition of the quotient, the inclusion into the completion and the isometry into $L_1(\mu)$, and $\widehat{T}:L_1(\mu)\rightarrow (Z_p,\normiii{\cdot})$ be the extension to $L_1(\mu)$ of the lattice homomorphism induced on $\widetilde{Y}_{(p)}/\ker \sigma$ by $\widetilde{T}$, so that $\widetilde{T}=\widehat{T} P$. Applying now a $p$-convexification procedure to each space we recover $\widetilde{Y}$ and $Z$ with the same norm and vector lattice structure, and we can lattice isometrically identify $L_1(\mu)^{(p)}$ with $L_p(\mu)$. Morevoer, by \Cref{lem: lattice homomorphism with new operations}, $P:\widetilde{Y} \rightarrow L_p(\mu)$ and $\widehat{T}: L_p(\mu)\rightarrow Z$ are contractive lattice homomorphisms such that $\widetilde{T}=\widehat{T} P$. It follows that $T=\widetilde{T} Q=\widehat{T} P Q$ factors through $L_p(\mu)$, as we wanted to show.
\end{proof}

We continue this section with some stronger factorizations when $L_r(\mu)$ is the target or the domain space. We will need the following application of Dunford--Pettis Theorem (cf. \cite[Theorem 5.2.8]{AK}), that characterizes relatively $w$-compact sets of $L_1(\mu)$.

\begin{lem}\label{lem: Dunford Pettis}
    Let $(\Omega,\Sigma, \mu)$ be a measure space, and $(g_n)_{n=1}^\infty\subseteq L_1(\mu)_+$ be a normalized sequence such that there is a constant $K>0$ and $1<q<\infty$ satisfying that for every finite choice of scalars $\lambda_1,\ldots,\lambda_m$ we have
    \[\norm[3]{\bigvee_{n=1}^m |\lambda_n g_n|}_{L_1(\mu)}\leq K \intoo[3]{\sum_{n=1}^m |\lambda_n|^{q}}^\frac{1}{q}. \]
    Then, $(g_n)_{n=1}^\infty$ is relatively weakly compact.
\end{lem}

\begin{proof}
    Without loss of generality, we can assume that $\mu$ is a probability measure. Indeed, the function $e=\sum_{n=1}^\infty 2^{-n}g_n \in L_1(\mu)_+$ has norm one, so the measure $e\cdot \mu$ is a probability measure, and satisfies that the band generated by $e$, $B_e$, coincides with the band generated by the whole sequence $(g_n)_{n=1}^\infty$. Radon--Nikodym Theorem establishes that $B_e$ is lattice isometric to $L_1(e\cdot \mu)$, so $(g_n)_{n=1}^\infty$ satisfies the inequality of the statement in $L_1(\mu)$ if and only if $(g_n/e)_{n=1}^\infty$ satisfies the corresponding property in $L_1(e\cdot \mu)$. Since linear operators are weak-to-weak continuous, it suffices to show that $(g_n/e)_{n=1}^\infty$ is relatively weakly compact in $L_1(e\cdot\mu)$ in order to conclude that $(g_n)_{n=1}^\infty$ is relatively weakly compact in $L_1(\mu)$.\\
    
    Hereby, let $\mu$ be a probability measure. First, we show that for every sequence of pairwise disjoint measurable subsets $(B_n)_{n=1}^\infty\subseteq \Omega$ we have
    \[\lim_n \int_{B_n} g_n\,d\mu =0.\]
    To do so, we define $\lambda_n=\intoo{\int_{B_n} g_n\,d\mu}^{q^*-1}$, so for every $m\in \N$ we have
    \begin{align*}
        \sum_{n=1}^m \intoo[3]{\int_{B_n} g_n\,d\mu}^{q^*} & = \sum_{n=1}^m \lambda_n \int_{B_n} g_n\,d\mu \leq \sum_{n=1}^m  \int_{B_n} \bigvee_{k=1}^m |\lambda_k g_k|\,d\mu \leq  \int_{\Omega} \bigvee_{k=1}^m |\lambda_k g_k|\,d\mu \\
        &\leq K  \intoo[3]{\sum_{n=1}^m |\lambda_n|^{q}}^\frac{1}{q} = K \intoo[3]{\sum_{n=1}^m \intoo[3]{\int_{B_n} g_n\,d\mu}^{q^*}}^\frac{1}{q}.
    \end{align*}
    Rearranging the terms we conclude that $\sum_{n=1}^\infty \intoo{\int_{B_n} g_n\,d\mu}^{q^*}\leq K^{q^*}$, so in particular 
    \[\lim_n \int_{B_n} g_n\,d\mu =0.\]
    Next we prove that for every sequence of measurable subsets $(C_n)_{n=1}^\infty\subseteq \Omega$ such that $\lim_n\mu(C_n)=0$ we have
    \[\lim_n \int_{C_n} g_n\,d\mu =0.\]
    Assume the contrary. Then, passing to a subsequence we can find $\eps>0$ and a sequence $(C_n)_{n=1}^\infty\subseteq \Omega$ such that $\lim_n\mu(C_n)=0$ but $\int_{C_n} g_n\,d\mu \geq \eps$. Note that, for every fixed $m\in \N$, $g_m\chi_{C_n}\xrightarrow[n]{}0$ $\mu$-a.e., so by the Dominated Convergence Theorem $\lim_n \int_{C_n} g_m\,d\mu =0$. Let $n_1=1$ and $j\geq 2$, and assume that $n_1<\ldots<n_{j-1}$ have been selected. We can choose $n_j>n_{j-1}$ such that $\int_{C_{n_j}} g_{n_i}\,d\mu <2^{-j}\eps$ for every $i=1,\ldots, j-1$. Let $B_i=C_{n_i}\setminus \bigcup_{j=i+1}^\infty C_{n_j}$. If $i<j$, $B_i\cap B_j\subseteq B_i\cap C_{n_j}= \varnothing$, so the new collection is pairwise disjoint. Moreover,
    \[\int_{B_i} g_{n_i}\,d\mu \geq \int_{C_{n_i}} g_{n_i}\,d\mu-\int_{\bigcup_{j=i+1}^\infty C_{n_j}} g_{n_i}\,d\mu \geq \int_{C_{n_i}} g_{n_i}\,d\mu-\sum_{j=i+1}^\infty\int_{C_{n_j}} g_{n_i}\,d\mu \geq \frac{\eps}{2}.\]
    This contradicts the fact that $\lim_n \int_{B_i} g_{n_i}\,d\mu =0$.\\

    This last property implies that $(g_n)_{n=1}^\infty$ is an equi-integrable set, i.e., for every $\eps>0$ there exists some $\delta>0$ such that $\sup_m \int_{B} g_{m}\,d\mu<\eps$ for every measurable subset $C$ with $\mu(C)<\delta$. Otherwise, we can find some $\eps>0$ and a sequence of subsets $(C_n)_{n=1}^\infty$ such that $\mu(C_n)\leq \frac{1}{n}$ but $\sup_m \int_{C_n} g_{m}\,d\mu>\eps$ for every $n\in \N$. For each $n$ there is some $m_n\in \N$ such that $\int_{C_n} g_{m_n}\,d\mu>\eps$. The fact that $\lim_n \int_{C_n} g_m\,d\mu =0$ for every $m\in \N$ implies that the set $\{m_n:n\in \N\}$ contains an increasing subsequence. We can assume that $(m_n)_{n=1}^\infty$ itself is strictly increasing, so by the above property $\lim_n \int_{C_n} g_{m_n}\,d\mu =0$, which leads to a contradiction. Therefore, $(g_n)_{n=1}^\infty$ is equi-integrable, so by \cite[Theorem 5.2.8]{AK} it is relatively weakly bounded.
\end{proof}

The following factorization theorem for $p$-convex operators was proven by Maurey \cite[Théorème 2]{MaureyFactorization}, using previous ideas of Nikishin \cite{Nikishin}. It asserts that every $p$-convex operator from a Banach space into $L_r(\mu)$, $1\leq r<p$, can be factored through $L_p(g\cdot\mu)$ for some change of density $g\in L_1(\mu)$. Note that the fact that such an operator factors through some $L_p(\nu)$ follows as a particular case of \Cref{thm: factorization Krivine} when $X=F=L_r(\mu)$ and $V$ is the identity operator. However, \Cref{thm: factorization Maurey set} provides a refinement, since we get additional information about the measure $\nu$, namely, that it is a probability measure defined on the same measure space as $\mu$, and $\nu$ is absolutely continuous with respect to $\mu$. We reproduce a proof suggested by Pisier in \cite{Pisier86} that works similarly for both the $p$-convex and $(p,\infty)$-convex setting, although a proof that remains closer to the original one can be found in \cite[Theorem 7.1.2]{AK}.

\begin{thm}\label{thm: factorization Maurey set}
    Let $1\leq r<p<\infty$, $A\subseteq L_r(\mu)$ and $C>0$. The following are equivalent:
    \begin{enumerate}
        \item For every $f_1,\ldots,f_n\in A$ and $\alpha_1,\ldots,\alpha_n\in \R$,
        \[\norm[3]{ \intoo[3]{\sum_{i=1}^n |\alpha_i f_i|^{p}}^\frac{1}{p}}_{L_r(\mu)}\leq C \intoo[3]{\sum_{i=1}^n |\alpha_i |^{p}}^\frac{1}{p}.\]
        \item There exists a normalized $g\in L_1(\mu)_+$ such that for every $f\in A$,
        \[\{g=0\}\subseteq \{f=0\}\quad \text{and}\quad \|g^{-\frac{1}{r}}f\|_{L_p(g\cdot\mu)}\leq C.\]
    \end{enumerate}
\end{thm}

Even though at first sight this theorem might not look like a factorization result, it can be restated in terms of operators and factorizations, as we will see in \Cref{cor: factorization Maurey operator}. Nevertheless, this way of stating it allows one for more general applications (see for instance \cite[Theorem 4.7]{GLTT}, where it is used to factor simultaneously two operators through the same $L_p$ space).

\begin{proof}
    $(2)\Rightarrow (1)$ Let us fix $f_1,\ldots,f_n\in A$ and $\alpha_1,\ldots,\alpha_n\in \R$. Then, using Hölder's inequality with exponent $\frac{p}{r}$ we get
    \begin{align*}
        \norm[3]{ \intoo[3]{\sum_{i=1}^n |\alpha_i f_i|^{p}}^\frac{1}{p}}_{L_r(\mu)} & = \intoo[3]{\int \intoo[3]{\sum_{i=1}^n |\alpha_i f_ig^{-\frac{1}{r}}|^{p}}^\frac{r}{p} g\, d\mu}^\frac{1}{r}\\
        & \leq \intoo[3]{\int \sum_{i=1}^n |\alpha_i f_i g^{-\frac{1}{r}}|^{p} g\, d\mu}^\frac{1}{p} \intoo[3]{\int  g\, d\mu}^{\frac{1}{r} -\frac{1}{p}}  \\
        & = \intoo[3]{ \sum_{i=1}^n |\alpha_i|^p \|f_i g^{-\frac{1}{r}}\|_{L_p(g\cdot \mu)}^{p}}^\frac{1}{p}  \leq C \intoo[3]{\sum_{i=1}^n |\alpha_i |^{p}}^\frac{1}{p}.
    \end{align*}
    $(1)\Rightarrow (2)$ Let us fix $n\in \N$ and define
    \[C_n=\sup\cbr[3]{\norm[3]{ \intoo[3]{\sum_{i=1}^n |\alpha_i f_i|^{p}}^\frac{1}{p}}_{L_r(\mu)}: f_1,\ldots,f_n\in A, \sum_{i=1}^n |\alpha_i |^{p}\leq 1}.\]
    Clearly $C_n\uparrow C$ and $C_n\neq 0$ as long as $A\neq \{0\}$. We also select a sequence of scalars $\delta_n\downarrow 1$. Then, we can choose $f_{n,1},\ldots, f_{n,n}\in A$ and $\alpha_{n,1},\ldots, \alpha_{n,n}\in \R$ such that
    \[\norm[3]{ \intoo[3]{\sum_{i=1}^n |\alpha_{n,i} f_{n,i}|^{p}}^\frac{1}{p}}_{L_r(\mu)}=1 \quad \text{and}\quad \intoo[3]{\sum_{i=1}^n |\alpha_{n,i} |^{p}}^\frac{1}{p}\leq \frac{\delta_n}{C_n}.\]
    Let $g_n= \intoo[3]{\sum_{i=1}^n |\alpha_{n,i} f_{n,i}|^{p}}^\frac{r}{p}$, which belongs to $L_1(\mu)_+$ and has norm one. Then, the sequence $(g_n)_{n=1}^\infty$ satisfies that for any finite choice of scalars $\lambda_1,\ldots,\lambda_m\in \R$,
    \begin{align*}
        \norm[3]{\bigvee_{n=1}^m |\lambda_n g_n|}_{L_1(\mu)}& \leq \norm[3]{\intoo[3]{\sum_{n=1}^m |\lambda_n g_n|^{\frac{p}{r}}}^\frac{r}{p}}_{L_1(\mu)}= \norm[3]{\intoo[3]{\sum_{n=1}^m \sum_{i=1}^n  |\lambda_n|^{\frac{p}{r}} |\alpha_{n,i} f_{n,i}|^{p} }^\frac{1}{p}}_{L_r(\mu)}^r \\
        & \leq C^r \intoo[3]{\sum_{n=1}^m \sum_{i=1}^n  |\lambda_n|^{\frac{p}{r}} |\alpha_{n,i}|^{p} }^\frac{r}{p} \leq C^r \intoo[3]{\sum_{n=1}^m |\lambda_n|^{\frac{p}{r}} \frac{\delta_n^p}{C_n^p} }^\frac{r}{p} \\
        & \leq \intoo[3]{\frac{C\delta_1}{C_1}}^r  \intoo[3]{\sum_{n=1}^m |\lambda_n|^{\frac{p}{r}} }^\frac{r}{p}, 
    \end{align*}
    so by \Cref{lem: Dunford Pettis} $(g_n)_{n=1}^\infty$ is relatively weakly compact in $L_1(\mu)$. The Eberlein--\v Smulian Theorem (cf. \cite[Theorem 1.6.3]{AK}) implies that there exists a subsequence $(g_{n_m})_{m=1}^\infty$ that converges weakly to some $g\in L_1(\mu)$. Clearly, $g\geq 0$ and $\|g\|_{L_1(\mu)}=1$. Next, we observe that for any $f\in A$, $\eps>0$ and $n\in \N$, we have
    \[\int (g_n^\frac{p}{r}+(\eps C_n^{-1} |f|)^p)^\frac{r}{p}\, d\mu= \norm[3]{ \intoo[3]{\sum_{i=1}^n |\alpha_{n,i} f_{n,i}|^{p} + (\eps C_n^{-1} |f|)^p }^\frac{1}{p}}_{L_r(\mu)}^r \leq \intoo[3]{\frac{C_{n+1}}{C_n}}^r(\delta_n^p+\eps^p)^\frac{r}{p}. \]
    If we denote by $q=\frac{p}{r}$, then we know by \Cref{prop: p-sum as supremum} that
    \[(g_n^\frac{p}{r}+(\eps C_n^{-1} |f|)^p)^\frac{r}{p}=\sup \cbr[3]{\bigvee_{j=1}^k (a_j g_n+b_j (\eps C_n^{-1} |f|)^r): a_j^{q^*}+b_j^{q^*}\leq 1, j=1,\ldots, k, k\in \N},\]
    and similarly,
    \[(g^\frac{p}{r}+(\eps C^{-1} |f|)^p)^\frac{r}{p}=\sup \cbr[3]{\bigvee_{j=1}^k (a_j g+b_j (\eps C^{-1} |f|)^r): a_j^{q^*}+b_j^{q^*}\leq 1, j=1,\ldots, k, k\in \N}.\]
    In particular, for any choice of scalars such that $a_j^{q^*}+b_j^{q^*}\leq 1$, $j=1,\ldots, k$, we can define the sets
    \[A_i=\cbr[3]{\bigvee_{j=1}^k (a_j g+b_j (\eps C^{-1} |f|)^r)= a_i g+b_i (\eps C^{-1} |f|)^r}\]
    for $i=1,\ldots, k$, that can be assumed to be pairwise disjoint without loss of generality. Recall that $(g_{n_m})_{m=1}^\infty$ converges weakly to $g$, and clearly $(\eps C_{n_m}^{-1} |f|)^r$ converges weakly to $(\eps C^{-1} |f|)^r$ in $L_1(\mu)$, so 
    \begin{align*}
        \int \bigvee_{j=1}^k (a_j g+b_j (\eps C^{-1} |f|)^r) \, d\mu & = \int \sum_{j=1}^k (a_j\chi_{A_j} g+b_j\chi_{A_j} (\eps C^{-1} |f|)^r) \, d\mu \\
        & = \lim_m \int \sum_{j=1}^k (a_j\chi_{A_j} g_{n_m}+b_j\chi_{A_j} (\eps C_{n_m}^{-1} |f|)^r) \, d\mu.
    \end{align*}
    Moreover, for every $m\in \N$,
    \begin{align*}
        &\int \sum_{j=1}^k (a_j\chi_{A_j} g_{n_m}+b_j\chi_{A_j} (\eps C_{n_m}^{-1} |f|)^r) \, d\mu  \leq  \int \bigvee_{j=1}^k (a_j g_{n_m}+b_j (\eps C_{n_m}^{-1} |f|)^r) \, d\mu\\
        & \leq \int (g_{n_m}^\frac{p}{r}+(\eps C_n^{-1} |f|)^p)^\frac{r}{p} \, d\mu  \leq \intoo[3]{\frac{C_{n_m+1}}{C_{n_m}}}^r(\delta_{n_m}^p+\eps^p)^\frac{r}{p},
    \end{align*}
    so 
    \[\int \bigvee_{j=1}^k (a_j g+b_j (\eps C^{-1} |f|)^r) \, d\mu \leq \lim_m \intoo[3]{\frac{C_{n_m+1}}{C_{n_m}}}^r(\delta_{n_m}^p+\eps^p)^\frac{r}{p} = (1+\eps^p)^\frac{r}{p} .\]
    Since $L_1(\mu)$ has the Fatou property ($h_\alpha \uparrow h$ implies $\|h_\alpha\|\uparrow \|h\|$ for every net $(h_\alpha)_\alpha$ in $L_1(\mu)_+$), it follows that
    \[\int (g^\frac{p}{r}+(\eps C^{-1} |f|)^p)^\frac{r}{p} \, d\mu \leq (1+\eps^p)^\frac{r}{p}\]
    (recall that $f\in A$ and $\eps>0$ are arbitrary). From here, we deduce that
    \begin{align*}
        1+\eps^r C^{-r} \int_{\{g=0\}}|f|^r \, d\mu & = \int_{\{g>0\}} g \, d\mu + \eps^r C^{-r} \int_{\{g=0\}}|f|^r \, d\mu  \\
        & \leq \int (g^\frac{p}{r}+(\eps C^{-1} |f|)^p)^\frac{r}{p} \, d\mu \leq (1+\eps^p)^\frac{r}{p},
    \end{align*}
    so
    \[0\leq \int_{\{g=0\}}|f|^r \, d\mu \leq C^r \lim_{\eps \rightarrow 0} \frac{(1+\eps^p)^\frac{r}{p}-1}{\eps^r}=0\]
    and hence $\{g=0\}\subseteq \{f=0\}$ for every $f\in A$. On the other hand, it is clear that 
    \[\lim_{\eps\rightarrow 0}\frac{(1+\eps^p (C^{-1} |f| g^{-\frac{1}{r}})^p)^\frac{r}{p}-1}{\eps^p}g =\frac{r}{p} (C^{-1} |f| g^{-\frac{1}{r}})^p g  \quad \mu\text{-a.e.},\]
    so applying Fatou's Lemma, we obtain that
    \begin{align*}
        \frac{r}{p} C^{-p} \int (|f| g^{-\frac{1}{r}})^p g \,d\mu \leq \liminf_{\eps\rightarrow 0} \frac{1}{\eps^p} \int ( (g^\frac{p}{r}+(\eps C^{-1} |f|)^p)^\frac{r}{p}-g)\,d\mu \leq \lim_{\eps\rightarrow 0} \frac{(1+\eps^p)^\frac{r}{p}-1}{\eps^p} = \frac{r}{p}.
    \end{align*}
    Therefore, we conclude that
    \[\|g^{-\frac{1}{r}}f\|_{L_p(g\cdot\mu)}\leq C\]
    for every $f\in A$.
\end{proof}

We can restate the previous result in the following way:

\begin{cor}\label{cor: factorization Maurey operator}
    Let $1\leq r<p<\infty$, $E$ a Banach space, $T:E\rightarrow L_r(\mu)$ a linear operator and $C>0$. The following are equivalent:
    \begin{enumerate}
        \item $T$ is $p$-convex with $K^{(p)}(T)\leq C$.
        \item There exists a normalized $g\in L_1(\mu)_+$ such that for every $x\in E$,
        \[\{g=0\}\subseteq \{Tx=0\}\quad \text{and}\quad \|g^{-\frac{1}{r}}Tx\|_{L_p(g\cdot\mu)}\leq C\|x\|.\]
        \item There exists a normalized $g\in L_1(\mu)_+$ such that $\{g=0\}\subseteq \{Tx=0\}$ for every $x\in E$, and $T=M R$, where $R:E\rightarrow L_p(g\cdot \mu)$ given by $Rx=g^{-\frac{1}{r}}Tx$ is bounded with $\|R\|\leq C$ and $M:L_p(g\cdot \mu)\rightarrow L_r(\mu)$ is the multiplication operator by $g^\frac{1}{r}$.
        \[\xymatrix{
		E  \ar@{->}[rd]_{R} \ar@{->}[rr]^{T} & & L_r(\mu)  \\
		 & L_p(g\cdot \mu) \ar@{->}[ru]_{M} & 
        }\]
    \end{enumerate}
\end{cor}

\begin{proof}
    $(1)\Leftrightarrow(2)$ Let $A=T(B_E)\subseteq L_r(\mu)$. Then, it is clear that conditions $(1)$ and $(2)$ from the statement are equivalent to $A$ satisfying conditions $(1)$ and $(2)$ of \Cref{thm: factorization Maurey set}, respectively.\\

    $(2)\Rightarrow(3)$ Since $p>r$ and $g\cdot \mu$ is a probability measure, the operator
    \[\fullfunction{M}{L_p(g\cdot \mu)}{L_r(\mu)}{h}{g^\frac{1}{r} h}\]
    always satisfies $\|Mh\|_{L_r(\mu)}=\|h\|_{L_r(g\cdot\mu)}\leq \|h\|_{L_p(g\cdot \mu)}$. On the other hand, the operator $R$ is bounded by $(2)$.\\

    $(3)\Rightarrow(2)$ It follows from the boundedness of $R$.
\end{proof}

A situation analogous to \Cref{thm: factorization Maurey set} takes place in the $(p,\infty)$-convex case. In this setting, Pisier \cite{Pisier86} showed that $L_{p,\infty}$ with the renorming $\|\cdot\|_{L^{[r]}_{p,\infty}}$ is the model space that replaces $L_p$ as the factoring space.

\begin{thm}\label{thm: factorization Pisier set}
    Let $1\leq r<p<\infty$, $A\subseteq L_r(\mu)$. The following are equivalent:
    \begin{enumerate}
        \item There exists $C>0$ such that for every $f_1,\ldots,f_n\in A$ and $\alpha_1,\ldots,\alpha_n\in \R$,
        \[\norm[3]{ \bigvee_{i=1}^n |\alpha_i f_i|}_{L_r(\mu)}\leq C \intoo[3]{\sum_{i=1}^n |\alpha_i |^{p}}^\frac{1}{p}.\]
        \item There exists $C'>0$ and a normalized $g\in L_1(\mu)_+$ such that for every $f\in A$ and measurable subset $B\subseteq \Omega$
        \[\|f\chi_B\|_{L_r(\mu)}\leq C' \intoo[3]{\int_B g\, d\mu}^{\frac{1}{r}-\frac{1}{p}}.\]
        \item There exists $C''>0$ and a normalized $g\in L_1(\mu)_+$ such that for every $f\in A$, 
        \[\{g=0\}\subseteq \{f=0\}\quad \text{and}\quad \|g^{-\frac{1}{r}}f\|_{L^{[r]}_{p,\infty}(g\cdot\mu)}\leq C''.\]
    \end{enumerate}
    Moreover, $C\leq C''\leq C'\leq (1-\frac{r}{p})^{\frac{1}{p}-\frac{1}{r}}C$.
\end{thm}

\begin{proof}
    $(2)\Rightarrow (3)$ Let us denote $\nu=g\cdot \mu$, which is a probability measure. For a fixed $f\in A$, we can consider the set $B=\{g=0\}\cap \{|f|>0\}$. If $\mu(B)$ was non-zero, we would have that
    \[0<\|f\chi_B\|_{L_r(\mu)}\leq C' \intoo[3]{\int_B g\, d\mu}^{\frac{1}{r}-\frac{1}{p}}=0,\]
    which is not possible. Therefore, $\{g=0\}\subseteq \{f=0\}$ (up to a set of measure zero). Now, given any measurable subset $B\subseteq \Omega$ with $\nu(B)=\int_B g\, d\mu >0$, we observe that
    \[\nu(B)^{\frac{1}{p}-\frac{1}{r}} \intoo[3]{\int_B |g^{-\frac{1}{r}} f|^r \,d\nu }^\frac{1}{r}=  \intoo[3]{\int_B g\, d\mu}^{\frac{1}{p}-\frac{1}{r}}  \|f\chi_B\|_{L_r(\mu)}\leq C', \]
    so taking the supremum over all such subsets $0<\nu(B)<\infty$ we obtain $(3)$.\\

    $(3)\Rightarrow (1)$ We first note that if $\nu$ is a probability measure, then $\|h\|_{L_r(\nu)}\leq \|h\|_{L^{[r]}_{p,\infty}(\nu)}$ for every $h\in L_{p,\infty}(\nu)$. Moreover, recall that $(L_{p,\infty}(\nu), \|\cdot \|_{L^{[r]}_{p,\infty}(\nu)})$ satisfies an upper $p$-estimate with constant one for any measure $\nu$. Thus, given $f_1,\ldots,f_n\in A$ and $\alpha_1,\ldots,\alpha_n\in \R$, since $g\cdot \mu$ is a probability measure we have that
    \begin{align*}
        \norm[3]{ \bigvee_{i=1}^n |\alpha_i f_i|}_{L_r(\mu)} & = \norm[3]{ \bigvee_{i=1}^n |\alpha_i f_i g^{-\frac{1}{r}}|}_{L_r(g\cdot \mu)} \leq \norm[3]{ \bigvee_{i=1}^n |\alpha_i f_i g^{-\frac{1}{r}}|}_{L^{[r]}_{p,\infty}(g\cdot \mu)}\\
        & \leq  \intoo[3]{\sum_{i=1}^n \|\alpha_i  f_i g^{-\frac{1}{r}}\|_{L^{[r]}_{p,\infty}(g\cdot \mu)}^{p}}^\frac{1}{p} \leq C \intoo[3]{\sum_{i=1}^n |\alpha_i |^{p}}^\frac{1}{p}.
    \end{align*}

    $(1)\Rightarrow (2)$ As in the proof of \Cref{thm: factorization Maurey set}, for every $n\in \N$ we define
    \[C_n=\sup\cbr[3]{\norm[3]{ \bigvee_{i=1}^n |\alpha_i f_i|}_{L_r(\mu)}: f_1,\ldots,f_n\in A, \sum_{i=1}^n |\alpha_i |^{p}\leq 1},\]
    that satisfy $C_n\uparrow C$ and $C_n\neq 0$ whenever $A\neq \{0\}$. Let $\delta_n\downarrow 1$ be a decreasing sequence of scalars, and find $f_{n,1},\ldots, f_{n,n}\in A$ and $\alpha_{n,1},\ldots, \alpha_{n,n}\in \R$ satisfying that
    \[\norm[3]{ \bigvee_{i=1}^n|\alpha_{n,i} f_{n,i}|}_{L_r(\mu)}=1 \quad \text{and}\quad \intoo[3]{\sum_{i=1}^n |\alpha_{n,i} |^{p}}^\frac{1}{p}\leq \frac{\delta_n}{C_n}.\]
    We denote by $g_n=\intoo{\bigvee_{i=1}^n|\alpha_{n,i} f_{n,i}|}^r\in L_1(\mu)_+$, that has norm one. For any finite choice of scalars $\lambda_1,\ldots,\lambda_m\in \R$, the sequence $(g_n)_{n=1}^\infty$ satisfies that
    \begin{align*}
        \norm[3]{\bigvee_{n=1}^m |\lambda_n g_n|}_{L_1(\mu)}& =\norm[3]{\bigvee_{n=1}^m |\lambda_n| \intoo{\bigvee_{i=1}^n|\alpha_{n,i} f_{n,i}|}^r}_{L_1(\mu)} =\norm[3]{\bigvee_{n=1}^m  \bigvee_{i=1}^n |\lambda_n|^\frac{1}{r}|\alpha_{n,i} f_{n,i}|}^r_{L_r(\mu)}\\
        & \leq C^r \intoo[3]{\sum_{n=1}^m \sum_{i=1}^n  |\lambda_n|^{\frac{p}{r}} |\alpha_{n,i}|^{p} }^\frac{r}{p} \leq C^r \intoo[3]{\sum_{n=1}^m |\lambda_n|^{\frac{p}{r}} \frac{\delta_n^p}{C_n^p} }^\frac{r}{p} \\
        & \leq \intoo[3]{\frac{C\delta_1}{C_1}}^r  \intoo[3]{\sum_{n=1}^m |\lambda_n|^{\frac{p}{r}} }^\frac{r}{p}. 
    \end{align*}
    \Cref{lem: Dunford Pettis} then yields that $(g_n)_{n=1}^\infty$ is relatively weakly compact in $L_1(\mu)$, so by the Eberlein--\v Smulian Theorem (cf. \cite[Theorem 1.6.3]{AK}) we can extract a subsequence $(g_{n_m})_{m=1}^\infty$ that converges weakly to some normalized and positive $g\in L_1(\mu)$. On the other hand, for any $f\in A$, $B\subseteq \Omega$ measurable, $\eps>0$ and $n\in \N$, we have
    \begin{align*}
        \int_{\Omega\setminus B} g_n\, d\mu + \int_B (\eps C_n^{-1} |f|)^r\, d\mu & \leq \int g_n\vee  (\eps C_n^{-1} |f|)^r\, d\mu = \| g_n^\frac{1}{r}\vee  (\eps C_n^{-1} |f|)\|^r_{L_r(\mu)}\\
        & = \norm[3]{ \intoo[3]{\bigvee_{i=1}^n|\alpha_{n,i} f_{n,i}|} \vee (\eps C_n^{-1} |f|) }_{L_r(\mu)} \leq \intoo[3]{\frac{C_{n+1}}{C_n}}^r(\delta_n^p+\eps^p)^\frac{r}{p}.
    \end{align*}
    Since $\int g_n\, d\mu=1$, we can rearrange the terms to obtain that
    \[\eps^r C_n^{-r} \int_B  |f|^r\, d\mu  \leq \int_B g_n\, d\mu + \intoo[3]{\frac{C_{n+1}}{C_n}}^r(\delta_n^p+\eps^p)^\frac{r}{p}-1.\]
    In particular, taking the limit over the subsequence $(n_m)_{m=1}^\infty$ we get that
    \[\eps^r C^{-r} \int_B  |f|^r\, d\mu  \leq \int_B g\, d\mu + (1+\eps^p)^\frac{r}{p}-1\leq \int_B g\, d\mu + \frac{r}{p}\eps^p,\]
    or equivalently,
    \[ C^{-r} \int_B  |f|^r\, d\mu  \leq \eps^{-r} \int_B g\, d\mu + \frac{r}{p}\eps^{p-r}\]
    for every $\eps>0$. The second term is minimized when $\epsilon^p=\frac{p}{p-r}\int_B g \, d\mu$, which yields
    \[ \int_B  |f|^r\, d\mu  \leq C^r \intoo[3]{1-\frac{r}{p}}^{\frac{r}{p}-1} \intoo[3]{\int_B g\, d\mu}^{1-\frac{r}{p}}\]
    for any $f\in A$ and $B\subseteq \Omega$, as we wanted to show.
\end{proof}

We can also restate \Cref{thm: factorization Pisier set} in the language of operators:

\begin{cor}\label{cor: factorization Pisier operator}
    Let $1\leq r<p<\infty$, $E$ a Banach space and $T:E\rightarrow L_r(\mu)$ a linear operator. The following are equivalent:
    \begin{enumerate}
        \item There exists $C>0$ such that $T$ is $(p,\infty)$-convex with $K^{(p,\infty)}(T)\leq C$.
        \item There exists $C'>0$ and a normalized $g\in L_1(\mu)_+$ such that for every $x\in E$ and measurable subset $B\subseteq \Omega$
        \[\|Tx\chi_B\|_{L_r(\mu)}\leq C'\|x\| \intoo[3]{\int_B g\, d\mu}^{\frac{1}{r}-\frac{1}{p}}.\]
        \item There exists $C''>0$ and a normalized $g\in L_1(\mu)_+$ such that for every $x\in E$,
        \[\{g=0\}\subseteq \{Tx=0\}\quad \text{and}\quad \|g^{-\frac{1}{r}}Tx\|_{L^{[r]}_{p,\infty}(g\cdot\mu)}\leq C''\|x\|.\]
        \item There exists $C''>0$ and a normalized $g\in L_1(\mu)_+$ such that $\{g=0\}\subseteq \{Tx=0\}$ for every $x\in E$, and $T=M R$, where $R:E\rightarrow L^{[r]}_{p,\infty}(g\cdot\mu)$ given by $Rx=g^{-\frac{1}{r}}Tx$ is bounded with $\|R\|\leq C''$ and $M:L^{[r]}_{p,\infty}(g\cdot\mu)\rightarrow L_r(\mu)$ is the multiplication operator by $g^\frac{1}{r}$.
        \[\xymatrix{
		E  \ar@{->}[rd]_{R} \ar@{->}[rr]^{T} & & L_r(\mu)  \\
		 & L^{[r]}_{p,\infty}(g\cdot\mu) \ar@{->}[ru]_{M} & 
        }\]
    \end{enumerate}
    Moreover, $C\leq C''\leq C'\leq (1-\frac{r}{p})^{\frac{1}{p}-\frac{1}{r}}C$.
\end{cor}

\begin{proof}
    Conditions $(1)$, $(2)$ and $(3)$ are equivalent to the corresponding conditions of \Cref{thm: factorization Pisier set} taking $A=T(B_E)$.\\

    $(3)\Rightarrow(4)$ Since $\int g\, d\mu=1$, the operator
    \[\fullfunction{M}{L^{[r]}_{p,\infty}(g\cdot\mu)}{L_r(\mu)}{h}{g^\frac{1}{r} h}\]
    is bounded with norm one:
    \[\|Mh\|_{L_r(\mu)}=\intoo[3]{\int |h|^r g\,d\mu}^\frac{1}{r}\leq \intoo[3]{\int g\, d\mu}^{\frac{1}{r}-\frac{1}{p}} \|h\|_{L^{[r]}_{p,\infty}(g\cdot\mu)}= \|h\|_{L^{[r]}_{p,\infty}(g\cdot\mu)}. \]
    On the other hand, the operator $R$ is by $(3)$.\\

    $(4)\Rightarrow(3)$ It follows from the boundedness of $R$.
\end{proof}

The dual statements for concave operators now follow easily. Recall that we are assuming that the function $g^{-\frac{1}{s}}$ vanishes in the set $\{g=0\}$.

\begin{cor}\label{cor: dual factorization Maurey operator}
    Let $1< q<s< \infty$, $E$ a Banach space, $T:L_s(\mu)\rightarrow E$ a linear operator and $C>0$. The following are equivalent:
    \begin{enumerate}
        \item $T$ is $q$-concave with $K_{(q)}(T)\leq C$.
        \item There exists a normalized $g\in L_1(\mu)_+$ such that for every $f\in L_s(\mu)$,
        \[\|Tf\|\leq C\|fg^{-\frac{1}{s}}\|_{L_q(g\cdot\mu)}.\]
        \item There exists a normalized $g\in L_1(\mu)_+$ such that $T=S \widetilde{M}$, where $\widetilde{M}:L_s(\mu) \rightarrow L_q(g\cdot \mu)$ is the multiplication operator by $g^{-\frac{1}{s}}$ and $S:L_q(g\cdot \mu)\rightarrow E$ is bounded with $\|S\|\leq C$.
        \[\xymatrix{
		L_s(\mu)  \ar@{->}[rd]_{\widetilde{M}} \ar@{->}[rr]^{T} & & E \\
		 & L_q(g\cdot \mu) \ar@{->}[ru]_{S} & 
        }\]
    \end{enumerate}
\end{cor}

\begin{proof}
    Let $p=q^*$ and $r=s^*$, so that $1< r<p<\infty$. We will establish the equivalence of the conditions in the statement with the corresponding conditions in \Cref{cor: factorization Maurey operator} for the operator $T^*:E^*\rightarrow L_r(\mu)$, that will be denoted by $(1^*)$, $(2^*)$ and $(3^*)$.\\

    $(1)\Leftrightarrow (1^*)$ It follows from \Cref{thm: duality convexity concavity}.\\

    $(2)\Rightarrow (3)$ Let $\nu=g\cdot \mu$. Note that for every $f\in L_s(\mu)$, $\nu (\{g=0\}\cap\{f\neq 0\})=0$, so $\widetilde{M}f=fg^{-\frac{1}{s}}$ is a measurable and finite $\nu$-a.e. function. Moreover, since $q<s$ and $\nu$ is a probability measure,
    \[\|\widetilde{M}f\|_{L_q(\nu)}\leq \|\widetilde{M}f\|_{L_s(\nu)}=\intoo[3]{\int_{\{g>0\}} |f|^s\, d\mu}^\frac{1}{s}\leq \|f\|_{L_s(\mu)},\]
    so $\widetilde{M}$ is bounded. Observe that for every $h\in \widetilde{M}(L_s(\mu))$, the operator $Sh=T(hg^\frac{1}{s})$ is bounded with norm below $C$. Since $\widetilde{M}(L_s(\mu))=L_s(\nu)\subseteq L_q(\nu)$ is dense in $L_q(\nu)$, we can extend it to $S:L_q(\nu)\rightarrow E$. We see that for every $f\in L_s(\mu)$,
    \[S \widetilde{M} f= S(fg^{-\frac{1}{s}})=S(fg^{-\frac{1}{s}}\chi_{\{g>0\}})=T(f\chi_{\{g>0\}}),\]
    but condition $(2)$ implies that 
    \[\|Tf-T(f\chi_{g>0})\|\leq C\|(f-f\chi_{g>0})g^{-\frac{1}{s}}\|_{L_q(\nu)}=0,\]
    so $T=S \widetilde{M}$.\\

    $(3)\Rightarrow (2)$ It follows from the fact that $Tf=S(f g^{-\frac{1}{s}})$ and $S$ is bounded.\\

    In order to prove $(3)\Leftrightarrow (3^*)$, we first observe that given a normalized function $g\in L_1(\mu)_+$, both operators 
    \[\fullfunction{\widetilde{M}}{L_s(\mu)}{L_q(g\cdot \mu)}{f}{fg^{-\frac{1}{s}}}\quad \text{and} \quad \fullfunction{M}{L_p(g\cdot \mu)}{L_r(\mu)}{h}{hg^\frac{1}{r}}\]
    are contractive. Moreover, for every $f\in L_s(\mu)$ and $h\in L_p(g\cdot \mu)$,
    \[\langle \widetilde{M}^*h, f \rangle_{L_s(\mu)} = \langle h ,\widetilde{M}f \rangle_{L_s(\mu)}= \int h (fg^{-\frac{1}{s}}) g \,d\mu = \int (hg^{\frac{1}{r}}) f \,d\mu = \langle Mh, f \rangle_{L_s(\mu)},  \]
    so $\widetilde{M}^*=M$. Since both $L_s(\mu)$ and $L_q(g\cdot \mu)$ are reflexive, it follows that $M^*=\widetilde{M}^{**}=\widetilde{M}$.\\

    Now, to check that $(3)\Rightarrow (3^*)$, note that if $T$ factors through $L_q(g\cdot \mu)$ as $T=S \widetilde{M}$, taking adjoints we get that $T^*$ factors through $L_p(g\cdot \mu)$ as $T^*=M S^*$. Moreover, $(3)$ (or equivalently, $(2)$) implies that $T(f\chi_{\{g=0\}})=0$ for every $f\in L_s(\mu)$. By the definition of $T^*$, this is equivalent to $\{g=0\}\subseteq \{T^*x^*=0\}$ for every $x^*\in E^*$, so $(3^*)$ holds.\\

    Conversely, to show $(3^*)\Rightarrow (3)$ we observe that, if $T^*$ factors through $L_p(g\cdot \mu)$ as $T^*=M R$, taking adjoints we get that $T^{**}:L_s(\mu)\rightarrow E^{**}$ factors as $T^{**}=R^* \widetilde{M}$. In general, if $U: F\rightarrow E$ is a bounded operators between Banach spaces, it holds that $U^{**} J_F=J_E U$, where $J_F$ and $J_E$ denote the corresponding canonical embeddings into the biduals. In the case $U=T$, $F=L_s(\mu)$ is reflexive, so $J_F=i_F$ and we get that $T^{**}=J_E T$, so we can write $T=J_E^{-1} R^*  \widetilde{M}$, as we wanted to show.
\end{proof}

\begin{cor}\label{cor: dual factorization Pisier operator}
    Let $1< q<s< \infty$, $E$ a Banach space and $T:L_s(\mu)\rightarrow E$ a linear operator. The following are equivalent:
    \begin{enumerate}
        \item There exists $C>0$ such that $T$ is $(q,1)$-concave with $K_{(q,1)}(T)\leq C$.
        \item There exists $C'>0$ and a normalized $g\in L_1(\mu)_+$ such that for every $f\in L_s(\mu)$ and measurable subset $B\subseteq \Omega$
        \[\|T(f\chi_B)\|\leq C'\|f\|_{L_s(\mu)} \intoo[3]{\int_B g\, d\mu}^{\frac{1}{q}-\frac{1}{s}}.\]
        \item There exists $C''>0$ and a normalized $g\in L_1(\mu)_+$ such that for every $f\in L_s(\mu)$,
        \[\|Tf\|\leq C''\|fg^{-\frac{1}{s}}\|_{L_{q,1}(g\cdot\mu)}.\]
        \item There exists $C''>0$ and a normalized $g\in L_1(\mu)_+$ such that $T=S \widetilde{M}$, where $\widetilde{M}:L_s(\mu)\rightarrow L_{q,1}(g\cdot\mu)$ is the multiplication operator by $g^{-\frac{1}{s}}$ and $S:L_{q,1}(g\cdot\mu)\rightarrow E$ is bounded with $\|S\|\leq C''$.
        \[\xymatrix{
		L_s(\mu)  \ar@{->}[rd]_{\widetilde{M}} \ar@{->}[rr]^{T} & & E  \\
		 & L_{q,1}(g\cdot\mu) \ar@{->}[ru]_{S} & 
        }\]
    \end{enumerate}
\end{cor}

\begin{proof}
    Again, put $1<r=s^*<p=q^*<\infty$, and denote by $(1^*)$, $(2^*)$, $(3^*)$ and $(4^*)$ the conditions in \Cref{cor: factorization Pisier operator} associated to the operator $T^*:E^*\rightarrow L_r(\mu)$. Throughout the proof, for any measure $\nu$, $L_{q,1}(\nu)$ will be endowed with the equivalent renorming
    \[\|f\|_{L^{[r]}_{q,1}(\nu)}=\sup \cbr[3]{\int fh \,d\nu : \|h\|_{L^{[r]}_{p,\infty}(\nu)}\leq 1}\]
    for every $f\in L_{q,1}(\nu)$.\\

    The equivalence $(1)\Leftrightarrow (1^*)$ follows from \Cref{thm: duality convexity concavity}, and $(2)\Leftrightarrow (2^*)$ is just a consequence of the fact that
    \[\langle x^*, T(f\chi_B) \rangle_{E}=\langle T^* x^*, f\chi_B \rangle_{L_s(\mu)}=\langle (T^* x^*)\chi_B, f \rangle_{L_s(\mu)}\]
    for every $f\in L_s(\mu)$, $x^*\in E^*$ and measurable $B\subseteq \Omega$.\\

    $(3)\Rightarrow (4)$ Note that for every $f\in L_s(\mu)$ and $h\in L_{p,\infty}(g\cdot \mu)$, using Hölder's inequality and the fact that $g\cdot \mu$ is a probability measure, we get that
    \[\left| \int fg^{-\frac{1}{s}} h g\,d\mu \right| \leq \int |fh|g^{\frac{1}{r}}  \,d\mu  \leq \intoo[3]{\int |f|^s\, d\mu}^\frac{1}{s} \intoo[3]{\int |h|^r g\, d\mu}^\frac{1}{r} \leq \|f\|_{L_s(\mu)} \|h\|_{L^{[r]}_{p,\infty}(g\cdot\mu)}, \]
    so the operator
    \[\fullfunction{\widetilde{M}}{L_s(\mu)}{L_{q,1}(g\cdot \mu)}{f}{fg^{-\frac{1}{s}}}\]
    is a contraction. Moreover, $\widetilde{M}(L_s(\mu))=L_s(g\cdot\mu)$ is dense in $L_{q,1}(g\cdot \mu)$. The rest of the proof follows as in \Cref{cor: dual factorization Maurey operator}.\\

    $(4)\Rightarrow (3)$ It is straightforward.\\

    To show $(4)\Leftrightarrow (4^*)$ we observe again that for any normalized function $g\in L_1(\mu)_+$, both operators 
    \[\fullfunction{\widetilde{M}}{L_s(\mu)}{L_{q,1}(g\cdot \mu)}{f}{fg^{-\frac{1}{s}}}\quad \text{and} \quad \fullfunction{M}{L_{p,\infty}(g\cdot \mu)}{L_r(\mu)}{h}{hg^\frac{1}{r}}\]
    are contractive and $\widetilde{M}^*=M$. Even though $M^*$ no longer coincides formally with $\widetilde{M}$, as $L_{q,1}(g\cdot \mu)$ is not reflexive, $L_s(\mu)$ is reflexive, so $M^*=\widetilde{M}^{**}=J_{L_{q,1}(g\cdot \mu)} \widetilde{M}$.\\

    Now $(4)\Rightarrow (4^*)$ follows by taking adjoints in the factorization of $T$ given in $(4)$ and observing that condition $(3)$ implies that $\{g=0\}\subseteq \{T^*x^*=0\}$. On the other hand, to show $(4^*)\Rightarrow (4)$ we write $T^*=M R$, so that taking adjoints and using the reflexivity of $L_s(\mu)$ we get that
    \[J_E T=T^{**}= R^* M^*=R^* J_{L_{q,1}(g\cdot \mu)} M,\]
    so $T = J_E^{-1} R^* J_{L_{q,1}(g\cdot \mu)} M$.
\end{proof}

Combining Krivine and Maurey--Nikishin results, Reisner proved the following theorem \cite[Corollary 7]{Reisner}. Similar results can be obtained for the remaining three relevant cases ($(p,p')$-convexity and $(q,q')$-concavity for $p'\in \{p,\infty\}$ and $q'\in \{1,q\}$) using the other results established in this section, see \cite[Theorem 7.10]{GLTT2}.

\begin{thm}\label{thm: factorization p-q}
    Let $E$ and $F$ be Banach spaces, $X$ a Banach lattice, $1\leq q< p<\infty$, $U:E\rightarrow X$ a $p$-convex operator and $V:X\rightarrow F$ a $q$-concave operator. Then, there exists a measure $\mu$, a function $h\in L_r(\mu)_+$ with $\|h\|_{L_r(\mu)}=1$, where $\frac{1}{q}=\frac{1}{p}+\frac{1}{r}$, and bounded operators $R:E\rightarrow  L_p(\mu)$, $S:L_p(\mu)\rightarrow  L_q(\mu)$ and $T:L_q(\mu)\rightarrow F$ such that $T S R=V U$, $\|R\|\leq K^{(p)}(U)$, $\|T\|\leq K_{(q)}(V)$ and $S$ is the operator of multiplication by $h$.
\end{thm}

\begin{proof}
    As in the proof of \Cref{thm: factorization Krivine}, it suffices to show the statement for a contractive lattice homomorphism $Q:Y\rightarrow Z$, where $Y$ is $p$-convex and $Z$ is $q$-concave, both with constant one.\\

    Since $p>q$, by \Cref{thm: p-conv implies q-conv} $Y$ is $q$-convex with constant one, so, by \Cref{thm: factorization Krivine}, there exists a measure $\mu$ and contractive lattice homomorphisms $W:Y\rightarrow L_q(\mu)$ and $T:L_q(\mu)\rightarrow Z$ such that $Q=T W$. It follows that $W$ is $p$-convex with constant one, so we can apply \Cref{cor: factorization Maurey operator} to find a normalized function $g\in L_1(\mu)_+$ such that the operators $P:Y\rightarrow L_p(g\cdot \mu)$ and $M:L_p(g\cdot \mu)\rightarrow L_q(\mu)$, given by $Py=g^{-\frac{1}{q}}Wy$ and $Mf=g^\frac{1}{q} f$, are contractive lattice homomorphisms that satisfy $W=M P$. Let $h=g^\frac{1}{r}$, which is a positive and normalized function in $L_r(\mu)$, $\frac{1}{q}=\frac{1}{p}+\frac{1}{r}$. Then, if we define $R:Y\rightarrow L_p(\mu)$ and $S:L_p( \mu)\rightarrow L_q(\mu)$ by $Ry=h^{-1} Wy$ and $Df=h f$, we get that $W=S R$, and moreover
    \[\|Ry\|_{L_p(\mu)}^p=\int (g^{\frac{1}{p}-\frac{1}{q}}|Wy|)^p\, d\mu = \int (g^{-\frac{1}{q}}|Wy|)^p g\, d\mu = \|Py\|_{L_p(g\cdot \mu)}^p\leq \|y\|^p\]
    and
    \[\|Sf\|_{L_q(\mu)}=\intoo[3]{\int (g^{\frac{1}{q}-\frac{1}{p}}|f|)^q\, d\mu}^\frac{1}{q}\leq \intoo[3]{\int g\, d\mu}^\frac{1}{r}\intoo[3]{\int |f|^p\, d\mu}^\frac{1}{p}=\|f\|_{L_p(\mu)}.\]
    Therefore, $Q$ factors as $T S R$, where $R:E\rightarrow  L_p(\mu)$, $S:L_p(\mu)\rightarrow  L_q(\mu)$ and $T:L_q(\mu)\rightarrow F$ are contractive lattice homomorphisms and $Sf=hf$ for some $h\in L_r(\mu)_+$ of norm one.
\end{proof}

\subsection{Factorization of \texorpdfstring{$(p,q)$}{}-concave operators on \texorpdfstring{$C(K)$}{}-spaces}\label{sec: concave on C(K)}
Note that Corollaries \ref{cor: dual factorization Maurey operator} and \ref{cor: dual factorization Pisier operator} do not cover the case $s=\infty$, that is, when the domain is an $L_\infty$-space, or more generally, a $C(K)$-space. However, in light of \Cref{prop: weak p-summing norm in C(K)}, we know that every $(p,q)$-concave operator on a $C(K)$-space is $(p,q)$-summing. Therefore, we can reinterpret Pietsch's Factorization Theorem for $p$-summing operators, and Pisier's extension to $(p,q)$-summing operators, as the missing factorization theorems for $(p,q)$-concave operators on $C(K)$-spaces. Moreover, the connection between $(p,q)$-concave operators and $(p,q)$-summing operators that we established in \Cref{thm: characterization (pq)-concave and (pq)-summing} can be further exploited to obtain \Cref{thm: relevant cases of concavity}, allowing us to conclude that the classes of $(p,q)$-concave operators and $(p,1)$-concave operators coincide, and that $(p,1)$-concavity implies $q$-concavity for every $q>p$. To establish these results, we will follow \cite[Section 21]{TJ} (see also \cite[Chapter 10]{DJT}). We start by proving the following interpolation inequality. Recall that, if $(\Omega,\Sigma,\mu)$ is a measure space, the norm in $L_{p,1}(\mu)$ is given by
\[\vvvert f\vvvert_{L_{p,1}}=\int_0^\infty t^{-\frac{1}{p^*}}f^*(t)\,dt.\]

\begin{lem}\label{lem: interpolation inequality Lp1}
    Let $1\leq q<p<\infty$. Then, for every $f\in L_\infty(\mu)\cap L_q(\mu)$ we have
    \[\vvvert f \vvvert_{L_{p,1}}\leq \intoo[2]{ p+\intoo[2]{\frac{1}{q}-\frac{1}{p}}^{-\frac{1}{q^*}}} \|f\|_q^{\frac{q}{p}}\|f\|_\infty^{1-\frac{q}{p}}.\]
\end{lem}

\begin{proof}
    Let us fix $f\neq 0$ and $b>0$. Then,
    \[\int_0^b t^{-\frac{1}{p^*}}f^*(t)\, dt\leq \int_0^b t^{-\frac{1}{p^*}}\|f\|_\infty \, dt = pb^\frac{1}{p}\|f\|_\infty, \]
    and
    \[\int_b^\infty t^{-\frac{1}{p^*}}f^*(t)\, dt\leq \intoo[3]{\int_b^\infty f^*(t)^q \, dt}^\frac{1}{q} \intoo[3]{\int_b^\infty t^{-\frac{q^*}{p^*}}\, dt}^\frac{1}{q^*}\leq \intoo[2]{\frac{p(q-1)}{p-q}}^\frac{1}{q^*}b^{\frac{1}{p}-\frac{1}{q}}\|f\|_q\]
    In particular, for $b=\intoo[1]{\frac{\|f\|_q}{\|f\|_\infty}}^q$ the above inequalities yield
    \[\vvvert f \vvvert_{L_{p,1}}\leq \intoo[2]{ p+\intoo[2]{\frac{p(q-1)}{p-q}}^\frac{1}{q^*}} \|f\|_q^{\frac{q}{p}}\|f\|_\infty^{1-\frac{q}{p}} \leq \intoo[2]{ p+\intoo[2]{\frac{1}{q}-\frac{1}{p}}^{-\frac{1}{q^*}}} \|f\|_q^{\frac{q}{p}}\|f\|_\infty^{1-\frac{q}{p}}. \qedhere \]
\end{proof}

Now, let $K$ be a compact Hausdorff space. Given a finite positive measure $\mu\in C(K)^*$, we will denote by $\vvvert f \vvvert_{L_{p,1}}$, $\|f\|_p$ and $\|f\|_\infty$ the corresponding norm of $f\in C(K)$ when seen as an element of $L_{p,1}(K,\mu)$, $L_p(K,\mu)$ and $L_\infty(K,\mu)$, respectively. 

\begin{lem}\label{lem: summing embedding of C(K)}
    Let $K$ be a compact Hausdorff space and $\mu\in C(K)^*$ a finite positive measure over $K$.
    \begin{enumerate}
        \item Let $1\leq p<\infty$, and denote by $J_p:C(K)\rightarrow L_p(K,\mu)$ the identity operator. Then $J_p$ is $p$-summing, and 
        \[\pi_p(J_p)\leq \mu(K)^\frac{1}{p}.\]
        \item Let $1< p<\infty$, and denote by $J_{p,1}:C(K)\rightarrow L_{p,1}(K,\mu)$ the identity operator. Then $J_{p,1}$ is $(p,q)$-summing for every $1\leq q<p$, and 
        \[\pi_{p,q}(J_{p,1})\leq \intoo[2]{ p+\intoo[2]{\frac{1}{q}-\frac{1}{p}}^{-\frac{1}{q^*}}}  \mu(K)^\frac{1}{p}.\]
    \end{enumerate}
\end{lem}

\begin{proof}
    $(1)$ Let $f_1,\ldots,f_n\in C(K)$. Then, by \Cref{prop: weak p-summing norm in C(K)},
    \[\intoo[3]{\sum_{i=1}^n \|f_i\|_p^p}^\frac{1}{p} = \intoo[3]{\int_K \sum_{i=1}^n |f_i(t)|^p \, d\mu}^\frac{1}{p} \leq \mu(K)^\frac{1}{p}  \intoo[3]{\sup_{t\in K} \sum_{i=1}^n |f_i(t)|^p }^\frac{1}{p} =\mu(K)^\frac{1}{p}  \|(f_i)_{i=1}^n\|_{p,w}. \]
    
    $(2)$ Fix $1\leq q<p$ and let $f_1,\ldots,f_n\in C(K)$. We denote by $A=\intoo[2]{ p+\intoo[2]{\frac{1}{q}-\frac{1}{p}}^{-\frac{1}{q^*}}}$. Then, using \Cref{lem: interpolation inequality Lp1}, the first statement of this lemma, and \Cref{prop: weak p-summing norm in C(K)}, we get
    \begin{align*}
        \intoo[3]{\sum_{i=1}^n \vvvert f_i \vvvert_{L_{p,1}}^p}^\frac{1}{p} &\leq A \intoo[3]{\sum_{i=1}^n \|f_i\|_q^q \|f_i\|_\infty^{p-q}}^\frac{1}{p}\leq  A \intoo[3]{\sum_{i=1}^n \|f_i\|_q^q }^\frac{1}{p} \sup_i \|f_i\|_\infty^{1-\frac{q}{p}}\\
        & \leq A \pi_q(J_q)^\frac{q}{p}  \|(f_i)_{i=1}^n\|_{q,w}^\frac{q}{p} \norm[3]{\intoo[3]{\sum_{i=1}^n |f_i|^q }^\frac{1}{q}}_\infty^{1-\frac{q}{p}}\leq A \mu(K)^\frac{1}{p}  \|(f_i)_{i=1}^n\|_{q,w}.\qedhere
    \end{align*}
\end{proof}

We are now ready to prove Pietsch's Factorization Theorem for $p$-summing operators on $C(K)$-spaces (a more general statement can be found for instance in \cite[Theorem 9.2]{TJ}).

\begin{thm}\label{thm: Pietsch factorization}
    Let $K$ be a compact Hausdorff space, $E$ be a Banach space, $T:C(K)\rightarrow E$ an operator, $C>0$ and $1\leq p<\infty$. The following are equivalent:
    \begin{enumerate}
        \item $T$ is $p$-summing (or equivalently, $p$-concave) with constant $C$.
        \item There exists $\mu\in C(K)^*_+$ a probability measure on $K$ such that $T$ factors through $L_p(K,\mu)$ as $T=\overline{T}J_p$, where $J_p:C(K)\rightarrow L_p(K, \mu)$ is the canonical inclusion and $\overline{T}:L_{p,1}(K,\mu)\rightarrow E$ has norm $\|\overline{T}\|\leq C$.
        \[\xymatrix{
		C(K)  \ar@{->}[rd]_{J_p} \ar@{->}[rr]^{T} & & E  \\
		 & L_p(K, \mu) \ar@{->}[ru]_{\overline{T}} & 
        }\]
    \end{enumerate}
\end{thm}

\begin{proof}
    $(2)\Rightarrow (1)$ It follows easily from the fact that $J_p$ is $p$-summing with constant $\mu(K)^\frac{1}{p}=1$ by \Cref{lem: summing embedding of C(K)}.\\

    $(1)\Rightarrow (2)$ Let 
    \[W=\cbr[3]{f\in C(K): f=C^p\sum_{i=1}^n|f_i|^p\text{ for some } (f_i)_{i=1}^n\subseteq C(K)\text{ with }\sum_{i=1}^n \|Tf_i\|^p=1},\]
    which is clearly a convex subset of $C(K)_+$. Since $T$ is $p$-concave, it follows that $\|f\|\geq 1$ for every $f\in W$, so $W$ is disjoint from $B$, the open unit ball of $C(K)$. Therefore, by Hahn--Banach Theorem there exists a measure $\nu\in C(K)$ such that $\nu(g)<1$ for every $g\in B$ and $\nu(f)\geq 1$ for every $f\in W$. In particular, $\|\nu\|\leq 1$. Let $\mu=|\nu|/\|\nu\|$. Then, $\mu$ is a probability measure such that $\mu(f)\geq \nu(f)/\|\nu\|\geq 1$ for every $f\in W$. In particular, for every $f\notin \ker T$, the function $C^p|f|^p/\|Tf\|^p$ belongs to $W$, which means that
    \[\|Tf\|^p\leq C^p\mu(|f|^p)=C^p\int_K|f|^p\, d\mu = C^p\|J_pf\|_{L_p(K,\mu)}^p.\]
    Therefore, we can define an operator $\overline{T}$ on $J_p(C(K))$ by $\overline{T} J_p f=Tf$ whose norm is bounded by $C$. Since $J_p$ has dense range in $L_p(K,\mu)$, $\overline{T}$ can be extended to the whole space.
\end{proof}

In \cite{Pisier86}, Pisier proved the following generalization of Pietsch's Factorization theorem for $(p,q)$-summing operators on $C(K)$ spaces, where $L_p$ is replaced by $L_{p,1}$.

\begin{thm}\label{thm: factorization (pq)-summing through Lp1}
    Let $K$ be a compact Hausdorff space, $E$ be a Banach space, $T:C(K)\rightarrow E$ an operator, and $1\leq q<p<\infty$. The following are equivalent:
    \begin{enumerate}
        \item $T$ is $(p,q)$-summing (or equivalently, $(p,q)$-concave).
        \item There exists $A'>0$ and $\mu\in C(K)^*_+$ a probability measure on $K$ such that 
        \[\|Tf\|\leq A' \|f\|_q^{\frac{q}{p}}\|f\|_\infty^{1-\frac{q}{p}}\]
        for every $f\in C(K)$.
        \item There exists $A''>0$ and $\mu\in C(K)^*_+$ a probability measure on $K$ such that 
        \[\|Tf\|\leq A'' \vvvert f \vvvert_{L_{p,1}}\]
        for every $f\in C(K)$.
        \item There exists $A''>0$ and $\mu\in C(K)^*_+$ a probability measure on $K$ such that $T$ factors through $L_{p,1}(K,\mu)$ as $T=\overline{T}J_{p,1}$, where $\overline{T}:L_{p,1}(K,\mu)\rightarrow E$ and $\|\overline{T}\|\leq A''$.
        \[\xymatrix{
		C(K)  \ar@{->}[rd]_{J_{p,1}} \ar@{->}[rr]^{T} & & E  \\
		 & L_{p,1}(K, \mu) \ar@{->}[ru]_{\overline{T}} & 
        }\]
    \end{enumerate}
    Moreover, 
    \begin{align*}
        & A' \leq p^\frac{1}{p}\pi_{p,q}(T)\\
        & A''  \leq p^{-1} A'\\
        & \pi_{p,q}(T) \leq \intoo[2]{ p+\intoo[2]{\frac{1}{q}-\frac{1}{p}}^{-\frac{1}{q^*}}} A''.
    \end{align*}
\end{thm}

\begin{proof}
    $(1)\Rightarrow (2)$ Fix $n\in\N$ and let $\delta_n=1+\frac{1}{n}$. If we denote by
    \[\pi^{(n)}_{p,q}(T)=\sup \cbr[3]{\frac{\intoo[0]{\sum_{i=1}^n \|Tf_i\|^p }^\frac{1}{p}}{\norm[0]{\intoo[0]{\sum_{i=1}^n |f_i|^q }^\frac{1}{q}}_\infty}: f_1,\ldots,f_n\in C(K)},\]
    we can find functions $f_1,\ldots,f_n\in C(K)$ such that 
    \[\intoo[3]{\sum_{i=1}^n \|Tf_i\|^p }^\frac{1}{p}=1\quad \text{and}\quad \norm[3]{\intoo[3]{\sum_{i=1}^n |f_i|^q }^\frac{1}{q}}_\infty \leq \frac{\delta_n}{\pi^{(n)}_{p,q}(T)}.\]
    For each $i=1,\ldots,n$, there exists a normalized functional $z_i^*\in E^*$ such that $z_i^*(Tf_i)=\|Tf_i\|$, so
    \begin{align*}
        1 &= \intoo[3]{\sum_{i=1}^n \|Tf_i\|^p }^\frac{1}{p}= \max \cbr[3]{\sum_{i=1}^n a_i\|Tf_i\|: (a_i)_{i=1}^n\in B_{\ell_{p^*}^n}}\\
        & =\max \cbr[3]{\sum_{i=1}^n a_iz_i^*(Tf_i): (a_i)_{i=1}^n\in B_{\ell_{p^*}^n}} =\sum_{i=1}^n a_iz_i^*(Tf_i)
    \end{align*}
    for some $(a_i)_{i=1}^n\in B_{\ell_{p^*}^n}$. If we denote by $y_i^*=a_iz_i^*$, then it follows that 
    \[\sum_{i=1}^n \|y^*_i\|^{p^*} =1 \quad \text{and} \quad  \sum_{i=1}^n y_i^*(Tf_i) =1.\]
    Let us define for every $n\in \N$ the functional $\mu_n$ by
    \[\mu_n(f)=\sum_{i=1}^n y_i^*(T(ff_i))\]
    for every $f\in C(K)$. Clearly, $\mu_n$ is linear and $\mu_n(\uno)=1$ for every $n\in \N$, where $\uno$ denotes the constant one function over $K$. Moreover, 
    \begin{align*}
        |\mu_n(f)| &\leq \sum_{i=1}^n \|y_i^*\|\|T(ff_i)\| \leq \intoo[3]{\sum_{i=1}^n \|y^*_i\|^{p^*} }^\frac{1}{p^*} \intoo[3]{\sum_{i=1}^n \|T(ff_i)\|^p }^\frac{1}{p}\\
        & \leq \pi^{(n)}_{p,q}(T) \norm[3]{\intoo[3]{\sum_{i=1}^n |ff_i|^q }^\frac{1}{q}}_\infty \leq \delta_n \|f\|_\infty \leq 2 \|f\|_\infty.
    \end{align*}
    so $\mu_n\in C(K)^*$ for every $n\in \N$ and they are uniformly bounded. Let $(\mu_{n_\alpha})_\alpha$ be a $w^*$-convergent subnet of $\{\mu_n\}_n$ that converges to a certain acumulation point $\mu\in C(K)^*$. It follows that
    \[|\mu(f)|=\lim_\alpha |\mu_{n_\alpha}(f)|\leq \lim_\alpha \delta_{n_\alpha} \|f\|_\infty =\|f\|_\infty\]
    for every $f\in C(K)$, and in particular
    \[\mu(\uno)=\lim_\alpha \mu_{n_\alpha}(\uno)=1,\]
    so $\|\mu\|=1$. Note that, since $C(K)^*$ is an $AL$-space, it follows that
    \[\mu_+(\uno)-\mu_-(\uno)=\mu(\uno)=1=\|\mu\|=\||\mu|\| =\|\mu_+\|+\|\mu_-\|=\mu_+(\uno)+\mu_-(\uno),\]
    so $0=\mu_-(\uno)=\|\mu_-\|$ and therefore $\mu$ is a (positive) probability measure. Now, let us fix a non-zero $f\in C(K)$ and $n\in \N$. We can define
    \[h_i=f_i\intoo[3]{\uno-\frac{|f|^q}{\|f\|^q_\infty}}^\frac{1}{q}, i=1,\ldots,n\quad \text{and}\quad h_{n+1}=\frac{\delta_n}{\pi^{(n)}_{p,q}(T)}\cdot \frac{f}{\|f\|_\infty},\]
    so that
    \begin{align*}
        \intoo[3]{\sum_{i=1}^{n+1} |h_i|^q }^\frac{1}{q} &= \intoo[3]{\sum_{i=1}^{n} |f_i|^q \intoo[3]{\uno-\frac{|f|^q}{\|f\|^q_\infty}} +\intoo[3]{\frac{\delta_n}{\pi^{(n)}_{p,q}(T)}}^q  \frac{|f|^q}{\|f\|^q_\infty}}^\frac{1}{q}\\
        & \leq \intoo[3]{\intoo[3]{\frac{\delta_n}{\pi^{(n)}_{p,q}(T)}}^q \intoo[3]{\uno-\frac{|f|^q}{\|f\|^q_\infty}} +\intoo[3]{\frac{\delta_n}{\pi^{(n)}_{p,q}(T)}}^q  \frac{|f|^q}{\|f\|^q_\infty}}^\frac{1}{q} = \frac{\delta_n}{\pi^{(n)}_{p,q}(T)}\uno.
    \end{align*}
    By the definition of $\pi^{(n+1)}_{p,q}(T)$,
    \[\intoo[3]{\sum_{i=1}^{n+1} \|Th_i\|^p }^\frac{1}{p}\leq \pi^{(n+1)}_{p,q}(T) \norm[3]{\intoo[3]{\sum_{i=1}^{n+1} |h_i|^q }^\frac{1}{q}}_\infty \leq \frac{\pi^{(n+1)}_{p,q}(T)}{\pi^{(n)}_{p,q}(T)}\delta_n, \]
    and by Hölder's inequality
    \[\abs[3]{\mu_n \intoo[3]{\intoo[3]{\uno-\frac{|f|^q}{\|f\|^q_\infty}}^\frac{1}{q}}}=\abs[3]{\sum_{i=1}^n y_i^*(Th_i)} \leq  \intoo[3]{\sum_{i=1}^n \|Th_i\|^p }^\frac{1}{p}.\]
    Therefore
    \begin{align*}
        \|Tf\|^p &= \intoo[3]{\frac{\pi^{(n)}_{p,q}(T)}{\delta_n}}^p \|f\|_\infty^p \|Th_{n+1}\|^p \\
        & \leq \intoo[3]{\frac{\pi^{(n)}_{p,q}(T)}{\delta_n}}^p \|f\|_\infty^p \intoo{\intoo[3]{\frac{\pi^{(n+1)}_{p,q}(T)}{\pi^{(n)}_{p,q}(T)}\delta_n}^p - \abs[3]{\mu_n \intoo[3]{\intoo[3]{\uno-\frac{|f|^q}{\|f\|^q_\infty}}^\frac{1}{q}}}^p}.
    \end{align*}
    Taking the limit over the net $n_\alpha$, we get
    \[\|Tf\|^p \leq \pi_{p,q}(T)^p \|f\|_\infty^p \intoo{1 - \abs[3]{\mu \intoo[3]{\intoo[3]{\uno-\frac{|f|^q}{\|f\|^q_\infty}}^\frac{1}{q}}}^p}.\]
    Now, observe that $0\leq \uno-\frac{|f|^q}{\|f\|^q_\infty}\leq 1$, so $0\leq \uno-\frac{|f|^q}{\|f\|^q_\infty}\leq \intoo[3]{\uno-\frac{|f|^q}{\|f\|^q_\infty}}^\frac{1}{q}$, and since $\mu$ is positive, we get
    \begin{align*}
        \|Tf\|^p &\leq \pi_{p,q}(T)^p \|f\|_\infty^p \intoo{1 - \abs[3]{\mu \intoo[3]{\uno-\frac{|f|^q}{\|f\|^q_\infty}}}^p} \\
        & = \pi_{p,q}(T)^p \|f\|_\infty^p \intoo{1 - \abs[3]{1-\mu \intoo[3]{\frac{|f|^q}{\|f\|^q_\infty}}}^p} .
    \end{align*}
    Using the inequality $1-|1-t|^p\leq pt$, that holds for every $t\in [0,1]$, we get
    \[\|Tf\|^p  \leq \pi_{p,q}(T)^p \|f\|_\infty^p p \mu \intoo[3]{\frac{|f|^q}{\|f\|^q_\infty}} = p \pi_{p,q}(T)^p \|f\|_\infty^p \int \frac{|f|^q}{\|f\|^q_\infty} \,d\mu ,  \]
    or equivalently
    \[\|Tf\|\leq p^\frac{1}{p} \pi_{p,q}(T) \|f\|_q^{\frac{q}{p}}\|f\|_\infty^{1-\frac{q}{p}}.\]

    $(2)\Rightarrow (3)$ Let $\mu$ be a probability measure on $K$ satisfying $(2)$. Then, the operator $T:C(K)\rightarrow E$ can be extended to $\hat{T}:L_\infty(K,\mu) \rightarrow E$. Indeed, by Lusin's Theorem, for every $f\in L_\infty(K,\mu)$ there exists a sequence $(f_n)_{n=1}^\infty\subseteq C(K)$ converging $\mu$-a.e. to $f$ and such that $\|f_n\|_\infty \leq \|f\|_\infty$ for every $n\in \N$.  Note that $|f_n-f|^q\leq (2\|f\|_\infty)^q\uno \in L_1(K,\mu)$ and $|f_n-f|^q\rightarrow 0$ $\mu$-a.e., so by the Dominated Convergence Theorem we get that $f_n\rightarrow f$ in $L_q(K,\mu)$. In particular, $(f_n)_{n=1}^\infty$ is a Cauchy sequence in $L_q(K,\mu)$, so by the hypothesis we have
    \[\|Tf_m-Tf_n\|\leq A'\|f_m-f_n\|_q^{\frac{q}{p}} \|f_m-f_n\|_\infty^{1-\frac{q}{p}} \leq A' (2\|f\|_\infty)^{1-\frac{q}{p}}\|f_m-f_n\|_q^{\frac{q}{p}}  \]
    and $(Tf_n)_{n=1}^\infty$ is a Cauchy sequence in $E$. Let us denote by $\hat{T}f$ the limit of $(Tf_n)_{n=1}^\infty$. This limit does not depend of the original choice of sequence $(f_n)_{n=1}^\infty\subseteq C(K)$: if $(g_n)_{n=1}^\infty\subseteq C(K)$ is another sequence with $\|g_n\|_\infty \leq \|f\|_\infty$ that converges $\mu$-a.e. to $f$ then the alternate sequence $(h_n)_{n=1}^\infty$ given by $h_{2n-1}=f_n$ and $h_{2n}=g_n$ also satisfies the previous conditions, so repeating the argument we conclude that $(Th_n)_{n=1}^\infty$ is a Cauchy sequence, so in particular $\lim_n Tf_n=\lim_n Th_n=\lim_n Tg_n$ and $\hat{T}f$ is well defined. Note that for the above argument to hold, it suffices to consider sequences $(f_n)_{n=1}^\infty\subseteq C(K)$ converging $\mu$-a.e. to $f$ with $\|f_n\|_\infty \leq C\|f\|_\infty$ for some fixed constant $C>0$ independent of the sequence. In particular, this implies that $T$ is additive (and hence, linear): if $f,g\in L_\infty(K,\mu)$, we can find sequences $(f_n)_{n=1}^\infty, (g_n)_{n=1}^\infty \subseteq C(K)$ satisfying Lusin's Theorem, so that $f_n+g_n\rightarrow f+g$ $\mu$-a.e. and
    \[\|f_n+g_n\|_\infty \leq \|f_n\|_\infty +\|g_n\|_\infty\leq \|f\|_\infty+ \|g\|_\infty\leq 2\|f\|_\infty\vee \|g\|_\infty \leq 2\|f+g\|_\infty.\]
    Therefore, the sequence $(T(f_n+g_n))_{n=1}^\infty$ is a Cauchy sequence in $E$ that converges to $\hat{T}(f+g)$ on one hand, and to $\hat{T}f+\hat{T}g$ on the other, so both limits must coincide. Recall that the sequences that we have considered satisfy $f_n\rightarrow f$ in $L_q(K,\mu)$, in particular $\|f_n\|_q\rightarrow \|f\|_q$, so the operator $\hat{T}$ satisfies the estimate in $(2)$ too:
    \[\|\hat{T}f\|=\lim_n \|Tf_n\|\leq A' \limsup_n\|f_n\|_q^{\frac{q}{p}} \|f_n\|_\infty^{1-\frac{q}{p}} \leq A'  \lim_n\|f_n\|_q^{\frac{q}{p}} \|f\|_\infty^{1-\frac{q}{p}}= A'  \|f\|_q^{\frac{q}{p}} \|f\|_\infty^{1-\frac{q}{p}} .\]
    Hence, $\hat{T}$ is a bounded operator.\\

    Next, let us fix a simple function $f\in L_\infty(K,\mu)$, given by $f=\sum_{i=1}^m a_i\chi_{K_i}$ with $K_i\subseteq K$ pairwise disjoint and $(|a_i|)_{i=1}^m$ ordered in a decreasing way. Let us define $\eps_j=\text{sgn }a_j$, $f_j=\sum_{i=1}^j \eps_i \chi_{K_i}$, $B_j=\bigcup_{i=1}^j K_i$ and $t_j=\mu(B_j)$ for $j=1,\ldots,m$. Observe that $|f_j|=\chi_{B_j}$ for every $j$. We can decompose $f$ using the functions $f_j$ in the following way (here, $a_{m+1}=0$):
    \begin{align*}
        f&=\sum_{i=1}^m |a_i|\eps_i \chi_{K_i}=\sum_{i=1}^m \sum_{j=i}^m (|a_j|-|a_{j+1}|) \eps_i \chi_{K_i}\\
        & =\sum_{j=1}^m (|a_j|-|a_{j+1}|)\sum_{i=1}^j  \eps_i \chi_{K_i} = \sum_{j=1}^m (|a_j|-|a_{j+1}|)f_j.
    \end{align*}
    We will need to compute the norm of $f$ in $L_{p,1}(K,\mu)$. To do so, first note that the decreasing rearrangement of the simple function $f$ is given by a simple function in $[0,\infty)$ that takes the value $|a_j|$ at the interval $[t_{j-1}, t_j)$ (here $t_0=0$) of measure $\mu(K_j)$, i.e., $f^*=\sum_{j=1}^m |a_j|\chi_{[t_{j-1}, t_j)}$. Therefore,
        \[\vvvert f \vvvert_{L_{p,1}} = \int_0^\infty t^{-\frac{1}{p^*}}f^*(t) \, dt = \sum_{j=1}^m |a_j| \int_{t_{j-1}}^{t_j} t^{-\frac{1}{p^*}} \, dt = p\sum_{j=1}^m |a_j|(t_j^\frac{1}{p}-t_{j-1}^\frac{1}{p})\]
    Now, we can use the above estimates to obtain that
    \begin{align*}
        \|\hat{T}f\|&\leq \sum_{j=1}^m (|a_j|-|a_{j+1}|)\|\hat{T}f_j\| \leq \sum_{j=1}^m A'(|a_j|-|a_{j+1}|)\|f_j\|_q^{\frac{q}{p}} \|f_j\|_\infty^{1-\frac{q}{p}} \\
        & = A' \sum_{j=1}^m (|a_j|-|a_{j+1}|) t_j^{\frac{1}{p}}  = A' \sum_{j=1}^m |a_j|(t_j^\frac{1}{p}-t_{j-1}^\frac{1}{p}) =\frac{A'}{p} \vvvert f \vvvert_{L_{p,1}}.
    \end{align*}
    Since simple functions are dense in $L_\infty (K,\mu)$, and the identity operator from $L_\infty (K,\mu)$ into $L_{p,1} (K,\mu)$ is continuous, we conclude that
    \[\|\hat{T}f\| \leq \frac{A'}{p} \vvvert f \vvvert_{L_{p,1}}\]
    for every $f\in L_\infty(K,\mu)$, so in particular $(3)$ holds for every $f\in C(K)$.\\

    $(3)\Leftrightarrow (4)$ Since $J_{p,1}(C(K))$ is dense in $L_{p,1}(K,\mu)$, condition $(3)$ allows us to define an operator $\overline{T}:L_{p,1}(K,\mu)\rightarrow E$ such that $\overline{T}J_{p,1}=T$. For the reverse implication, just apply the boundedness of $\overline{T}$ to any $J_{p,1}f$ with $f\in C(K)$.\\

    $(4)\Rightarrow (1)$ By \Cref{lem: summing embedding of C(K)}, $J_{p,1}$ is $(p,q)$-summing, so $T=\overline{T}J_{p,1}$ is also $(p,q)$-summing, with constant
    \[\pi_{p,q}(T)\leq \|\overline{T}\|\pi_{p,q}(J_{p,1}) \leq \intoo[2]{ p+\intoo[2]{\frac{1}{q}-\frac{1}{p}}^{-\frac{1}{q^*}}} A''.  \qedhere\]
\end{proof}

As a consequence, we obtain the following inclusions between spaces of $(p,q)$-summing operators.

\begin{cor}\label{cor: inclusions (p1)-summing operators}
    Let $K$ be a compact Hausdorff space, $E$ be a Banach space and $1<p<\infty$. Then:
    \begin{enumerate}
        \item If $1< q<p$, $\Pi_{p,q}(C(K),E)=\Pi_{p,1}(C(K),E)$ as sets, and for every $T\in \Pi_{p,1}(C(K),E)$
        \[\pi_{p,1}(T) \leq \pi_{p,q}(T) \leq  \intoo[2]{ p^{\frac{1}{p}}+\intoo[2]{\frac{q}{p-q}}^{\frac{1}{q^*}}} \pi_{p,1}(T).\]
        \item If $p<q<\infty$, $\Pi_{p,1}(C(K),E)\subseteq \Pi_q(C(K),E)$, and for every $T\in \Pi_{p,1}(C(K),E)$
        \[\pi_q(T) \leq p^{-\frac{1}{p^*}} \intoo[2]{ q^*\intoo[2]{\frac{1}{p}-\frac{1}{q}}}^{-\frac{1}{q^*}} \pi_{p,1}(T).\]
    \end{enumerate}
\end{cor}

\begin{proof}
    $(1)$ Let us consider first $T\in \Pi_{p,q}(C(K),E)$. It follows that
    \[\intoo[3]{\sum_{i=1}^n \|Tf_i\|^p}^\frac{1}{p} \leq \pi_{p,q}(T) \sup_{\mu\in B_{C(K)^*}} \intoo[3]{\sum_{i=1}^n |\mu(f_i)|^q}^\frac{1}{q} \leq \pi_{p,q}(T) \sup_{\mu\in B_{C(K)^*}}\sum_{i=1}^n |\mu(f_i)|\]
    for every $f_1,\ldots,f_n\in C(K)$, so the first inequality holds. For the reverse inclusion, we exploit the fact that condition $(4)$ in \Cref{thm: factorization (pq)-summing through Lp1} does not depend on the second exponent $q$. Therefore, if $T\in \Pi_{p,1}(C(K),E)$, using $(1)\Rightarrow (4)$ with exponent $1$ we learn that there exists a probability measure $\mu\in C(K)^*$ and an operator $\overline{T}:L_{p,1}(K,\mu)\rightarrow E$ such that $T=\overline{T}J_{p,1}$ and $\|\overline{T}\|\leq p^{-\frac{1}{p^*}}\pi_{p,1}(T)$. Therefore, using now $(4)\Rightarrow (1)$ with exponent $q$ we get that
    \[\pi_{p,q}(T)\leq \intoo[2]{ p+\intoo[2]{\frac{1}{q}-\frac{1}{p}}^{-\frac{1}{q^*}}}\|\overline{T}\|\leq  \intoo[2]{ p^{\frac{1}{p}}+\intoo[2]{\frac{q}{p-q}}^{\frac{1}{q^*}}}\pi_{p,1}(T). \]

    $(2)$ Let $T\in \Pi_{p,1}(C(K),E)$, and apply \Cref{thm: factorization (pq)-summing through Lp1} to find a probability measure $\mu\in C(K)^*$ and an operator $\overline{T}:L_{p,1}(K,\mu)\rightarrow E$ such that $T=\overline{T}J_{p,1}$ and $\|\overline{T}\|\leq p^{-\frac{1}{p^*}}\pi_{p,1}(T)$. Next, we observe that $L_q(K,\mu)\subseteq L_{p,1}(K,\mu)$ and the identity operator $j:L_q(K,\mu)\rightarrow L_{p,1}(K,\mu)$ is bounded. Indeed, if $f\in L_q(K,\mu)$, then
    \begin{align*}
        \vvvert f \vvvert_{L_{p,1}}& =\int_0^\infty t^{-\frac{1}{p^*}}f^*(t) \, dt = \int_0^{\mu(K)} t^{-\frac{1}{p^*}}f^*(t) \, dt\\
        & \leq \intoo[3]{\int_0^1 t^{-\frac{q^*}{p^*}} \, dt}^\frac{1}{q^*} \intoo[3]{\int_0^1 f^*(t)^q \, dt}^\frac{1}{q} = \intoo[2]{ q^*\intoo[2]{\frac{1}{p}-\frac{1}{q}}}^{-\frac{1}{q^*}} \|f\|_q.
    \end{align*}
    In particular, we can write $J_{p,1}=j J_q$. By \Cref{lem: summing embedding of C(K)}, $J_q$ is $q$-summing with $\pi_q(J_q)\leq 1$, so it follows that
    \[\pi_q(T)\leq \|\overline{T}\|\|j\|\pi_q(J_q) \leq p^{-\frac{1}{p^*}} \intoo[2]{ q^*\intoo[2]{\frac{1}{p}-\frac{1}{q}}}^{-\frac{1}{q^*}} \pi_{p,1}(T).\qedhere \]
\end{proof}

Putting together \Cref{thm: characterization (pq)-concave and (pq)-summing} and \Cref{cor: inclusions (p1)-summing operators}, we can conclude that the only relevant cases of $(p,q)$-concave operators are actually $p$-concave and $(p,1)$-concave operators, up to constants depending only on $p$ and $q$.

\begin{thm}\label{thm: relevant cases of concavity}
    Let $X$ be a Banach lattice, $E$ be a Banach space, $T:X\rightarrow E$ and $1<p<\infty$.
    \begin{enumerate}
        \item If $1< q<p$, then $T$ is $(p,q)$-concave if and only if $T$ is $(p,1)$-concave. Moreover
        \[K_{(p,1)}(T) \leq K_{(p,q)}(T) \leq  \intoo[2]{ p^{\frac{1}{p}}+\intoo[2]{\frac{q}{p-q}}^{\frac{1}{q^*}}} K_{(p,1)}(T).\]
        \item If $T$ is $(p,1)$-concave, then it is also $q$-concave for every $p<q<\infty$, with constant
        \[K_{(q)}(T) \leq p^{-\frac{1}{p^*}} \intoo[2]{ q^*\intoo[2]{\frac{1}{p}-\frac{1}{q}}}^{-\frac{1}{q^*}} K_{(p,1)}(T).\]
    \end{enumerate}
\end{thm}

\begin{proof}
    $(1)$ Let us fix $1\leq q<p$. Then, by \Cref{thm: characterization (pq)-concave and (pq)-summing}, $T$ is $(p,q)$-concave if and only if for every compact Hausdorff space $K$ and every positive operator $S:C(K)\rightarrow X$, $TS:C(K)\rightarrow E$ is $(p,q)$-summing. By \Cref{cor: inclusions (p1)-summing operators}, the operators $TS$ are $(p,q)$-summing if and only if they are $(p,1)$-summing. Using again \Cref{thm: characterization (pq)-concave and (pq)-summing}, we conclude that the latter happens if and only if $T$ is $(p,1)$-concave. Moreover, if $T$ is $(p,q)$-concave with constant $K_{(p,q)}(T)$, then $\pi_{p,1}(TS)\leq \pi_{p,q}(TS)\leq \|S\| K_{(p,q)}(T)$, so $K_{(p,1)}(T) \leq K_{(p,q)}(T)$. On the other hand, if $T$ is $(p,1)$-concave with constant $K_{(p,1)}(T)$,  then repeating the argument we get the reverse inequality.\\

    $(2)$ Let us fix $p< q<\infty$. By \Cref{thm: characterization (pq)-concave and (pq)-summing}, if $T$ is $(p,1)$-concave with constant $K_{(p,1)}(T)$, then for every compact Hausdorff space $K$ and every positive operator $S:C(K)\rightarrow X$, $TS:C(K)\rightarrow E$ is $(p,1)$-summing, and hence $q$-summing by \Cref{cor: inclusions (p1)-summing operators}, with 
    \[\pi_q(T)\leq p^{-\frac{1}{p^*}} \intoo[2]{ q^*\intoo[2]{\frac{1}{p}-\frac{1}{q}}}^{-\frac{1}{q^*}} \pi_{p,1}(TS)\leq p^{-\frac{1}{p^*}} \intoo[2]{ q^*\intoo[2]{\frac{1}{p}-\frac{1}{q}}}^{-\frac{1}{q^*}}\|S\| K_{(p,1)}(T).\]
    By \Cref{thm: characterization (pq)-concave and (pq)-summing} again, this implies that $T$ is $q$-concave with constant 
    \[K_{(q)}(T) \leq p^{-\frac{1}{p^*}} \intoo[2]{ q^*\intoo[2]{\frac{1}{p}-\frac{1}{q}}}^{-\frac{1}{q^*}} K_{(p,1)}(T). \qedhere\]
\end{proof}

A straightforward application of \Cref{thm: duality convexity concavity} yields the corresponding statement for $(p,q)$-convex operators.

\begin{cor}\label{cor: relevant cases of convexity}
    Let $X$ be a Banach lattice, $E$ be a Banach space, $T:E\rightarrow X$ and $1<p<\infty$.
    \begin{enumerate}
        \item If $p<q< \infty$, then $T$ is $(p,q)$-convex if and only if $T$ is $(p,\infty)$-convex. Moreover
        \[K^{(p,\infty)}(T) \leq K^{(p,q)}(T) \leq  \intoo[2]{ (p^*)^{\frac{1}{p^*}}+\intoo[2]{\frac{q^*}{p^*-q^*}}^{\frac{1}{q}}} K^{(p,\infty)}(T).\]
        \item If $T$ is $(p,\infty)$-convex, then it is also $q$-convex for every $1<q<p$, with constant
        \[K^{(q)}(T) \leq (p^*)^{-\frac{1}{p}} \intoo[2]{ q\intoo[2]{\frac{1}{p^*}-\frac{1}{q^*}}}^{-\frac{1}{q}} K^{(p,\infty)}(T).\]
    \end{enumerate}
\end{cor}

With these two results, we have completed the diagrams of implications between convexity and concavity properties that we advanced at the end of \Cref{sec: general facts}. It should be noted that $q$-convexity for every $q<p$ does not imply upper $p$-estimates, not even when the $q$-convexity constant behaves asymptotically as the constant in \Cref{cor: relevant cases of convexity}(2), as the following example shows: 

\begin{example}\label{ex: q convex but not upe}
    Let $1<p<\infty$. Let $X=\ell_{p,\infty}(\ell_p)$ be the space of sequences of vectors $\overline{x}=(x_k)_{k\in\N}$ such that $x_k\in \ell_p$ for every ${k\in\N}$ and $\|\overline{x}\|_X:=\|(\|x_k\|_{\ell_p})_k\|_{\ell_{p,\infty}}$ is finite. Then, $X$ does not satisfy an upper $p$-estimate, but it is $q$-convex with constant $(\frac{p}{p-q})^{\frac{1}{q}}$ for every $1\leq q<p$.
\end{example}

\begin{proof}
    This example is based on the same ideas that we used in \Cref{prop: uniform copies of fd weak Lp}. We start by showing that $X$ is $q$-convex for any $1\leq q<p$. Recall from \Cref{sec: weak Lp} that both $\ell_p$ and $\ell_{p,\infty}$ are $q$-convex with constants 1 and $\left(\frac{p}{p-q}\right)^{\frac{1}{q}}$, respectively. Given $\overline{x}^{(i)}=(x_k^{(i)})_{k\in\N}\in X$, $i=1,\ldots,n$, we have
    \begin{align*}
        \norm[3]{\intoo[3]{\sum_{i=1}^n |\overline{x}^{(i)}|^q}^\frac{1}{q}}_X & = \norm[4]{\intoo[4]{\intoo[3]{\sum_{i=1}^n |x_k^{(i)}|^q}^\frac{1}{q}}_k}_X = \norm[4]{\intoo[4]{\norm[3]{\intoo[3]{\sum_{i=1}^n |x_k^{(i)}|^q}^\frac{1}{q}}_{\ell_p}}_k}_{\ell_{p,\infty}}\\
        & \leq \norm[4]{\intoo[4]{\intoo[3]{\sum_{i=1}^n \|x_k^{(i)}\|_{\ell_p}^q}^\frac{1}{q}}_k}_{\ell_{p,\infty}} =  \norm[4]{\intoo[3]{\sum_{i=1}^n \abs[2]{\intoo[2]{\|x_k^{(i)}\|_{\ell_p}}_k}^q}^\frac{1}{q}}_{\ell_{p,\infty}} \\
        & \leq \left(\frac{p}{p-q}\right)^{\frac{1}{q}}  \intoo[3]{\sum_{i=1}^n \norm[2]{\intoo[2]{\|x_k^{(i)}\|_{\ell_p}}_k}_{\ell_{p,\infty}} ^q}^\frac{1}{q} = \left(\frac{p}{p-q}\right)^{\frac{1}{q}}  \intoo[3]{\sum_{i=1}^n \|\overline{x}^{(i)}\|_X^q}^\frac{1}{q},
    \end{align*}
    where the first and second inequalities follow from the $q$-convexity of $\ell_p$ and $\ell_{p,\infty}$, respectively.\medskip

    To show that $X$ does not satisfy an upper $p$-estimate, let us assume that the coordinates in $\ell_p$ and $\ell_{p,\infty}$ are indexed starting at $0$. Let us fix $n\geq 1$ and denote by $(e_j)_{j=0}^\infty$ the canonical basis of $\ell_p$. For $k=0,\ldots,n-1$, we write $a^{(k)}=\sum_{j=0}^{n-1} \alpha_{(k+j)_n} e_j$ for the cyclic permutations of the coefficients $\alpha_k=(k+1)^\frac{1}{p^*}-k^\frac{1}{p^*}$ over the set $\{0,\ldots,n-1\}$. Here, by $i_n$ we mean $i \mod n$. We consider for $i=0,\ldots,n-1$ the (finite) sequences $\overline{x}^{(i)}=(x_k^{(i)})_{k=0}^\infty\in X$ given by $x_k^{(i)}=\alpha_{(k+i)_n} e_i$ for $k=0,\ldots,n-1$ and $x_k^{(i)}=0$ for $k\geq n$. It is clear that for every $i$, 
    \[\|\overline{x}^{(i)}\|_X=\|(\alpha_{(k+i)_n})_{k=0}^{n-1}\|_{\ell_{p,\infty}}=\|(\alpha_{k})_{k=0}^{n-1}\|_{\ell_{p,\infty}}=1,\quad \text{so  }\intoo[3]{\sum_{i=0}^{n-1} \|\overline{x}^{(i)}\|_X^p}^\frac{1}{p}=n^\frac{1}{p}.\]
    On the other hand, for every $k=0,\ldots, n-1$,
    \[\bigvee_{i=0}^{n-1}|x_k^{(i)}|=\sum_{i=0}^{n-1} \alpha_{(k+i)_n} e_i=a^{(k)},\]
    so
    \[\norm[3]{\bigvee_{i=0}^{n-1} |\overline{x}^{(i)}|}_X = \norm[4]{\intoo[4]{\norm[3]{\bigvee_{i=0}^{n-1} |x_k^{(i)}|}_{\ell_p}}_{k=0}^{n-1}}_{\ell_{p,\infty}}= \norm[2]{\intoo[2]{\|a^{(k)}\|_{\ell_p}}_{k=0}^{n-1}}_{\ell_{p,\infty}}= A_n n^\frac{1}{p},\]
    where $A_n=\intoo{\sum_{j=0}^{n-1} \alpha_j^p}^\frac{1}{p}$ is the norm of $a^{(k)}$ in ${\ell_p}$ for every $k$, and $n^\frac{1}{p}$ is the norm in $\ell_{p,\infty}$ of the constant 1 vector on the set $\{0,\ldots,n-1\}$. As we checked in \Cref{prop: uniform copies of fd weak Lp}, $\alpha_k \approx (k+1)^{-\frac{1}{p}}$ uniformly in $k$, so $A_n\approx \intoo{\sum_{k=1}^n k^{-1}}^\frac{1}{p}$, which diverges. To conclude the argument, observe that the quotients $\norm{\bigvee_{i=0}^{n-1} |\overline{x}^{(i)}|}_X /\intoo{\sum_{i=0}^{n-1} \|\overline{x}^{(i)}\|_X^p}^\frac{1}{p}=A_n$ are not uniformly bounded, so $X$ cannot be $({p,\infty})$-convex.
\end{proof}

\section{Representations}\label{sec: representations}
The previous factorization results can be used to prove representation theorems for convex and concave Banach lattices. Recall that general Banach lattices are not far from being spaces of functions: a classical result from Lotz \cite[Lemma 3.4]{Lotz} establishes that every Banach lattice embeds isometrically as a sublattice of an $\ell_\infty$ sum of $L_1$-spaces. Using \Cref{cor: factorization Maurey operator} and \Cref{cor: factorization Pisier operator}, we can improve this representation result for $p$-convex Banach lattices and Banach lattices with upper $p$-estimates \cite[Proposition 3.3]{GLTT}.

\begin{thm}\label{thm: representation infty sums}
    Let $X$ be a Banach lattice and $1<p<\infty$.
    \begin{enumerate}
        \item If $X$ is $p$-convex, then there are a family $\Gamma$ of probability measures and a lattice isomorphism
        \[ J : X\rightarrow \ell_\infty( L_p(\mu))_{\mu\in\Gamma} \]
        such that $\|x\| \leq \|Jx\| \leq K^{(p)}(X)\|x\|$ for all $x\in X$.
        \item If $X$ satisfies an upper $p$-estimate, then there are a family $\Gamma$ of probability measures and a lattice isomorphism
        \[ J : X\rightarrow \ell_\infty( L_{p,\infty}(\mu))_{\mu\in\Gamma} \]
        such that $\|x\| \leq \|Jx\| \leq \gamma_p K^{(\uparrow p)}(X)\|x\|$ for all $x\in X$, where $\gamma_p=(1-\frac{1}{p})^{\frac{1}{p}-1}=(p^*)^{\frac{1}{p^*}}$.
    \end{enumerate}
\end{thm}

\begin{proof}
    Let $x\in S=S_X\cap X_+$. Choose $x^*\in X^*_+$ such that $\|x^*\|=1$ and $x^*(x) = 1$ and define $\rho_x:X\rightarrow \R$ by $\rho_x(z) = x^*(|z|)$. Then $\rho_x$ induces an $AL$-norm $\hat{\rho}_x$ on $X/\ker \rho_x$. Let $q_x: X\rightarrow X/\ker\rho_x$ be the quotient map. By Kakutani's Representation Theorem for $AL$-spaces, we know that the completion of $(X/\ker \rho_x,\hat{\rho}_x)$ is lattice isometric to $L_1(\mu_x)$ for some measure $\mu_x$, so let us denote by $i_x: X/\ker\rho_x\rightarrow L_1(\mu_x)$ the formal inclusion, that satisfies $\|i_x q_x x\|_{L_1(\mu_x)}=\hat{\rho}_x(q_xx)=\rho_x(x)=1$. Since $X$ is $p$-convex or satisfies an upper $p$-estimate (i.e., it is $(p,\infty)$-convex), depending on the case, and $i_x q_x$ is positive (actually, a lattice homomorphism) and contractive, it follows from \Cref{rem: facts about convexity}(7) that $i_x q_x$ is $p$-convex or $(p,\infty)$-convex, respectively, with the same constant as $X$. Therefore, \Cref{cor: factorization Maurey operator} and \Cref{cor: factorization Pisier operator} yield that there exists $g_x\in L_1(\mu_x)_+$ with $\|g_x\|_{L_1(\mu_x)}= 1$, such that $i_x q_x$ factors as $M_x R_x$, with $R_x z=g_x^{-1}i_x q_x z$ for any $z\in X$, $M_x$ is the multiplication by $g_x$ operator, and 
    \[ \|R_x z\|_{L_p(g_x\cdot \mu_x)} \leq K^{(p)}(X)\|x\|, \quad \text{respectively,}\quad \|R_x z\|_{L_{p,\infty} (g_x\cdot \mu_x)} \leq \gamma_p K^{(\uparrow p)}(X)\|x\|.\]
    Define 
    \[ J: X\rightarrow \ell_\infty( L_p(g_x\cdot \mu_x))_{x\in S} , \quad \text{respectively}\quad J : X\rightarrow  \ell_\infty( L_{p,\infty}(g_x\cdot \mu_x))_{x\in S}\]
    by $Jz = (R_xz)_{x\in S}$ for every $z\in X$. It is clear that $R_x$, and thus, $J$, are lattice homomorphisms. Moreover, $\|J\|=\sup_{x\in S}\|R_x\|$ is bounded by $K^{(p)}(X)$ or $\gamma_p K^{(\uparrow p)}(X)$, respectively. On the other hand, since $M_x$ is a contraction in both cases, for every $x\in S$ we have that, in the $p$-convex setting,
    \[ \|Jx\| \geq  \|R_x x\|_{L_p(g_x\cdot\mu_x)} \geq \|i_x q_xx\|_{L_1(\mu_x)}=1,\]
    and in the $(p, \infty)$-convex setting
    \[ \|Jx\| \geq  \|R_x x\|_{L_{p,\infty}(g_x\cdot\mu_x)} \geq \|i_x q_xx\|_{L_1(\mu_x)}=1.\]
    It follows that  $\|Jx\| \geq \|x\|$ for all $x\in X$ in both cases.
\end{proof}

\begin{rem}\label{rem: optimal gamma p}
    It is worth noting that (1) provides an isometric representation for $p$-convex Banach lattices with constant one, whereas (2) only provides a $\gamma_p$-isomorphic representation for Banach lattices satisfying upper $p$-estimates with constant one. This seems to be a fundamental difference between $p$-convexity and upper $p$-estimates. In fact, this constant $\gamma_p=(p^*)^{\frac{1}{p^*}}$ is optimal. In \cite{MontgomerySmith}, Montgomery-Smith showed that the constant $p^\frac{1}{p}$ appearing in the proof of $(1)\Rightarrow (2)$ in \Cref{thm: factorization (pq)-summing through Lp1} is optimal. The same argument can be used to show that $\gamma_p$ is optimal in \Cref{thm: representation infty sums}, and hence, in \Cref{thm: factorization Pisier set} for $r=1$ (see \cite{GLTT2}). In other words, there exist Banach lattices satisfying upper $p$-estimates with constant one that do not embed isometrically into any $\ell_\infty$ sum of $L_{p,\infty}$-spaces.
\end{rem}

\begin{rem}
    Observe that \Cref{thm: representation infty sums} together with \Cref{prop: wlp upe and r convex} provides an alternative proof of \Cref{cor: relevant cases of convexity}(2) which is based on \Cref{thm: factorization Pisier set} instead of \Cref{thm: factorization (pq)-summing through Lp1}. There is a third proof of \Cref{cor: relevant cases of convexity}(2) available at \cite[Theorem 1.f.7]{LT2} that uses a probabilistic argument. The author is not aware of any alternative proof of \Cref{cor: relevant cases of convexity}(1). If there were a simple proof of the fact that $L_{p,\infty}(\mu)$ is $(p,q)$-convex for any $q>p$, then \Cref{thm: representation infty sums} would automatically yield the equivalence between $(p,\infty)$-convexity and $(p,q)$-convexity.
\end{rem}

We can formulate a dual version of the above representation theorem for $q$-concave Banach lattices and Banach lattices with lower $q$-estimates:

\begin{thm}\label{thm: representation 1 sums}
    Let $X$ be a Banach lattice and $1<q<\infty$.
    \begin{enumerate}
        \item If $X$ is $q$-concave, then there are a family $\Gamma$ of probability measures and an onto interval preserving positive operator
        \[ Q : Y= \ell_1(L_q(\mu))_{\mu\in\Gamma}\rightarrow X\]
        such that $B_X\subseteq Q(B_Y)\subseteq K_{(q)}(X) B_X$.
        \item If $X$ satisfies an lower $q$-estimate, then there are a family $\Gamma$ of probability measures and an onto interval preserving positive operator
        \[ Q : Y= \ell_1(L_{q,1}(\mu))_{\mu\in\Gamma} \rightarrow X\]
        such that $B_X\subseteq Q(B_Y)\subseteq \gamma_{q^*} K_{(\downarrow q)}(X) B_X$ (here, $L_{q,1}$ is endowed with the dual norm induced by $\|\cdot\|_{L^{[1]}_{q^*,\infty}}$).
    \end{enumerate}
\end{thm}

For the proof, we will need the following general facts:

\begin{lem}\label{lem: interval preserving constructions}
    \begin{enumerate}
        \item[]
        \item Let $T_\gamma:Y_\gamma\rightarrow X$ be almost interval preserving operators for every $\gamma\in \Gamma$ such that $\sup_{\gamma\in \Gamma}\|T_\gamma\|<\infty$, and let $Y=\ell_1( Y_\gamma)_{\gamma\in\Gamma}$. Then,
        \[\fullfunction{T}{Y}{X}{\overline{y}=(y_\gamma)_{\gamma\in\Gamma}}{T\overline{y}=\sum_{\gamma\in\Gamma}T_\gamma y_\gamma}\]
        is an almost interval preserving operator with $\|T\|\leq \sup_{\gamma\in \Gamma}\|T_\gamma\|$.
        \item Let $Y$ be order continuous, $X$ a KB-space and $T:Y\rightarrow X$ an almost interval preserving operator. Then $T$ is interval preserving.
    \end{enumerate}
\end{lem}

\begin{proof}
    $(1)$ It is easy to check that $Y^*=\ell_\infty( Y^*_\gamma)_{\gamma\in\Gamma}$ and that $T^*x^*=(T^*_\gamma x^*)_{\gamma\in\Gamma}$, which is clearly a lattice homomorphism. Therefore, $T$ must be almost interval preserving.\\

    $(2)$ $Y$ is order continuous, so $J_Y:Y\rightarrow Y^{**}$ is interval preserving \cite[Theorem 2.4.2]{MN}. Moreover, $X$ is a KB-space, so $J_X(X)$ is a projection band in $X^{**}$ \cite[Theorem 2.4.12 and Theorem 1.2.9]{MN}, so there exists an interval preserving operator $P:X^{**}\rightarrow X$ such that $P J_X$ is the identity on $X$. Since $T$ is almost interval preserving, $T^{**}:Y^{**}\rightarrow X^{**}$ is interval preserving. But observe that
    \[T=P J_X T=P T^{**} J_Y,\]
    so $T$ must be interval preserving, as it is the composition of interval preserving operators.
\end{proof}

\begin{proof}[Proof of \Cref{thm: representation 1 sums}]
    We start as in the proof of \Cref{thm: representation infty sums} by fixing $x\in S=S_X\cap X_+$ and considering the quotient $q_x:X^*\rightarrow X^*/\ker \rho_x$, where $\rho_x(x^*)=|x^*|(x)$, and $i_x:X^*/\ker \rho_x\rightarrow L_1(\mu_x)$ the inclusion into the completion of $(X^*/\ker \rho_x,\rho_x)$, which is an $AL$-space. It is clear that $i_x q_x$ is a $q^*$-convex operator with constant $K_{(q)}(X)$ (respectively, $(q^*,\infty)$-convex operator with constant $K_{(\downarrow q)}(X)$). By \Cref{cor: factorization Maurey operator} (respectively, \Cref{cor: factorization Pisier operator}), there exists a normalized $g_x\in L_1(g_x)_+$ such that $i_x=M_x R_x$, where $R_x x^*=g_x^{-1}i_x q_x x^*$ is defined from $X^*$ into $L_{q^*}(g_x\cdot \mu_x)$ with norm $K_{(q)}(X)$ (respectively, $L_{q^*,\infty}(g_x\cdot \mu_x)$ with norm $\gamma_{q^*}K_{(\downarrow q)}(X)$) and $M_x$ is the operator of multiplication by $g_x$. \\
    
    Recall that in both cases $X$ is a KB-space (\Cref{lem: lpe implies KB space}), so by \cite[Theorem 2.4.12 and Theorem 1.2.9]{MN} there exists an interval preserving operator $P:X^{**}\rightarrow X$ such that $P J_X$ is the identity on $X$. Let us define $Y:=\ell_1(L_q(\mu))_{\mu\in\Gamma}$ and $Q_x=P R_x^*:L_q(g_x\cdot \mu_x)\rightarrow X$ when $X$ is $q$-concave, and $Y:=\ell_1(L_{q,1}(\mu))_{\mu\in\Gamma}$ and $Q_x=P R_x^* J_{L_{q,1}}:L_{q,1}(g_x\cdot \mu_x)\rightarrow X$ when $X$ satisfies a lower $q$-estimate, and denote by $j_x$ and $P_x$, $x\in S$, the embeddings and projections associated to the factors that form $Y$, which are clearly interval preserving. In both cases, the operators $Q_x$ are interval preserving (for the second case, use the fact that $L_{q,1}(g_x\cdot \mu_x)$ is order continuous by \Cref{lem: lpe implies KB space}, so it embeds as an ideal into its bidual \cite[Theorem 2.4.2]{MN}, which is $L_{q^*,\infty}(g_x\cdot \mu_x)^*$), so the compositions $Q_x P_x$ are interval preserving as well, and their norms are uniformly bounded. Applying \Cref{lem: interval preserving constructions}(1) we obtain an almost interval preserving operator $Q:Y\rightarrow X$ given by $Q\overline{f}=\sum_{x\in S} Q_xf_x$ for every $\overline{f}=(f_x)_{x\in S}\in Y$, with norm bounded by either $K_{(q)}(X)$ or $\gamma_{q^*}K_{(\downarrow q)}(X)$. Since $Y$ is $q$-concave or satisfies an upper $q$-estimate, it is order continuous, and applying \Cref{lem: interval preserving constructions}(2) we get that $Q$ is actually interval preserving.\\

    To check the surjectivity of $Q$, observe that $Q( j_x\uno_x)=x$ for every $x\in S$, where $\uno_x$ denotes the constant one function in the corresponding $L_{q}(g_x\cdot \mu_x)$ or $L_{q,1}(g_x\cdot \mu_x)$ (we only check it in the lower $q$-estimate setting, as the $q$-concave setting is slightly easier): for every $x^*\in X^*_+$, 
    \begin{align*}
        (R_x^* J_{L_{q,1}}\uno_x)(x^*)& =(J_{L_{q,1}}\uno_x)(R_x x^*)=(R_x x^*)(\uno_x) =\int \frac{i_x q_x x^*}{g_x}\uno_x g_x\, d\mu_x\\
        &=\|i_x q_x x^*\|_{L_1(\mu_x)}=\rho_x(x^*)=x^*(x)=J_Xx(x^*),
    \end{align*}
    so $R_x^* J_{L_{q,1}}\uno_x=J_X x$ and hence $Q( j_x\uno_x)=PJ_X x=x$. This fact alone already implies that $Q$ is onto, but actually we can show that $B_X\subseteq Q(B_Y)$. Indeed, let $0\neq z\in B_X$. If $z\geq 0$, then $z=Q(\|z\|j_x\uno_x)\in Q(B_Y)$ for $x=z/\|z\|\in S$. Otherwise, let $x=|z|/\|z\|\in S$ and use the fact that $Q_x$ is interval preserving to find $f_\pm\in [0,\uno_x]$ such that $Q_xf_\pm=z_\pm\in [0,x]$. Clearly, $f=f_+-f_-\in [-\uno_x,\uno_x]$, so $z=Q(j_xf)\in Q(B_Y)$, and the proof is concluded.
\end{proof}

We proved in \Cref{cor: isomorphic ALp-spaces} that if a Banach lattice $X$ satisfies both an upper and a lower $p$-estimate for some $1\leq p\leq \infty$, then $X$ is lattice isomorphic to some $L_p(\mu)$ (respectively, to some $AM$-space when $p=\infty$). It turns out that if the norm of any sum of two disjoint elements of $X$ can be computed exclusively in terms of the norms of the elements, then we are necessarily in the above situation, as the following result from Bohnenblust shows (see \cite[Theorem 1.b.7]{LT2} for a proof).

\begin{thm}\label{thm: Bohnenblust}
    Let $X$ be a Banach lattice of dimension at least $3$ for which there exists a function $F$ defined in $\R^2_+$ such that, for all disjoint $x,y\in X$ we have $\|x+y\|=F(\|x\|,\|y\|)$. Then $X$ is either an $AL_p$-space for some $1\leq p<\infty$ or an $AM$-space.
\end{thm}

Something similar happens in the isomorphic setting \cite[Theorem 1.b.12]{LT2}.

\begin{thm}\label{thm: Tzafriri}
    Let $X$ be an order continuous Banach lattice. Then $X$ is lattice isomorphic to either $L_p(\mu)$ for some measure $\mu$ and some $1\leq p<\infty$, or to $c_0(\Gamma)$ for some set $\Gamma$, if and only if there exists a non-negative valued function of infinitely many real variables $F(t_1,t_2,\ldots)$ and a constant $A>0$ so that for every disjoint sequence $(x_n)_{n=1}^\infty\subseteq X$ such that $\sum_{n=1}^\infty x_n$ converges, we have
    \[A^{-1}F(\|x_1\|,\|x_2\|,\ldots)\leq \norm[3]{\sum_{n=1}^\infty x_n}\leq A F(\|x_1\|,\|x_2\|,\ldots).\]
\end{thm}

The above results point to $p$-sums as the natural expressions for describing norms of Banach lattices. One can wonder if something can be said about a Banach lattice if in the previous statement we allow the upper and lower bounds to be given by different expressions, for instance $p$ and $q$-sums with $p\neq q$. Following this direction, it is well known \cite[Theorem 2.7.8]{MN} that if $X$ is an order continuous Banach lattice with a weak unit, then there exists a probability space $(\Omega,\Sigma,\mu)$ and an ideal $(\widetilde{X},\|\cdot\|_{\widetilde{X}})$ of $L_1(\mu)$ such that $X$ is lattice isometric to $\widetilde{X}$ and $L_\infty(\mu)\subseteq \widetilde{X} \subseteq L_1(\mu)$ with both inclusions being continuous and dense. Moreover, $X^*$ is lattice isometric to the ideal $\widetilde{X}^*$ of all the measurable functions $g$ such that
\[\|g\|_{\widetilde{X}^*}=\sup \cbr[3]{\int fg\,d\mu: \|f\|_{\widetilde{X}}\leq 1}<\infty,\]
and $L_\infty(\mu)\subseteq \widetilde{X}^* \subseteq L_1(\mu)$ with continuous inclusions, the first one being order dense and the second one, norm dense. It turns out that under convexity and concavity assumptions, this statement can be refined \cite[p. 14]{JMST}.

\begin{thm}\label{thm: representation thm refined}
    Let $X$ be an order continuous Banach lattice with a weak unit, $1< p<q<\infty$, and consider $\widetilde{X}\subseteq L_1(\mu)$ as above. 
    \begin{enumerate}
        \item If $X$ is $p$-convex with constant one, then $\widetilde{X} \subseteq L_p(\mu)$ with continuous inclusion.
        \item If $X$ is $q$-concave with constant one, then $L_q(\mu)\subseteq \widetilde{X}$ with continuous inclusion.
    \end{enumerate}
\end{thm}

\begin{proof}
    (1) Without loss of generality, we can assume that the inclusion of $\widetilde{X}$ into $L_1(\mu)$ is contractive. Let $f\in X$ be a simple function. Then, for every $t\geq 0$ we have
    \[F(t)=\int (t|f|^p+1)^\frac{1}{p}\,d\mu =  \|(t|f|^p+1)^\frac{1}{p}\|_1 \leq \|(t|f|^p+1)^\frac{1}{p}\|_{\widetilde{X}}\leq (t\|f\|_{\widetilde{X}}^p+1)^\frac{1}{p}=G(t). \]
    Note that $F(0)=1=G(0)$ and both $F$ and $G$ are differentiable, since $f$ is simple (one can also argue using a derivation under the integral sign argument, cf. \cite[Example 2.4.6]{Cohn}), so 
    \[\frac{1}{p}\|f\|_p^p=\frac{1}{p}\int|f|^p=F'(0)\leq G'(0)=\frac{1}{p}\|f\|_{\widetilde{X}}^p,\]
    that is, $\|f\|_p\leq \|f\|_{\widetilde{X}}$ for every simple function. Now, recall that if $f$ is positive measurable function, we can always find an increasing sequence $(f_n)_n$ of simple functions that converges pointwise to $f$. In particular, for every $f\in \widetilde{X}_+$, let $(f_n)_n$ be a sequence of this kind, so that $f_n^p\uparrow f^p$. Using Fatou's Lemma, we conclude that
    \[\int f^p\, d\mu \leq \liminf \int f_n^p \, d\mu \leq \liminf \|f_n\|_{\widetilde{X}}^p \leq \|f\|_{\widetilde{X}}^p,\]
    so $f\in L_p(\mu)$ and the inclusion is contractive.\\

    (2) If $X$ is $q$-concave, we can assume by rescaling the norm of $\widetilde{X}$ if needed that $\widetilde{X}^*$ is contractively contained in $L_1(\mu)$. Now, repeating the argument from (1), we deduce that $ \widetilde{X}^*\subseteq L_{q^*}(\mu)$ and $\|g\|_{q^*}\leq \|g\|_{\widetilde{X}^*}$ for every $g\in \widetilde{X}^*$. Now, given any simple function $f\in \widetilde{X}_+$, for every $\eps>0$ we can find $g\in \widetilde{X}^*$ with $\|g\|_{\widetilde{X}^*}=1$ such that
    \[\|f\|_{\widetilde{X}}-\eps \leq \int fg\, d\mu\leq \|f\|_q \|g\|_{q^*}\leq \|f\|_q\|g\|_{\widetilde{X}^*}\leq \|f\|_q,\]
    so $\|f\|_{\widetilde{X}}\leq \|f\|_q$ for every simple function. Now, let $f\in L_q(\mu)_+$ and $(f_n)_n$ be a sequence of simple functions such that $f_n\uparrow f$. Since $L_q(\mu)$ is order continuous, $(f_n)_n$ converges to $f$, so in particular $(f_n)_n$ is a Cauchy sequence in $L_q(\mu)$. By the inequality above, $\|f_n-f_m\|_{\widetilde{X}}\leq \|f_n-f_m\|_q$, so $(f_n)_n$ is also Cauchy in $\widetilde{X}$, so it converges to some $h\in \widetilde{X}$. Since $(f_n)_n$ is increasing, $h=\sup_n f_n$, but the suprema of sequences in $L_q(\mu)$ and $\widetilde{X}$ coincide, as they are both ideals in $L_1(\mu)$, so $h=f$. We conclude that $\|f\|_{\widetilde{X}}=\lim_n \|f_n\|_{\widetilde{X}}\leq \lim_n \|f_n\|_q= \|f\|_q$.
\end{proof}

\section*{Aknowledgments}
The author wishes to express his gratitude towards Vladimir G. Troitsky for the invitation to teach this course. He also wants to thank all the attendants to the course, especially Vladimir G. Troitsky, Eugene Bilokopytov, Tomasz Szczepanski, and Kevin Abela, for their interest in the topic and their valuable suggestions.\\

The author wants to thank his coauthors Denny H. Leung, Mitchell A. Taylor, and Pedro Tracadete, as this survey would not exist without their fruitful collaboration. He also thanks Emiel Lorist for his questions concerning the renorming results, which led to the improved theorem contained in the second version of this survey.\\ 

The research of the author is partially supported by the grants PID2020-116398GB-I00, PID2024-162214NB-I00, CEX2019-000904-S-21-3 and CEX 2023-001347-S funded by the MICIU/AEI/10.13039/501100011033 and by ``ESF+''.

\end{document}